\documentclass[reqno,11pt]{amsart}
\usepackage[a4paper]{geometry}
\usepackage{tgheros}
\usepackage{amssymb,amsmath,amsthm,enumerate,mathrsfs,bbm}

\DeclareMathOperator{\Ran}{Ran}
\DeclareMathOperator{\Dom}{Dom}
\DeclareMathOperator{\Ker}{Ker}

\DeclareMathOperator{\supp}{supp}

\DeclareMathOperator{\Span}{span}

\newcommand{\BMO}{\operatorname{BMO}}
\newcommand{\BMOA}{\operatorname{BMOA}}
\newcommand{\VMOA}{\operatorname{VMOA}}
\newcommand{\simp}{\text{\rm simp}}

\newcommand{\dist}{\operatorname{dist}}
\renewcommand{\Im}{\operatorname{Im}}
\renewcommand{\Re}{\operatorname{Re}}
\newcommand{\spn}{\operatorname{span}}

\newcommand{\abs}[1]{\lvert#1\rvert}

\newcommand{\norm}[1]{\lVert#1\rVert}

\newcommand{\jap}[1]{\langle#1\rangle}
\newcommand{\fdot}{\,\cdot\,}

\newcommand{\dd}{\mathrm{d}}

\newcommand{\ute}{\ut{e}}

\newcommand{\R}{\mathbb{R}} 

\newcommand{\bS}{{\mathbf S}}
\newcommand{\bC}{{\mathbf C}}
\newcommand{\bJ}{{\mathbf J}}

\newcommand{\bbD}{{\mathbb D}}
\newcommand{\bbT}{{\mathbb T}}
\newcommand{\bbR}{{\mathbb R}}

\newcommand{\bbN}{{\mathbb N}}
\newcommand{\bbZ}{{\mathbb Z}}

\newcommand{\calA}{{\mathcal A}}
\newcommand{\calB}{{\mathcal B}}

\newcommand{\calH}{{\mathcal H}}
\newcommand{\calM}{\mathcal{M}}
\newcommand{\calC}{\mathcal{C}}
\newcommand{\calU}{\mathcal{U}}

\newcommand{\calD}{\mathcal{D}}
\newcommand{\calS}{\mathcal{S}}

\newcommand{\fC}{\mathfrak{C}}

\newcommand{\fS}{\mathfrak{S}}

\newcommand{\f}{\varphi}

\numberwithin{equation}{section}

\theoremstyle{plain}
\newtheorem{theorem}{\bf Theorem}[section]
\newtheorem*{theorem*}{Theorem}
\newtheorem{lemma}[theorem]{\bf Lemma}
\newtheorem{proposition}[theorem]{\bf Proposition}
\newtheorem*{proposition*}{\bf Proposition}
\newtheorem{corollary}[theorem]{\bf Corollary}

\theoremstyle{definition}

\newtheorem*{definition*}{\bf Definition}

\theoremstyle{remark}
\newtheorem*{remark*}{\bf Remark}
\newtheorem{remark}[theorem]{\bf Remark}

\newcommand{\wt}{\widetilde}
\newcommand{\wh}{\widehat}
\newcommand{\eps}{\varepsilon}

\newcommand{\1}{\mathbbm{1}}

\newcommand{\ci}[1]{{\rule[-0.65ex]{0ex}{1.0ex}}_{#1}}

\newcommand{\ti}[1]{_{\scriptstyle \text{\rm #1}}}
\newcommand{\ut}[1]{^{\text{\rm #1}}}

\renewcommand{\labelenumi}{\textup{(\roman{enumi})}}

\newcounter{vremennyj}

\newcommand\cond[1]{\setcounter{vremennyj}{\theenumi}\setcounter{enumi}{#1}\labelenumi\setcounter{enumi}{\thevremennyj}}

\begin{document}

\title[Inverse problem for Hankel operators]{An inverse spectral problem for non-compact Hankel operators with simple spectrum}

\author{Patrick G\'erard}
\address{Universit\'e Paris-Saclay, Laboratoire de Math\'ematiques d'Orsay, CNRS, UMR 8628, France}
\email{patrick.gerard@universite-paris-saclay.fr}

\author{Alexander Pushnitski}
\address{Department of Mathematics, King's College London, Strand, London, WC2R~2LS, U.K.}
\email{alexander.pushnitski@kcl.ac.uk}

\author{Sergei Treil}
\address{Department of Mathematics, Brown University, Providence, RI 02912, USA}
\email{treil@math.brown.edu}
\thanks{Work of S.~Treil was supported  in part by the National Science Foundation under the grants 
DMS-1856719, DMS-2154321}



\begin{abstract}
We consider an inverse spectral problem for a class of non-compact Hankel operators $H$ such that the modulus of $H$ (restricted onto the orthogonal complement to its kernel) has simple spectrum. Similarly to the case of compact operators, we prove a uniqueness result, i.e. we prove that a Hankel operator from our class is uniquely determined by the spectral data. In other words, the spectral map, which maps a Hankel operator to the spectral data, is injective. Further, in contrast to the compact case, we prove the failure of surjectivity of the spectral map, i.e. we prove that not all spectral data from a certain natural set correspond to Hankel operators. We make some progress in describing the image of the spectral map. We also give applications to the cubic Szeg\H{o} equation. In particular, we prove that not all solutions with initial data in $\BMOA$ are almost periodic; this is in a sharp contrast to the known result for initial data in $\VMOA$. 
\end{abstract}

\maketitle

\setcounter{tocdepth}{1}
\tableofcontents
\setcounter{tocdepth}{3}

\section{Introduction}\label{sec.a}

\subsection{Overview}
In the mid-1980s, Khrushchev and Peller \cite{KP},
motivated by the spectral theory of stationary Gaussian processes, 
asked to  describe all 
non-negative self-adjoint operators that are unitarily equivalent to the \emph{modulus} of a 
Hankel operator $\Gamma$ (i.e.~to the operator $|\Gamma|:=(\Gamma^*\Gamma)^{1/2}$). 

This problem was actively studied from the mid 1980s to early 1990s,  
see \cite{Tr-ISP85-Dokl, Tr-ISP85, VasTr-ISP89, Ober87}, 
until the final result was obtained by Treil \cite{Tr-ISP90}: any 
positive semi-definite self-adjoint operator that is non-invertible and whose kernel is either trivial or 
infinite-dimensional is unitarily equivalent to the modulus of a Hankel operator.  This gives a 
complete solution to the problem, since it is easy to see that any Hankel operator is not invertible and cannot 
have a finite-dimensional kernel. 

Later, motivated by problems in control theory,  Megretskii, Peller and Treil started investigation 
of the analogous problem for self-adjoint Hankel operators. The question was to describe all
possible types of spectral measures  and the multiplicity functions, corresponding to 
self-adjoint Hankel operators.       

A complete solution to this problem was given in \cite{MPT} (also see \cite{MPT} for the history of 
the problem). The answer was slightly more complicated than for the modulus 
of a Hankel operator: besides the obvious properties of non-invertibility and the absence of a
finite-dimensional kernel, some ``almost symmetry'' property of the spectral multiplicity function was also 
required. 

In both of the these problems, the spectral datum\footnote{we use the convention \emph{singular:} 
datum, \emph{plural:} data.} (i.e.~the type of the spectral measure and the multiplicity function) 
does not determine the corresponding Hankel operator uniquely: in fact, with the exception of 
trivial cases, there are infinitely many self-adjoint Hankel operators with the same spectral 
datum. 

In the early 2010s, the interest to inverse spectral problems for Hankel operators was renewed due 
to the work of G\'erard and Grellier \cite{GG00,GG0} on the \emph{cubic Szeg\H{o} equation}. This 
is a totally non-dispersive evolution equation which is completely integrable and possesses a Lax 
pair, which involves a Hankel operator (see Section~\ref{sec.a11} for the details). Motivated by 
this, in \cite{GG1} G\'erard and Grellier developed a new type of direct and inverse spectral 
theory for \emph{compact} Hankel operators. The Hankel operators appearing in this theory are generally not 
self-adjoint, and the language of anti-linear operators gives a convenient way to represent the 
spectral datum in this case. 

Another new feature of this theory is that the spectral datum was constructed from the pair of 
Hankel operators $\Gamma$ and $\Gamma S$, where $S$ is the \emph{shift operator} in the Hardy space 
$H^2$. 
In this case there is a bijection between compact Hankel operators and the corresponding spectral 
data, and the class of spectral data sets corresponding to the compact operators can be explicitly 
described. In this construction, the evolution of the spectral datum under the cubic Szeg\H{o} equation is very simple, which makes the bijectivity very desirable.

The next natural step in this line of research is the study of the direct and inverse spectral 
problem for \emph{non-compact} Hankel operators. For a few years, the work of two of the authors 
(G\'erard and Pushnitski) was motivated by the conjecture that the bijective spectral map of 
\cite{GG1} admits a natural extension to the non-compact case; some preliminary steps  in this 
direction  were made in \cite{GP}. One of the aims of the present paper is to show that \emph{this 
conjecture is false}, in some precise sense to be explained below. For a suitable class of Hankel 
operators (which includes many non-compact ones), we construct a natural extension of the spectral 
map of \cite{GG1} and show that it is injective, but \emph{not surjective}. We also give an 
application to the cubic Szeg\H{o} equation, corresponding to non-compact Hankel operators. We show that in general, solutions to this equation with the initial data in $\BMOA$ are NOT 
almost-periodic, in contrast with the case of the initial data in $\VMOA$.

An important new component of the present work is the functional model for contractions 
($=$operators of norm $\leq1$) on a Hilbert space.  A key ingredient to proving that a given spectral 
datum corresponds to some Hankel operator is checking that a certain contraction, constructed from the spectral datum, is asymptotically stable. (A contraction $T$ is called \emph{asymptotically 
stable} if $T^n\to0$ in the strong operator topology as $n\to\infty$.) In the compact case, it 
turns out that the asymptotic stability always holds. In the non-compact case, we show that the 
asymptotic stability sometimes holds and sometimes doesn't, depending on some spectral properties 
of the contraction. Here we use some latest advances \cite{LT} from the theory of the Clark model. 
A more precise discussion is postponed to Section~\ref{sec.bb}.

\subsection{The structure of the paper}
In this section we introduce Hankel operators, describe the direct spectral problem and the spectral data, and present our first main result: uniqueness. Proofs are postponed to Section~\ref{sec.b}. In Section~\ref{sec.bb}, we discuss the problem of surjectivity of the spectral map and informally describe our main results concerning the failure of surjectivity. In Section~\ref{sec.e} we collect without proof some operator theoretic background (which mainly concerns the spectral theory of contractions on Hilbert spaces and the Clark model) that is required for the construction of the rest of the paper.  Sections~\ref{sec.dd1}--\ref{sec.8} are the core of the paper; here we state and prove our main results concerning the failure of surjectivity and the description of the image of the spectral map. In Section~\ref{sec.d} we describe the special case of self-adjoint Hankel operators. In Section~\ref{sec.sz} we give an application to the cubic Szeg\H{o} equation. Some technical parts of proofs are postponed to Appendices.

\subsection{Notation}
For a Hilbert space $X$, we denote the inner product  of elements $f,g\in X$ by $\jap{f,g}_X$; we 
omit the subscript $X$ if there is no danger of confusion.
For a bounded self-adjoint operator $A$ in $X$ and for $v\in X$, we denote by
\[
\langle v\rangle _A:={\rm clos\ span}\{ A^nv, n=0,1,2,\dots\} 
\]
the cyclic subpace of $A$ generated by $v$. We recall that $A$ is said to have simple spectrum if 
$X=\langle v\rangle_A$ for some element $v\in X$; any such element is called \emph{cyclic} for $A$. 
We denote by $\rho_v^A$ the spectral measure of $A$ corresponding to $v$, i.e. 
\begin{equation}
\jap{f(A)v,v}=\int_\bbR f(s)d\rho_v^A(s)
\label{b1}
\end{equation}
for any continuous function $f$. 

We denote by $\bS_p$, $p>0$, the standard Schatten class of compact operators; in particular, 
$\bS_1$ is trace class and $\bS_2$ is the Hilbert-Schmidt class. 

For a finite measure $\rho$ on $\bbR$, we denote 
$L^2(\rho)\equiv L^2(\bbR,d\rho)$, and we usually use the letter $s$ to denote the independent 
variable in $\bbR$. We denote by  $\1\in L^2(\rho)$ the function identically equal to one.

We denote by $H^2=H^2(\bbT)$ is the standard Hardy space of functions on the unit circle $\bbT$, 
\[
f=f(z)=\sum_{j=0}^\infty \wh f_jz^j, \quad \abs{z}=1, \quad \sum_{j=0}^\infty \abs{\wh 
f_j}^2<\infty; 
\]
the above series converges in $L^2(\bbT)$. Note that this series also converges uniformly on 
compact subsets of $\bbD$, so $f$ can be interpreted as an analytic function in the unit disc 
$\bbD$. The values of $f$ on $\bbT$ can be found as the \emph{non-tangential boundary values} of 
this analytic function; according to classical results these non-tangential limits exist a.e.~on $\bbT$. 
Note also that the set $H^\infty=H^\infty(\bbD)$ of all bounded analytic function is a subset of 
$H^2$.  

We denote by $\{z^m\}_{m=0}^\infty$ the standard basis in the Hardy space $H^2$; in 
particular, we denote by $z^0$ the element of $H^2$ identically equal to one (as notation 
$\1$ is already taken).   
The Szeg\H{o} projection $P$ is the orthogonal projection onto $H^2$ in $L^2(\bbT)$, 
\begin{align*}
P: \sum_{k=-\infty}^\infty \wh f_k z^k\mapsto \sum_{k=0}^\infty \wh f_k z^k .
\end{align*}
Recall that the shift operator $S$ on $H^2$ is the multiplication by $z$, $Sf =zf(z)$, $f\in H^2$, 
and its adjoint (the backward shift) $S^*$ is given by 
\[
S^* f(z) =\frac{f(z) - f(0)}{z}. 
\]
We refer e.g. to \cite[Appendix 2]{Peller} for the definition of the classes $\BMOA(\bbT)$ and 
$\VMOA(\bbT)$.

We shall denote by $A_{\rm ac}$ the a.c. part of a self-adjoint operator $A$ and by $\simeq$ the 
unitary equivalence between operators. For a linear operator $A$, we denote by $\overline{ \Ran}A$ the closure of the range of $A$.

\subsection{Hankel operators $\Gamma_u$}\label{sec.a3}

A \emph{Hankel matrix} is an infinite matrix of the form $\{\gamma_{j+k}\}_{j,k=0}^\infty$, i.e. 
the entries must depend on the sum of indices. A \emph{Hankel operator} is a bounded operator in 
the Hardy space $H^2$, whose matrix in the standard basis $\{z^k\}_{k=0}^\infty$ is a 
Hankel matrix. 
An equivalent alternative definition is that a Hankel operator is a bounded operator $\Gamma$ in 
$H^2$ such that the commutation relation 
\begin{align}
\label{e: Gamma S}
\Gamma S = S^*\Gamma , 
\end{align}
is satisfied, where $S$ is the shift operator in $H^2$. 

For a Hankel operator $\Gamma$ one can define its \emph{analytic symbol} $u$ as 
\[
u(z) := \Gamma z^0 = \sum_{k=0}^\infty \gamma_k z^k. 
\]
In this paper we will skip the word \emph{analytic} and use the term 
\emph{symbol} for $u$. We will also use the notation $\Gamma_u$ to indicate the Hankel operator with the 
symbol $u$.  
It is a well-known fact \cite[Theorem~1.1.2]{Peller} that the 
operator $\Gamma_u$ is bounded if and only if 
the symbol $u$ belongs to the class $\BMOA(\bbT)$ of the functions of bounded mean oscillation. 
On the other hand, we have $u=\Gamma_uz^0\in H^2$; it will be important for us to consider the 
symbol $u$ as an element of $H^2$.

One can give a more ``analytic'' formula for the Hankel operator $\Gamma_u$. 
Namely, denote by $J$ the involution in $L^2(\bbT)$, 
\[
Jf(z)=f(\overline{z}).  
\]
Then for $u\in\BMOA$, the Hankel operator $\Gamma_u$ with the matrix $\{\wh u_{j+k}\}_{j,k=0}^\infty$ is defined by 
\[
\Gamma_uf=P(uJf), 
\]
initially on the set of polynomials $f\in H^2$.

\subsection{Anti-linear Hankel operators $H_u$}

Clearly, Hankel matrices $\{\gamma_{j+k}\}_{j,k=0}^\infty$ are symmetric (with respect to transposition). This can be expressed as the statement that Hankel operators belong to the class of so-called \emph{complex symmetric operators}. Namely, let us denote by  $\bC$ the anti-linear (a.k.a.\ conjugate-linear)  
involution in $H^2$, 
\begin{align}
\label{e:def bC}
\bC f(z)=\overline{f(\overline{z})};
\end{align}
in other words, for $f(z) = \sum_{k=0}^\infty a_k z^k$ we have $\bC f(z) = \sum_{k=0}^\infty \overline {a}_k 
z^k$. Then the symmetry of Hankel matrices means that 
Hankel operators satisfy the identity 
\begin{align}
\Gamma_u\bC=\bC \Gamma_u^*,
\label{b1a}
\end{align}
which is exactly the definition of the so-called $\bC$-symmetric operators, cf.\ \cite{Garcia2014}. 

As it is customary in the theory of complex symmetric operators, it will be convenient to deal with 
the \emph{anti-linear} version of Hankel operators: 
$$
H_uf=\Gamma_u\bC f=P(u\overline{f}), \quad f\in H^2.
$$
Through the rest of the paper, we focus on anti-linear Hankel operators $H_u$; one exception is the 
discussion of the self-adjoint case, when it is more convenient to talk about the linear version 
$\Gamma_u$. 
Since $\bC$ satisfies
$$
\jap{\bC f,g}=\jap{\bC g,f},\quad f,g\in H^2,
$$
from the symmetry property \eqref{b1a} it follows that 
Hankel operators (in fact, all complex symmetric operators) satisfy the identity
\begin{align}
\jap{H_uf,g}=\jap{H_ug,f}, \quad f,g\in H^2. 
\label{a0a}
\end{align}
Note that for the anti-linear Hankel operator  $H_u$ we have 
\begin{align*}
H_u^2 = \Gamma_u\bC \Gamma_u\bC = \Gamma_u\bC^2 \Gamma_u^* =  \Gamma_u  \Gamma_u^*;
\end{align*}
thus $H_u^2$ is linear, self-adjoint and positive semi-definite. 
Furthermore, since the conjugation $\bC$ commutes with the shift $S$, it follows from \eqref{e: 
Gamma S} that the anti-linear  Hankel operators also satisfy the commutation relation
\begin{align}
H_uS=S^*H_u,
\label{a00}
\end{align} 
and that any bounded anti-linear operator $H_u$ on $H^2$ satisfying this commutation relation 
is a Hankel operator. 

By \eqref{a00}, the kernel of $H_u$ is an $S$-invariant subspace of $H^2$. It follows that $\Ker 
H_u$ is either trivial or infinite-dimensional. Furthermore, $\overline{\Ran} H_u$ is an invariant subspace for $S^*$. 

One ot the advantages of working with the anti-linear Hankel operators $H_u$ instead of their 
linear counterparts $\Gamma_u$ is that $\overline{ \Ran} H_u = (\Ker H_u)^\perp$. Indeed, 
\begin{align*}
\overline{\Ran} H_u = \overline{\Ran} \Gamma_u \bC = \overline{\Ran}\Gamma_u = (\Ker \Gamma_u^* 
)^\perp = (\Ker \bC\Gamma_u^* )^\perp ,  
\end{align*}
and the desired identity follows since  $\bC\Gamma_u^* = \Gamma_u\bC = H_u$.

We will denote by $H_u\ute$ the \emph{essential} part of the Hankel operator $H_u$, 
\begin{align*}
H_u\ute:= H_u|_{\overline{\Ran} H_u }. 
\end{align*}
The subspace $\overline{\Ran} H_u$ is invariant for $H_u$, and for any element $f\in H^2$ we have
$$
H_u f=H_u\ute P_{\overline{\Ran}H_u}f,
$$
where $P_{\overline{\Ran}H_u}$ is the orthogonal projection onto $\overline{\Ran}H_u$. 

\subsection{The truncated operators $\wt\Gamma_u$ and $\wt H_u$}
Along with the Hankel operators $\Gamma_u$ and $H_u$ we will consider their truncated versions
$$
\wt\Gamma_u=\Gamma_u S = S^* \Gamma_u = \Gamma\ci{S^*u},\quad  
\wt H_u = H_u S = S^* H_u = H\ci{S^*u}.
$$ 
Note that $\wt\Gamma_u$ is also a Hankel operator (with symbol $S^*u$), and 
its matrix is obtained from the matrix of $\Gamma_u$ by removing the first row (or the first column). 

As it turns out, under the assumptions discussed below, the spectral invariants of the Hankel operators $H_u$ and $\wt H_u$, described in 
Proposition~\ref{prp.spth} below,  uniquely determine the symbol $u$. 

We recall that the shift operator satisfies the identities
$$
S^*S=I, \quad SS^*=I-\jap{\cdot,z^0}z^0,
$$
where $\jap{\cdot,z^0}z^0$ is the rank one projection onto constant functions in $H^2$. From here and from the definition of $\wt H_u$ we get the rank one identity
\begin{align}
\label{e: rk1 H_u}
\wt H_u^2 = H_u^2 - \langle \fdot , u\rangle u . 
\end{align}
This identity is key to the whole inverse spectral theory of Hankel operators.

Similarly to $H_u\ute$, we denote by $\wt H_u\ute$ the \emph{essential} part of  $\wt H_u$, viz.
\begin{align*}
\wt H_u\ute:= \wt H_u|_{\overline{\Ran} H_u };
\end{align*}
since $\overline{\Ran}H_u$ is an invariant subspace for both $H_u$ and $S^*$, it is also an invariant subspace for $\wt H_u$.  We should emphasize that 
unlike $H_u\ute$, the operator $\wt H_u\ute$ can have a non-trivial (one-dimensional) kernel. 
The rank one identity \eqref{e: rk1 H_u} translates to 
\begin{align}
\label{e: rk1 H_u ess}
(\wt H_u\ute)^2 = (H_u\ute)^2 - \langle \fdot , u\rangle u . 
\end{align}

\subsection{The simplicity of the spectrum}

Our main assumption on $H_u$ and $\wt H_u$ in this paper is
\begin{align}
(H_u\ute)^2 
\quad \text{ and }\quad (\wt H_u\ute)^2 
\qquad
\text{have simple spectra.}
\label{a4}
\end{align}
We will denote by $\BMOA_\simp(\bbT)$ the set of all $u\in\BMOA(\bbT)$ satisfying \eqref{a4}. 

\begin{remark*}
On the one hand, it is very easy to construct examples of Hankel operators that do not satisfy this assumption: it suffices to consider self-adjoint Hankel operators with eigenvalues with multiplicity $>1$. 
On the other hand, there is one important particular case when the simplicity condition  \eqref{a4} holds true. This case is most conveniently described in terms of the linear realisation of Hankel operators. 
By \cite[Theorem~2.4]{GP}, if both $\Gamma_u$ and $\Gamma_{S^*u}$ are positive semi-definite, then 
the simplicity condition \eqref{a4} holds. 
\end{remark*}

Our first auxiliary result (proved in Section~\ref{sec.b}) is 
\begin{theorem}\label{thm.b1}
Let $u\in\BMOA_\simp(\bbT)$, i.e. \eqref{a4} holds. 
Then  
$u$ is a cyclic element for both $(H_u\ute)^2$ and $(\wt H_u\ute)^2$, i.e.
$$
\jap{u}_{H_u^2}=\jap{u}_{\wt H_u^2}=\overline{\Ran} H_u .
$$
\end{theorem}

\begin{remark*}
In general, $\overline{\Ran} \wt H_u\not=\jap{u}_{\wt H_u^2}$. For example, if $u=1$, then 
$\wt H_u=0$ and so $\{0\}=\overline{\Ran} \wt H_u\not=\jap{u}_{\wt H_u^2}=\Span(u)$.
\end{remark*}

\subsection{Anti-linear operators with simple spectrum of modulus}

Here we discuss a ``spectral theorem'' for a class of anti-linear operators that have 
properties mirroring those of Hankel operators. 
Let $A$ be a bounded anti-linear operator in a Hilbert space $X$, satisfying the 
identity (cf.  \eqref{a0a})
\begin{align}
\jap{Af,g}=\jap{Ag,f}
\label{a0aa}
\end{align}
for any elements $f$ and $g$ in the Hilbert space; we will call such operators \emph{symmetric 
anti-linear  operators.} Then 
\[
\jap{A^2f,f}=\jap{Af,Af}\geq0,
\]
and so $A^2$ is a (linear) positive semi-definite operator. 

Recall that for a \emph{linear} operator $T$ its \emph{modulus} $|T|$ is defined as 
$|T|:=(T^*T)^{1/2}$; the operator $T^*T$ is positive semi-definite, so its non-negative square root is well defined. 
Similarly, for an anti-linear operator $A$ satisfying \eqref{a0aa} the operator $A^2$ is 
positive semi-definite, so the non-negative square root is well defined, and we set  $|A|:=(A^2)^{1/2}$; this is a linear positive semi-definite operator.

Let us assume that $A^2$ has a simple spectrum with a cyclic element $v$.
Then trivially, $v$ is also a cyclic vector for $|A|:=(A^2)^{1/2}$. 
Let 
$\rho=\rho^{|A|}_v$ be the scalar  spectral measure for $|A|$ corresponding to the vector $v$, see \eqref{b1}. 
Note that $\rho$ is a finite measure with $\supp\rho\subset[0,\infty)$.

The spectral theorem for self-adjoint operators says that the operator $|A|$ is unitarily 
equivalent to the 
multiplication by the independent variable $s$ in $L^2(\rho)$, and the corresponding unitary 
operator $U:  L^2(\rho)\to X$ intertwining $|A|$ and the multiplication operator is given by 
\begin{align}
\label{e: U spectral Thm}
U f = f(|A|) v
\end{align}
(defined initially on polynomials $f$ and  extended by continuity).

The statement below can be regarded as a substitute for polar decomposition of linear operators. 
\begin{proposition}[Spectral Theorem for symmetric anti-linear operators]\label{prp.spth}
Let $A$ be a bounded symmetric anti-linear operator in a Hilbert space. Assume that $|A|$
has a simple spectrum with a cyclic element $v$,  and let $\rho =\rho^{|A|}_v$. 
Then there exists a unimodular Borel function $\psi$ such that 
the operator $A$ is unitariy equivalent to its model $\calA$ in $L^2(\rho)$, 
\begin{align}
\label{e: model A 01}
\calA f(s) = s\psi(s)  \overline{f(s)}\ , \qquad f\in L^2(\rho), 
\end{align}
where the unitary operator $U:L^2(\rho)\to X$, $AU=U\calA$ is given by \eqref{e: U spectral Thm}.
\end{proposition}
The proof is given in Section~\ref{sec.b}.

\begin{remark}
\label{r: uniq psi} 
It will be seen from the proof of the proposition that the function $\psi$ is uniquely defined as 
an element of $L^\infty(\rho_0)$, where $\rho_0$ is the restriction of the measure $\rho$ to 
$(0,\infty)$. Note that $\rho_0$ differs from $\rho$ if and ony if $\rho$ has an atom at $0$. 
On the other hand, it is clear that the value $\psi(0)$ is of no importance for the action of $\calA$. 
\end{remark}

\begin{remark*}
One can see from the definition \eqref{e: U spectral Thm} of the unitary operator $U$ that 
\begin{align*}
X = \{  f(|A|) v  :  f\in L^2(\rho) \};
\end{align*}
while the operators $f(|A|)$ can be unbounded, the vector $v$ is always in the domain of $f(|A|)$ 
for $f\in L^2(\rho)$. 
Thus, we can rewrite the representation \eqref{e: model A 01} for the model $\calA$ as an abstract 
representation for $A$, 
\begin{align}
\label{e: model A 02}
A f(|A|)v = |A| \psi(|A|) \overline f (|A|) v. 
\end{align}
\end{remark*}

\subsection{Direct spectral problem: spectral measures and unimodular functions}\label{sec.direct}
Let $u\in \BMOA_\simp(\bbT)$, i.e. \eqref{a4} is satisfied. 
Let us apply Proposition~\ref{prp.spth} to the anti-linear  
operators $H_u\ute$ and $\wt H_u\ute$; 
we will use the same cyclic vector $v=u$ in both cases.  

For the operator $H_u\ute$ we get its spectral measure 
$\rho=\rho_{u}^{\abs{H_u\ute}}$; note that since $u\in \Ran H_u$, we have
$\rho_{u}^{\abs{H_u\ute}}=\rho_u^{\abs{H_u}}$; we will use the notation 
$\rho_{u}^{\abs{H_u}}$ for typographical reasons. 
We also get  the unitary operator 
 $U:L^2(\rho)\to \overline{\Ran}H_u^2$  given by \eqref{e: U spectral Thm} with $A= 
H_u\ute $ and $v=u$, 
\begin{align}
\label{e: U spectral Thm 01}
U f  = f( | H_u\ute|  ) u , 
\qquad f\in L^2(\rho),  
\end{align}
so 
\[
U^*|H_u\ute|U =\calM, 
\]
where $\calM$ is the operator of multiplication by the independent variable $s$ in $L^2(\rho)$. 
By Proposition~\ref{prp.spth} we have
\begin{align}
\label{e: model H 01}
\left[U^*  H_u\ute U \right] f(s) =  \overline{\Psi_u(s)} s 
\overline{f(s)}, 
\qquad &f\in L^2(\rho) 
\end{align}
where $\Psi_u$ is a complex-valued unimodular Borel function; we write $\overline{\Psi}_u$ rather 
than $\Psi_u$ in the above formula for consistency of notation with \cite{GGAst}. 

Similarly, defining the spectral measure 
$\wt\rho=\rho_{u}^{\abs{\wt H_u}}$ (again, it coincides with the spectral measure of the operator 
$|\wt H_u\ute|$) and the unitary operator $\wt U: L^2(\wt\rho) \to \overline{\Ran}H_u^2$ by 
\begin{align*}
\wt U f  = f( |\wt H_u\ute| ) u, \qquad f\in L^2(\wt\rho), 
\end{align*}
we get that 
\begin{align}
\label{e: model H 02}
\left[\wt U^* \wt H_u\ute \wt U \right] f(s) = {\wt \Psi_u(s)} 
s 
\overline{f(s)}, 
\qquad &f\in L^2(\wt\rho) , 
\end{align}
where $\wt \Psi_u$ is a Borel unimodular function. 

To summarise: we have two  measures $\rho$, $\wt\rho$ and two unimodular functions $\Psi_u$ and $\wt\Psi_u$ as spectral characteristics of the Hankel operator $H_u$.

Since the measure $\rho$ does not have an atom at $0$, by Remark~\ref{r: uniq psi}  the function 
$\Psi_u$ is unique as element of $L^\infty(\rho)$. However, the measure $\wt\rho$ can have an atom 
at 
$0$, so we can only say that $\wt\Psi_u$ is unique as an element of $L^\infty(\wt\rho_0)$, 
where $\wt\rho_0$ is the restriction of $\wt\rho$ to $(0,\infty)$.  Also, one can see from 
\eqref{e: model H 02} that the value $\wt \Psi_u(0)$ does not matter for the action of $\wt H_u\ute$, so we can assume that $\wt\Psi_u$ is unique in $L^\infty(\wt\rho)$. 

\subsection{Remarks about the measures $\rho$ and $\wt\rho$}\label{s: 
normalization rho}

The measure $\rho$ must satisfy 
\begin{align}
\int_0^\infty \frac{d\rho(s)}{s^2}\leq1.   
\label{rholeq1}
\end{align}
Indeed, we  know that 
\begin{align*}
u=H_u z^0 =  H_u\ute P^{\vphantom{\dagger}}_{\overline{\Ran} H_u} z^0, 
\qquad 
U^* u =\1 , 
\end{align*}
so the representation \eqref{e: model H 01} implies that $U^*$ 
maps the vector $P\ci{\overline{\Ran} H_u} z^0$ to the function $q\in L^2(\rho)$,  $q(s) = 
\overline{\Psi_u(s)}/s$. Since $\| P\ci{\overline{\Ran} H_u} z^0\|\ci{H^2}\le \| 
z^0\|\ci{H^2}=1$, we conclude that $\|q\|\ci{L^2(\rho)}\le 1$, which is exactly the estimate 
\eqref{rholeq1}.

The measures $\rho$ and $\wt\rho$ are not  independent, and that $\wt\rho$ is 
uniquely defined by $\rho$. To explain this, we introduce two important operators $\calM$ and $\wt\calM$ in $L^2(\rho)$ that will play a key role in our construction below. We have already defined  $\calM$ in the previous subsection; this is the multiplication operator by 
the independent variable $s$ in $L^2(\rho)$. Now consider the operator 
\begin{align*}
\calM^2 - \jap{\fdot , \1} \1 = \calM ( I -  \jap{\fdot , q_0}q_0  ) \calM, 
\end{align*}
where $q_0(s) =1/s$. 
The inequality \eqref{rholeq1} implies  that $\|q_0\|\ci{L^2(\rho)}\le 1$, so the 
above operator is trivially non-negative. Let us consider its (non-negative) square root   
\begin{equation}
\wt\calM := \left( \calM^2 - \jap{\fdot , \1} \1  \right)^{1/2}.
\label{defwtcalM}
\end{equation}
The definition of $\wt\calM$ can be equivalently rewritten as
$$
\wt\calM^2=\calM^2-\jap{\cdot,\1}\1,
$$
which mirrors the rank one identity \eqref{e: rk1 H_u ess}. 

We can easily see that the unitary equivalence $U$ maps the triple $(\calM,\wt\calM,\1)$ to the triple $(\abs{H_u\ute},\abs{\wt H_u\ute},u)$,  so  $\wt\rho$ is the spectral measure of the operator $\wt\calM$ with respect to the vector 
$\1\in L^2(\rho)$. Thus $\wt\rho$ is uniquely determined by $\rho$.

\subsection{The spectral data and Uniqueness}
\label{s: SpecData}
To conclude, with each Hankel operator $H_u$ with $u\in \BMOA_\simp(\bbT)$ we associate the 
following \emph{spectral datum:}
\begin{enumerate}
\item The measure $\rho$ with bounded support on $(0,\infty)$ satisfying the normalization 
\eqref{rholeq1} (and 
the measure $\wt\rho$ on $[0,\infty)$, uniquely defined by $\rho$ as described above in Section~\ref{s: normalization rho}). 
\item Two unimodular functions $\Psi_u\in L^\infty(\rho)$ and $\wt\Psi_u\in L^\infty(\wt\rho_0)$, 
where 
$\wt\rho_0 :=\wt\rho\big|_{(0,\infty)}$; the functions $\Psi_u$ and $\wt\Psi_u$ are unique as 
vectors 
in the corresponding $L^\infty$ spaces.  
\end{enumerate}
So, formally speaking the spectral datum for $u$ (equivalently $H_u$) is given by the triple  
$$
\Lambda(u):=(\rho,\Psi_u,\wt \Psi_u).
$$ 
We do not include the measure $\wt\rho$ in the spectral data because $\wt\rho$ is determined by $\rho$, as explained in the previous subsection. 

Our first main result is 
\begin{theorem}[Uniqueness]\label{thm.b3}
Any symbol $u\in\BMOA_\simp(\bbT)$ is uniquely determined by the spectral datum 
$\Lambda(u)$, i.e.~the spectral map 
\begin{equation}
\BMOA_{\simp}(\bbT)\ni u\mapsto \Lambda(u)=(\rho,\Psi_u,\wt\Psi_u)
\label{eq.specmap}
\end{equation}
is injective.
\end{theorem} 
Moreover, we will give an explicit formula for the symbol $u$ in terms of the spectral datum, see 
\eqref{e: u_k 01} and \eqref{e: u(z) 01} below. 
The proof of Theorem~\ref{thm.b3} is given in Section~\ref{sec.b}.

Recall that $\Ker H_u$ is either trivial or infinite dimensional. 
It turns out that one can easily distinguish between these two cases by looking at the spectral data.

\begin{theorem}[Triviality of kernel]\label{thm.b3aa}
For $u\in\BMOA_\simp(\bbT)$, we have $\Ker H_u=\{0\}$ if and only if 
\begin{align}
\label{e: triv ker}
\int \frac{d\rho(s)}{s^2} =1 \quad\text{and}\qquad \int \frac{d\rho(s)}{s^4} = \infty. 
\end{align}
\end{theorem}

This theorem was proved in \cite[Theorem 4]{GG1}. More precisely, in \cite{GG1}, it was stated for compact $H_u$ and in slightly different terms, but the idea of the proof remains the same. 
For the case of self-adjoint Hankel operators  it also appeared earlier in  \cite[Theorem 
III.2.1]{MPT}; a similar dynamical systems approach also works in the general case.

For completeness we give a proof in the Appendix~\ref{app.b}.
We note that the first condition in \eqref{e: triv ker} is equivalent to $z^0\in\overline{\Ran H_u}$,  and the second one is equivalent to $z^0\notin\Ran H_u$, see the proof. 

\subsection{The self-adjoint case}
Here we discuss the interesting special case when the linear Hankel operator $\Gamma_u$ is self-adjoint. Evidently, $\Gamma_u$ is self-adjoint if and only if all Fourier coefficients $\widehat u_j$ are real; if $\Gamma_u$ is self-adjoint, then so is $\wt\Gamma_u$. 

\begin{theorem}\label{thm.sa1}
Let $u\in\BMOA_\simp(\bbT)$; then $\Gamma_u$ is self-adjoint if and only if both $\Psi_u$ and 
$\wt \Psi_u$ are functions with values $\pm1$. 
\end{theorem}

Moreover, in the self-adjoint  case formulas \eqref{e: model H 01}, \eqref{e: model H 02} for the action of $H_u$ and $\wt H_u$ can be interpreted as polar decompositions of $\Gamma_u$ and $\wt\Gamma_u$. In order to state this precisely, we recall the relevant key definitions and facts.

For a bounded operator $T$ on a Hilbert space, there exists a unique partial isometry $\Phi$ with the initial subspace $\overline{\Ran}T^*$ and the final subspace $\overline{\Ran}T$ such that the \emph{polar decomposition} $T=\Phi\abs{T}$ holds, where $\abs{T}=\sqrt{T^*T}$. If $T$ is self-adjoint, then $\Phi$ is also self-adjoint and commutes with $\abs{T}$. Furthermore, if the spectrum of $\abs{T}$ is simple, then one can write $\Phi=\varphi(\abs{T})$, where $\varphi$ is a Borel function with values $\pm1$. The function $\varphi$ is uniquely defined up to values on sets of measure zero with respect to the spectral measure of $\abs{T}$. One can also write $\Phi=\varphi(\abs{T})$ if $T$ has a multi-dimensional kernel but the spectrum of the restriction $\abs{T}|_{\overline{\Ran}T}$ is simple; in this case one must set $\varphi(0)=0$. 

We apply this to the case $T=\Gamma_u$ or $T=\wt\Gamma_u$; note that in this case $\abs{\Gamma_u}=\abs{H_u}$ and $\abs{\wt\Gamma_u}=\abs{\wt H_u}$. 

\begin{theorem}\label{thm.sa2}
Let $u\in\BMOA_\simp(\bbT)$ be such that then $\Gamma_u$ is self-adjoint. Then the polar decompositions of $\Gamma_u$ and $\wt\Gamma_u$ can be written as
\begin{equation}
\Gamma_u=\Psi_u(\abs{\Gamma_u})\abs{\Gamma_u}, \quad
\wt\Gamma_u=\wt\Psi_u(\abs{\wt\Gamma_u})\abs{\wt\Gamma_u},
\label{pd}
\end{equation}
where one should set $\Psi_u(0)=\wt\Psi_u(0)=0$ in case of non-trivial kernels. 
\end{theorem}

The proofs of the above two theorems are given in Section~\ref{sec.b}. 

\subsection{What can be said about the case of non-trivial spectral multiplicity?}
We conclude this section with remarks on the case when the simplicity assumption \eqref{a4} is not 
satisfied. What would be the natural choice for the spectral datum in this case? 


This question was answered in \cite{GGAst} for the case of compact Hankel operators $H_u$. Observe 
that in this case, the measure $\rho$ is purely atomic, supported on the set of singular values of 
$H_u$. The spectral datum is still the triple $(\rho,\Psi_u,\widetilde\Psi_u)$, but the functions 
$\Psi_u$ and $\widetilde\Psi_u$ (defined on the set of singular values of $H_u$ and $\wt H_u$ 
respectively) are no longer scalar-valued but take values in the set of all finite \emph{Blaschke 
products}. In \cite{GGAst} it is proved that the spectral map, defined in a suitable way, is 
injective and surjective. 

Another case was considered in \cite{GP2}: all Hankel operators $H_u$ such that the spectrum of 
$\abs{H_u}$ is finite.  In a similar spirit, the spectral datum is the triple 
$(\rho,\Psi_u,\widetilde\Psi_u)$, where $\Psi_u$ and $\widetilde\Psi_u$ are functions from the 
spectrum of $\abs{H_u}$ and $\abs{\wt H_u}$ into the set of all \emph{inner functions}, and the 
spectral map was proved to be injective and surjective. 

As for the general case, in \cite{LiangTr} an abstract approach to the inverse spectral problem for
general Hankel operators was considered. The abstract spectral datum there is similar in spirit to
what is presented here, but the values of functions $\Psi_u$ and $\wt \Psi_u$ are unitary operators.
In addition, a special anti-linear conjugation $\bJ$, commuting with both $|H_u|$ and $|\wt
H_u|$ (which is implicit in this paper) is also a part of spectral datum. The spectral map is
injective, if one treats the spectral data as natural equivalence classes. And similarly to the
present paper, the abstract spectral datum corresponds to a Hankel operator if and only if an
appropriately constructed operator is asymptotically stable. 

For the case of compact operators, the (non-trivial) translation from the language used in
\cite{LiangTr} to the description in \cite{GP2} was provided in \cite{LiangTr}.

It is likely that the constructions of \cite{GGAst} and \cite{GP2} can be combined to give a
description of spectral map in the case when $\abs{H_u}$ has only point spectrum. It could also 
be possible to use the ideas from \cite{LiangTr} to extend the result to the case of purely singular
spectrum.

However, the fundamental question  of transparent representation of the spectral data in the 
general case when $\abs{H_u}$ has non-trivial absolutely continuous spectrum and non-trivial 
multiplicity remains a mystery.

\section{The problem of surjectivity}\label{sec.bb}

\subsection{The abstract spectral data and the problem of surjectivity} 
\label{s: ASD}

First let us discuss

\medskip
\textbf{Question:}
\emph{What is the natural target space for the spectral map \eqref{eq.specmap}?} 
\medskip

\noindent
Below we describe the set of triples $(\rho,\Psi,\wt\Psi)$, that we call the \emph{abstract spectral data}, that plays the role of the target space. 

Let  $\rho$ be a finite Borel measure with a bounded support on $(0,\infty)$,  satisfying the normalization condition \eqref{rholeq1}. We then define the operators $\calM$ and $\wt\calM$ in $L^2(\rho)$ exactly as explained in Section~\ref{s: normalization rho}, i.e. 
$\calM$ is the multiplication by the independent variable and 
$$
\wt\calM := \left( \calM^2 - \jap{\fdot , \1} \1  \right)^{1/2}.
$$ 
Let $\wt\rho$ be the spectral measure of $\wt\calM$, corresponding to the vector $\1$. 
Picking two unimodular functions $\Psi\in L^\infty(\rho)$ and $\wt\Psi\in 
L^\infty(\wt\rho_0)$ (where $\wt\rho_0 :=\wt\rho\big|_{(0,\infty)}$), we get the triple 
$$
\Lambda = (\rho, \Psi, \wt\Psi),
$$ 
which we will call the \emph{abstract spectral datum}: the word 
\emph{abstract} here emphasizes the fact that this datum a priori does not have to come from a Hankel operator. The set of all abstract spectral data is the natural target space for the spectral map  \eqref{eq.specmap}. 

We arrive at the main problem addressed in this paper: 

\medskip
\textbf{Question:}
\emph{Is the spectral map \eqref{eq.specmap} surjective?}
\medskip

\noindent
In other words, does every abstract spectral datum come from a Hankel operator?

For several years, the authors of this paper believed that the answer is ``yes''. For example, as it was proved in 
\cite{GG1}, the answer is affirmative in the case of compact Hankel operators: in this case the 
measure $\rho$ is a purely atomic measure with $0$ being the only possible accumulation point of 
its support.

The other case is the so-called \emph{double positive} case, treated in \cite{GP}, where both 
operators $\Gamma_u$ and $\wt \Gamma_u$ are non-negative self-adjoint operators. It was shown 
in \cite[Theorem~2.4]{GP} that in this case the simplicity condition \eqref{a4} is satisfied. 
In this case both unimodular functions $\Psi_u$ and $\wt \Psi_u$ are identically equal to $1$.  It 
was also shown in \cite{GP} that in this case any abstract spectral datum (i.e.~any measure $\rho$ satisfying the normalization condition \eqref{rholeq1}) comes from a self-adjoint Hankel operator $\Gamma_u$.

\subsection{Informal description of main results} 
In order to simplify our discussion, we introduce the following notation. For an abstract spectral datum $\Lambda_*=(\rho,\Psi,\wt\Psi)$, we will write $\Lambda_*\in\Lambda(\BMOA_\simp)$, if $\Lambda_*$ is in the range of the spectral map \eqref{eq.specmap}, i.e. if $\Lambda_*$ is the spectral datum of some Hankel operator $H_u$. 

Here we informally describe our main results.
\begin{itemize}
\item
The spectral map \eqref{eq.specmap} is NOT surjective, i.e. there are abstract spectral data with $\Lambda_*\notin\Lambda(\BMOA_\simp)$. 
\end{itemize}

We do not have a simple easy-to-check criterion for an abstract spectral datum to be in $\Lambda(\BMOA_\simp)$, but we come close to it. 

\begin{itemize}
\item
For an abstract spectral datum $\Lambda_*$, 
we have $\Lambda_*\in\Lambda(\BMOA_\simp)$ if and only if a certain contraction $\Sigma^*$, constructed from $\Lambda_*$, is asymptotically stable (i.e. $(\Sigma^*)^n\to0$ strongly as $n\to\infty$.) See Theorem~\ref{thm.a5} for the precise statement. 
\end{itemize}

The asymptotic stability of $\Sigma^*$ is not easy to check. However, in many cases we can reduce it to a more explicit condition. 
\begin{itemize}
\item
Under some mild additional assumptions (e.g. $\Psi$ and $\wt\Psi$ are H\"older continuous at $0$), we have $(\rho,\Psi,\wt\Psi)\in\Lambda(\BMOA_\simp)$ if and only if the unitary operator 
$$
\wt\Psi(\wt\calM)\Psi(\calM)
$$ 
has a purely singular spectrum. Here $\calM$ and $\wt\calM$ are the operators in $L^2(\rho)$ defined in the previous subsection. 

\item
Using the previous result, we construct a wide range of examples of spectral data that  are (or are not) in $\Lambda(\BMOA_\simp)$. 
\end{itemize}

\subsection{Introducing the model $(\calH, \wt\calH,\Sigma^*)$}\label{sec.model}
Let $u\in\BMOA_{\simp}(\bbT)$. 
Restricting the identity $\wt H_u = S^*H_u$ to the $S^*$-invariant subspace $\overline{\Ran} H_u$ we write 
\begin{equation}
\wt H_u \ute = \left(S^* \big|_{\overline{\Ran} H_u} \right) 
 H_u \ute.
\label{w1}
\end{equation}
Recall also the rank one identity \eqref{e: rk1 H_u ess}. 
Let us map these identities to $L^2(\rho)$, where $\rho=\rho_u^{\abs{H_u}}$, by using the unitary operator $U$ defined in \eqref{e: U spectral Thm 01}. 
In order to do this, let us define the anti-linear operators $\calH$, $\wt\calH$ and the (linear) contraction $\Sigma$ in $L^2(\rho)$ by
\begin{align}
\notag
\calH    & = U^*   H_u\ute  U, \\
\notag
\wt\calH    & = U^*  \wt H_u\ute  U, \\
\label{e: model Sigma*}
\Sigma^* & = U^*\left(S^* \big|_{\overline{\Ran} H_u} \right) U , \qquad \Sigma := (\Sigma^*)^* . 
\end{align}
Multiplying \eqref{w1} and \eqref{e: rk1 H_u ess} by  $U^*$ on the left and by $U$ on the right, we obtain the identities
\begin{align}
\wt \calH&=\Sigma^*\calH, 
\label{w5}
\\
\wt\calH^2&=\calH^2-\jap{\cdot,\1}\1
\label{w6}
\end{align}
in $L^2(\rho)$.

\emph{The triple $(\calH, \wt\calH,\Sigma^*)$ is our model for $(H_u\ute,\wt H_u\ute,S^*|_{\overline{\Ran}H_u})$; this model plays a central role in our construction.}

Rewriting \eqref{e: model H 01}, \eqref{e: model H 02} in terms of the model operators $\calH$, $\wt\calH$, we obtain 
\begin{align}
\calH f&=\calM\overline{\Psi}(\calM)\overline{f}, \quad f\in L^2(\rho),
\label{w3}
\\
\wt\calH f&=\wt\calM\wt\Psi(\wt\calM)\overline{f}, \quad f\in L^2(\rho),
\label{w4}
\end{align}
where $\Psi=\Psi_u$, $\wt\Psi=\wt\Psi_u$ and the operators $\calM$ and $\wt\calM$ are as discussed in Section~\ref{s: normalization rho}.

\subsection{Model coming from abstract spectral data}\label{sec.absmodel}
One can also set up a triple $(\calH,\wt\calH,\Sigma^*)$ starting from an abstract spectral datum $\Lambda=(\rho,\Psi,\wt\Psi)$. We define the operators $\calM$ and $\wt\calM$ as described in Section~\ref{s: ASD} and define the anti-linear operators $\calH$ and $\wt\calH$ by \eqref{w3} and \eqref{w4}. In order to define $\Sigma$, we first note that $\wt\calM^2\leq \calM^2$, i.e. 
$$
\norm{\wt\calM f}\leq\norm{\calM f}, \quad \forall f\in L^2(\rho),
$$
and therefore (see Douglas' lemma in Section~\ref{sec.e}) the operator $\wt\calM\calM^{-1}$, defined initially on the dense set $\Ran \calM$, extends to $L^2(\rho)$ as a contraction. We then define the contraction
\begin{equation}
\Sigma^*=\wt\Psi(\wt\calM)\wt\calM\calM^{-1}\Psi(\calM)
\label{w7}
\end{equation}
and set $\Sigma=(\Sigma^*)^*$. 
With these definitions, the key identities \eqref{w5} and \eqref{w6} are satisfied.

\subsection{Surjectivity: reduction to the asymptotic stability of $\Sigma^*$}
Let $\Lambda=(\rho,\Psi,\wt\Psi)$ be an abstract spectral datum, and let $\Sigma^*$ be as defined in \eqref{w7}. 
Our first main result concerning surjectivity is
\begin{theorem}\label{thm.a5}
The triple $\Lambda=(\rho,\Psi,\wt\Psi)$ is the spectral datum for some Hankel operator $H_u$ with $u\in\BMOA_\simp$ if and only if $\Sigma^*$ is asymptotically stable (i.e. $\Sigma^{*n}\to0$ in the strong operator topology). 
\end{theorem}

\subsection{Explicit formula for the symbol}
One can give an explicit formula for the symbol $u$ in terms of the corresponding operator $\Sigma^*$, defined via the spectral data of $u$. The following statement is logically part of the uniqueness Theorem~\ref{thm.b3}, but we place it here because it was convenient for us to state the uniqueness theorem before describing the model $(\calH,\wt\calH,\Sigma^*)$.

\begin{theorem}\label{thm.a5a}
Let $u\in\BMOA_\simp$, and let $\Sigma^*$ be defined by \eqref{e: model Sigma*}. Then $u$ can be found through the explicit formula
\begin{align}
\label{e: u_k 01}
\widehat u_k &= \jap{(\Sigma^*)^k \1, q}\ci{L^2(\rho)}, \qquad k\ge 0, 
\intertext{or equivalently}
\label{e: u(z) 01}
u(z) &= \jap{(I-z\Sigma^*)^{-1} \1, q}\ci{L^2(\rho)} , \qquad z\in\bbD ,  
\end{align}
where $q\in L^2(\rho)$, $q(s) := \overline{\Psi(s)}/s$. 
\end{theorem}

\section{Proofs of preliminary results}\label{sec.b}

\subsection{Proof of the ``spectral theorem'' (Proposition~\ref{prp.spth})}

Let $\calA$ be the anti-linear operator in $L^2(\rho)$, defined by $\calA := U^* A U$, where 
$U$ is 
the unitary operator defined by \eqref{e: U spectral Thm}. Then, since $A$ trivially commutes 
with 
$|A|^2=A^2$ and $U^*|A|U = \calM$, where $\calM$ is the multiplication by the independent variable 
$s$ 
in $L^2(\rho)$, we conclude that $\calA$ commutes with $\calM^2$ and so with $\calM$. 

Denote by $\fC$ the standard conjugation acting on functions on $\R$, 
\[
\fC f (s) = \overline{f(s)}, 
\]
and define the \emph{linear} operator $\calB$ in $L^2(\rho)$ as $\calB:= \fC \calA$. 
Since $\fC$ commutes with $\calM$, we find that  $\calB$ also commutes with $\calM$. 
Therefore, $\calB$ is the multiplication by 
a a 
function $g\in L^\infty(\rho)$, 
\begin{align*}
\calB f = g f \qquad \forall f \in L^2(\rho); 
\end{align*}
note that $g$ as an element of $L^\infty(\rho)$ is unique. 

For any $f\in L^2(\rho)$
\begin{align*}
\|gf\|\ci{L^2(\rho)}^2 = \|\calB f \|\ci{L^2(\rho)}^2 =  \|\calA f \|\ci{L^2(\rho)}^2 = 
\jap{\calA^2 f,f} =\jap{\calM^2 f,f}= \int_\R s^2 |f(s)|^2 \dd\rho (s);
\end{align*}
the second equality holds because $\fC$ preserves the norm.  
So we conclude that $|g(s)|^2 = s^2$ $\rho$-a.e., therefore it can be 
represented 
as 
\begin{align*}
g(s) = s \overline{\psi(s)}, 
\end{align*}
where $\psi$ is a unimodular function, i.e.~$|\psi(s)|=1$ $\rho$-a.e.\ (the reason for the 
complex 
conjugation is purely notational, and will be clear in a moment). 

Using the fact that $\calA=\fC\calB$ we conclude that 
\begin{align*}
\calA f(s) = \overline{ g(s) f(s)} =\overline{ s \overline{\psi(s)} f(s) } = s \psi(s) 
\overline{f(s)}, 
\end{align*}
which is exactly the conclusion of the proposition. 
\hfill\qed

\subsection{Cyclicity of $u$: preliminaries}
\label{sec.b3}

To prove  Theorem~\ref{thm.b1} we start with a trivial observation.
\begin{lemma}\label{l: cyclicity} 
Let $R=R^*$ be a bounded self-adjoint operator, and let $R_\alpha : = R +\alpha 
\jap{\fdot, p}p$, $\alpha\in\R$, be its rank one perturbation. Then 
\begin{enumerate}
\item There holds $\jap{p}\ci R =\jap{p}\ci{R_\alpha}$.
\item If both $R$ and $R_\alpha$ have  simple spectrum, then there exists a vector 
$v_2\in\jap{p}\ci{R}^\perp$ such that  the vector $v=p+v_2$
cyclic  for both $R$ and $R_\alpha$. 
\end{enumerate}
\end{lemma}

\begin{proof}
The first statement is easy: by induction we find 
\begin{align*}
R_\alpha^n p \in \spn \{R^k p : 0\le k \le n\}, 
\end{align*}
which implies  the inclusion $\jap{p}\ci{R_\alpha} \subset \jap{p}\ci{R}$ for all $\alpha\in\R$. 
Since $R = R_\alpha -\alpha \jap{\fdot, p}p$, the converse inclusion follows. 

By the statement \cond1, the subspace $\jap{p}\ci{R}$ is an invariant subspace for both $R$ and 
$R_\alpha$; since both operators are self-adjoint it is in fact reducing for both. Furthemore, the action of the operators $R$ and $R_\alpha$ coincide on $\jap{p}\ci R^\perp$. Now it remains to take $v=p+v_2$, where $v_2$ is a cyclic vector for $R\,\big|\jap{p}\ci{R}^\perp$. 
\end{proof}

In the proof of Theorem~\ref{thm.b1} we use the model $(\calH,\wt\calH,\Sigma^*)$ introduced in Section~\ref{sec.model}. Note, however, that in Section~\ref{sec.model} we have used the fact that $u$ is cyclic for $H_u \ute$ and $\wt H_u \ute$. In order to avoid a circular argument, here we start the proof by setting up a slight modification of the same model, using another cyclic vector, which exists by Lemma~\ref{l: cyclicity}.

Thus, let $v=u+v_2\in\overline{\Ran} H_u$ be the cyclic vector for both 
$(H_u\ute)^2$ and $(\wt H_u\ute)^2$, which exists by 
statement \cond2 (with $p=u$) of Lemma~\ref{l: cyclicity}.  Let $\rho=\rho_v^{|H_u|}$ be the 
spectral measure of $|H_u\ute|$ corresponding to $v$. 
Let $U$  be  the unitary operator  given by \eqref{e: U spectral Thm} with $A=H_u\ute$, 
\[
Uf := f(|H_u\ute|) v ,
\qquad f\in L^2(\rho) . 
\]
As in Section~\ref{sec.model}, we define the operators $\calH$, $\wt\calH$ and $\Sigma$ in $L^2(\rho)$ by
\begin{align*}
\calH    & = U^*   H_u  U, \\
\wt\calH    & = U^*  \wt H_u\ute  U, \\
\Sigma^* & = U^*\left(S^* \big|_{\overline{\Ran} H_u} \right) U , \qquad \Sigma := (\Sigma^*)^* . 
\end{align*}
For these operators, from the definition of $\wt H_u$ and from the rank one identity \eqref{e: rk1 H_u ess}
we obtain 
\begin{align}
\label{e: Sigma* 01}
\wt \calH&= \Sigma^* \calH =\calH \Sigma, 
\\
\wt\calH^2&= \calH^2 - \jap{\fdot , p} p, 
\label{e: Sigma* 02-1}
\end{align}
where $p=U^* u$. 
Note that $U^* v=\1$, so $U^*u=\chi\ci E $ for some Borel set $E\subset\sigma(|H_u|)$; what will be essential here is that both $U^*v$ and $U^*u$ are real-valued.  

\subsection{Proof of Theorem~\ref{thm.b1}}

Step 1: the action of $\calH$ and $\wt\calH$ on $L^2(\rho)$. 

By construction, $v$ is a cyclic element for both $\abs{H_u\ute}$ and $\abs{\wt H_u\ute}$, thus $\1=U^*v$ is a cyclic element for both $\abs{\calH}$ and $\abs{\wt\calH}$. 
By Proposition~\ref{prp.spth}, we find
\begin{align*}
\calH f(\abs{\calH})\1&=\abs{\calH}\psi(\abs{\calH})\overline{f}(\abs{\calH})\1,
\\
\wt\calH f(\abs{\wt\calH})\1&=\abs{\wt\calH}\wt\psi(\abs{\wt\calH})\overline{f}(\abs{\wt\calH})\1
\end{align*}
for some unimodular functions $\psi$ and $\wt\psi$ and for all admissible $f$ (i.e. $f\in L^2(\rho)$ for the first identity and $f(\abs{\wt\calH})\1\in L^2(\rho)$ for the second one).

Step 2: conjugations on $L^2(\rho)$. 

Since $\abs{\calH}$ coincides with the operator $\calM$ of multiplication by the independent variable in $L^2(\rho)$, we find that 
$$
\overline{g}=\overline{f}(\abs{\calH})\1, \quad \text{ if $g=f(\abs{\calH})\1$ }
$$
for any admissible $f$. Further, if $f(x)=x^{2n}$, using the fact that $p$ is real-valued, from \eqref{e: Sigma* 02-1} we find that 
\begin{equation}
\overline{g}=\overline{f}(\abs{\wt\calH})\1, \quad \text{ if $g=f(\abs{\wt\calH})\1$.}
\label{c.*1}
\end{equation}
Taking linear combinations and using an approximation argument, we obtain \eqref{c.*1} for all admissible $f$. To conclude, combining with the previous step, we find that
$$
\calH g=\abs{\calH}\psi(\abs{\calH})\overline{g},
\qquad
\wt\calH g=\abs{\wt\calH}\wt\psi(\abs{\wt\calH})\overline{g}
$$
for all $g\in L^2(\rho)$. 

Step 3: The action of $\Sigma^*$ in $L^2(\rho)$. 

Recall that $\jap{p}_{\calH^2}^\perp$ is an invariant (in fact, reducing) subspace for both operators $\abs{\calH}$ and $\abs{\wt\calH}$ and the actions of these operators coincide on this subspace. 
Thus, for all $g\in\jap{p}_{\calH^2}^\perp$ we have
$$
\calH g=\abs{\calH}\psi(\abs{\calH})\overline{g},
\qquad
\wt\calH g=\abs{\calH}\wt\psi(\abs{\calH})\overline{g}.
$$
By \eqref{e: Sigma* 01}, we find that $\jap{p}_{\calH^2}^\perp$ is an invariant subspace for $\Sigma^*$ and the action of $\Sigma^*$ on this subspace reduces to the multiplication by a unimodular function:
$$
\Sigma^*g=\wt\psi(\abs{\calH})\overline{\psi}(\abs{\calH})g, \quad g\in \jap{p}_{\calH^2}^\perp. 
$$
It follows that 
$$
\norm{(\Sigma^*)^n g}_{L^2(\rho)}=\norm{g}_{L^2(\rho)}, \quad g\in \jap{p}_{\calH^2}^\perp
$$
for all $n\geq0$. On the other hand, $\Sigma^*$ is unitarily equivalent to the restriction of $S^*$ onto its invariant subspace $\overline{\Ran}H_u$, and we know that $(S^*)^n\to0$ in the strong operator topology as $n\to\infty$. It follows that $(\Sigma^*)^n\to0$ in the strong operator topology. We have arrived at a contradiction. 
\qed

\subsection{Proof of Theorems~\ref{thm.b3} (uniqueness) and \ref{thm.a5a} (formula for 
$u$)}\label{sec.dd}
Throughout this section, we fix $u\in\BMOA_\simp(\bbT)$ and the corresponding Hankel operator 
$H_u$, and set  $\rho=\rho_u^{\abs{H_u}}$. We use the model $(\calH,\wt\calH,\Sigma)$ of Section~\ref{sec.model}.

We first recall that by the definition \eqref{e: U spectral Thm 01} of $U$, we have $U^*u=\1$. Further, as discussed in Section~\ref{s: normalization rho}, we have $u=H_u\ute P_{\overline{\Ran}H_u}z^0$ and 
$$
U^*P_{\overline{\Ran}H_u}z^0=q, \quad q(s)=\overline{\Psi_u(s)}/s. 
$$
Thus, for $k\geq0$ we have
\begin{align}
\wh u_k&=\jap{u,z^k}=\jap{u,S^kz^0}=\jap{(S^*)^ku,z^0}
=\jap{(S^*|_{\overline{\Ran}H_u})^ku,P_{\overline{\Ran}H_u}z^0}
\notag
\\
&=\jap{(\Sigma^*)^kU^*u,U^*P_{\overline{\Ran}H_u}z^0}
=\jap{(\Sigma^*)^k\1,q}.
\label{e: u_k}
\end{align}
Since all objects in the right hand side are defined in terms of the spectral datum $\Lambda(u) = 
(\rho, \Psi_u, \wt\Psi_u)$, the injectivity of the map $u\mapsto \Lambda(u)$ is proved. 
The proof of Theorem~\ref{thm.b3} is complete. \qed

Multiplying both sides of \eqref{e: u_k} by $z^k$ and summing over $k\ge0$ we get an explicit formula for $u$, 
\[
u(z) = \jap{(I-z\Sigma^*)^{-1} \1, q}\ci{L^2(\rho)} , \qquad z\in\bbD . 
\]
The proof of Theorem~\ref{thm.a5a} is complete. \qed

\subsection{Self-adjoint case: proof of Theorems~\ref{thm.sa1} and \ref{thm.sa2}}
\label{s:self-adjoint-easy}
First let us assume that $\Gamma_u$ is self-adjoint, i.e. that all coefficients $\wh u_j$ are real. We will prove that both $\Psi_u$ and $\wt\Psi_u$ take values $\pm1$ and the polar decomposition \eqref{pd} holds. 

Let us rewrite \eqref{e: model A 02} for $A = H_u\big|_{\overline{\Ran} H_u}$, $v=u$   in terms 
of the linear realization $\Gamma_u$. Since $\Gamma_u=\Gamma_u^*$, we have  $|H_u| = 
|\Gamma_u^*|=|\Gamma_u|$, and $H_u=\bC \Gamma_u =\Gamma_u\bC$ (so~$\Gamma_u$ commutes with the 
conjugation $\bC$ defined in \eqref{e:def bC}). 

Noticing that $\overline{\Ran}\Gamma_u = \overline{\Ran}H_u$ is $\Gamma_u$-invariant, and defining 
$\Gamma_u\ute := \Gamma_u \big|_{\overline{ \Ran}\Gamma_u} = \Gamma_u \big|_{\overline{ \Ran} H_u} 
$,  we can rewrite \eqref{e: model A 
02} for $A = H_u\ute :=H_u\big|_{\overline{\Ran} H_u}$, with $v=u$  as
\begin{align}
\label{e: Gamma_u bC=}
\Gamma_u\ute\bC 
f(\abs{\Gamma_u\ute})u=\abs{\Gamma_u\ute}\overline{\Psi}_u(\abs{\Gamma_u\ute})\overline{f}(\abs{\Gamma_u})u,
 \quad f\in L^2(\rho) . 
\end{align}

We have $\bC u=u$ and therefore $\bC 
f(\abs{\Gamma_u\ute})u=\overline{f}(\abs{\Gamma_u\ute})u$ (by a standard approximation from 
polynomials $f$). Using this, we can rewrite \eqref{e: Gamma_u bC=} as
$$
\Gamma_u\ute \overline{f}(\abs{\Gamma_u\ute})u
=\abs{\Gamma_u\ute}\overline{\Psi}_u(\abs{\Gamma_u\ute})\overline{f}(\abs{\Gamma_u\ute})u, \quad 
f\in L^2(\rho),
$$
so
$$
\Gamma_u\ute=\abs{\Gamma_u\ute}\overline{\Psi}_u(\abs{\Gamma_u\ute}) = 
\overline{\Psi}_u(\abs{\Gamma_u\ute}) \abs{\Gamma_u\ute} . 
$$
The last identity gives the polar decomposition of $\Gamma_u\ute$, and since it is self-adjoint, 
the operator $\Psi_u(\abs{\Gamma_u\ute})$ is a self-adjoint unitary operator, so $\Psi_u$ takes 
values $\pm1$. If we assign $\Psi_u(0):=0$, we get the polar decomposition 
\[
\Gamma_u = \abs{\Gamma_u}  \overline{\Psi}_u(|\Gamma_u|)= \overline{\Psi}_u(|\Gamma_u|) \abs{\Gamma_u}  , 
\]
where $\Psi_u(|\Gamma_u|)$ is a self-adjoint partial isometry, $\Ker \Psi_u(|\Gamma_u|) =\Ker\Gamma_u$.

Similar argument can be applied to $\wt\Psi_u$. If $\Ker \wt\Gamma_u\ute=\{0\}$, the reasoning is 
exactly the same; if $\Ker \wt\Gamma_u\ute\ne\{0\}$ (which may happen), a slight modification is 
needed. Namely, we need first to consider the polar decomposition of $\wt\Gamma_u\big|_{\overline{ 
\Ran} \wt \Gamma_u}$. Noticing that $\wt\rho_0$ is the spectral measure of the operator 
$|\wt\Gamma_u|\,\big|_{\overline{\Ran}\wt \Gamma_u}$ with respect to the vector $\wt u:= 
P_{\overline{ \Ran}\wt\Gamma_u}u$, we the can write, assigning $\wt\Psi_u(0):=0$, that 
\begin{align*}
\wt\Gamma_u  = \abs{\wt\Gamma_u}  \wt\Psi_u(|\wt\Gamma_u|) = \wt\Psi_u(|\wt\Gamma_u|) \abs{\wt\Gamma_u}. 
\end{align*}
So $\wt\Psi_u$ takes values $\pm1$ and the polar decomposition for $\wt\Gamma_u$ has the required form \eqref{pd}. 

Finally, assume that both $\Psi$ and $\wt \Psi$ are real-valued, and let us prove that the 
Fourier coefficients $\wh u_m$ are real for all $m$. We use formula \eqref{e: u_k}. 
Denote $A=\calM^{-1}\Psi(\calM)$ and $B=\wt\Psi(\wt \calM)\wt \calM$. By our assumptions, both 
$A$ and $B$  are self-adjoint, $A$ may be unbounded, but  $BA$ is 
bounded (extends to a bounded operator from a dense set). The operator
$$
(\Sigma^*)^m\wt\Psi(\wt \calM)\wt \calM=(BA)^mB
$$
is self-adjoint for all $m\geq0$. 
Since $\1=\Psi(\calM)\calM q$, we can write 
$$
\wh u_m=\left\langle(\Sigma^*)^{m-1}\wt\Psi(\wt \calM)\wt \calM q,q\right\rangle_{L^2(\rho)}
=
\left\langle(BA)^{m-1}Bq,q\right\rangle_{L^2(\rho)}, 
$$
and since $(BA)^{m-1} B$ is self-adjoint, $\wh u_m$  is real for all $m\geq1$. 
The proof of Theorems~\ref{thm.sa1} and \ref{thm.sa2} is complete. 
\qed

\section{Operator theoretic background}\label{sec.e}

In the following sections, we will use some more specialised operator theoretic material, related 
mainly to the functional model for contractions. In this section, we collect without proof the corresponding background facts. 
\subsection{Douglas' Lemma}
\label{s:DLemma}
We will need the following simple fact. 
\begin{lemma}
\label{l: DLemma}
Let $A$, $B$ be operators in a Hilbert space such that $\Ker A =\Ker A^* =\{0\}$ and 
\begin{align*}
B^*B \le A^*A. 
\end{align*}
Then the operator $BA^{-1}$, defined  on a dense set $\Ran A$ extends to a contraction $T$.
The adjoint $T^*$ is given by the formula $T^* = (A^*)^{-1} B^*$; note that the boundedness of $T$ 
implies that $\Ran B^*\subset \Dom (A^*)^{-1}$, so the above expression is defined on the whole 
space.  
\end{lemma}

In this paper we will often apply this lemma to self-adjoint operators $\calM$, $\wt\calM$, 
$\Ker\calM=\{0\}$, such that $\wt\calM^2\le \calM^2$, to define contractions $\wt\calM\calM^{-1}$, 
$\calM^{-1}\wt\calM$.

\subsection{Inner functions, model spaces and the compressed shift}\label{sec.e1}
A non-constant function $\theta\in H^2(\bbT)$ is called \emph{inner}, if 
$\abs{\theta}=1$ a.e. on the unit circle. For an inner function $\theta$, the \emph{model 
space} $K_\theta$ is the subspace of $H^2(\bbT)$, defined by 
$$
K_\theta=H^2(\bbT)\cap (\theta H^2(\bbT))^\perp.
$$
We refer to \cite{Nikolski,GMR} for background on model spaces.

\emph{Beurling's theorem} states that if a non-trivial subspace $K\subsetneq H^2(\bbT)$ is 
invariant for the backward shift $S^*$, then $K=K_\theta$ for some inner $\theta$.  

Let $\theta$ be an inner function and let $P_\theta$ be the orthogonal projection onto $K_\theta$ 
in $H^2(\bbT)$. The operator $S_\theta =P_\theta S$ on $K_\theta$ is called the \emph{compressed 
shift}. Since $K_\theta$ is an invariant subspace for $S^*$, we have $S_\theta^*f=S^*f$ for $f\in 
K_\theta$. 
It is not difficult to compute that 
\begin{align}
I-S_\theta S_\theta^*=\jap{\cdot,P_\theta\1}P_\theta\1, 
\quad
I-S_\theta^* S_\theta=\jap{\cdot,S^*\theta}S^*\theta, 
\label{e1a}
\end{align}
where $P_\theta\1=1-\overline{\theta(0)}\theta$.

\subsection{Contractions in a Hilbert space}
Let $T$ be a contraction in a Hilbert space. The defect spaces of $T$ and $T^*$ are defined as 
\begin{align*}
\calD\ci{T}:=\overline{\Ran} (I-T^*T), \qquad \calD\ci{T^*}:=\overline{\Ran} (I-TT^*), 
\end{align*}
and the \emph{defect indices} of $T$ is the 
(ordered) pair of numbers
$$
(\partial_T,\partial_{T^*}), \qquad
\partial_T=\dim\calD\ci{T}, \quad \partial_{T^*}=\dim\calD\ci{T^*}.
$$
In particular, the shift operator $S$ has the defect indices $(0,1)$ and  the compressed shift 
$S_\theta$ (for any inner $\theta$) has the defect indices $(1,1)$, see \eqref{e1a}.

A contraction $T$ is called \emph{completely non-unitary} (c.n.u.), if $T$ is not unitary on any of 
its invariant subspaces. 
The following result is known as Langer's lemma (see e.g. \cite[Lemma 1.2.6]{Nikolski2}).

\begin{lemma}\label{lma.b2}
	Let $T$ be a contraction in a Hilbert space $X$. 
	Then $X$ can be represented as an orthogonal sum 
	$X=X_{\rm u}\oplus X_{\rm cnu}$, such that 
	$$
	T=\begin{pmatrix} 
	T_{\rm u} & 0
	\\
	0 & T_{\rm cnu}
	\end{pmatrix}
	\quad \text{ in $X_{\rm u}\oplus X_{\rm cnu}$,}
	$$
	where $T_{\rm u}$ is unitary and $T_{\rm cnu}$ is completely non-unitary. 
\end{lemma}

\subsection{Contractions with defect indices $(0,1)$ and $(1,1)$}

\begin{theorem}\label{thm.b3a}
Let $T$ be a c.n.u. contraction with defect indices $(0,1)$.  
Then $T$ is unitarily equivalent to the forward shift operator $S$. 
In this case $\Re T$ has a purely a.c. spectrum $[-1,1]$ of multiplicity one. 
\end{theorem}
The first part follows from  the Kolmogorov--Wold decomposition, see \cite[Theorem 
I.1.1]{SzNF2010}. For the second part, we note that the matrix of $2\Re S$ in the standard basis in 
$H^2(\bbT)$ is 
the Jacobi matrix
$$
2\Re S=
\begin{pmatrix}
0 & 1 & 0 & 0 & \cdots
\\
1 & 0 & 1 & 0 & \cdots
\\
0 & 1 & 0 & 1 & \cdots
\\
0 & 0 & 1 & 0 & \cdots
\\
\vdots & \vdots & \vdots & \vdots & \ddots
\end{pmatrix}
$$
and it is well known that the spectrum of this matrix is purely a.c., coincides with the interval
$[-2,2]$ and has multiplicity one (see e.g. \cite[Section 1.1.3]{Teschl}). 

The following statement will be crucial in  our construction. 

\begin{theorem}\label{thm2}
Let $T$ be a c.n.u. contraction with defect indices 
$(1,1)$. 
Then the following statements are equivalent: 
\begin{enumerate}[\rm (i)]
	\item
	${T^*}^n\to0$ strongly as $n\to\infty$;
	\item
	$T^n\to0$ strongly as $n\to\infty$; 
	\item
	The operator $\Re T$ has a purely singular spectrum; 
	\item
	The operator $T$ is unitarily equivalent to the compressed shift operator $S_\theta$
	for some inner function $\theta$. 
\end{enumerate}
\end{theorem}
We discuss the proof in Appendix~\ref{app.c}.

\subsection{Dilations of contractions}
Let $T$ be a contraction on a Hilbert space $X$. Further, let $Y$ be another Hilbert space such that $X$ is a subspace of $Y$, let $P_X$ be the orthogonal projection onto $X$ in $Y$ and let $V$ be a bounded operator in $Y$. Then $V$ is called a \emph{dilation} of $T$, if for any $n\geq0$ we have
\begin{equation}
T^nf=P_X V^n f, \quad \forall f\in X.
\label{eq.dilation}
\end{equation}

\begin{theorem}\cite[Theorem II.6.4]{SzNF2010}
\label{thm.dilation}
Let $T$ be a c.n.u. contraction. Then there exists a dilation $V$ of $T$ such that $V$ is a unitary operator with a purely a.c. spectrum. 
\end{theorem}
In fact, any \emph{minimal} unitary dilation (this means that the span of $V^n X$ for $n\geq0$ is dense in $Y$) of a c.n.u. contraction has a purely a.c. spectrum; see \cite{SzNF2010} for details.

\subsection{Trace class perturbations}\label{sec.e2}

\begin{theorem}[Kato-Rosenblum]\label{thm.KR}
	Let $A$ and $B$ be self-adjoint (or unitary) operators in a Hilbert space $X$ such that the 
	difference $A-B$ 
	is trace class. Then the absolutely continuous parts of $A$ and $B$ are unitarily equivalent.
\end{theorem}
The following generalisation of the Kato-Rosenblum theorem was found by Ismagilov in 
\cite{Ismagilov}; see also \cite{HowlandKato,Suslov} for different proofs. 
\begin{theorem}[Ismagilov]
Let $A$ and $B$ be bounded self-adjoint operators such that $AB$ is trace class. 
Then the a.c. parts of the operators $A+B$ and $A\oplus B$ are unitarily equivalent.
\end{theorem}

We will also need the following result on trace class perturbations, due to M.~G.~Krein \cite{Krein}. 
(Much more precise results in terms of the class of $f$ are now 
available, see e.g. \cite{Peller2}).

\begin{theorem}\label{thm.krein}
Let $A$ and $B$ be bounded self-adjoint operators in a Hilbert space $X$ such that the	difference $A-B$ is trace class. Let $f$ be a differentiable function on $\bbR$ such that the derivative $f'$ is a Fourier transform of a finite complex-valued measure on $\bbR$.  Then $f(A)-f(B)$ is also trace class. 
\end{theorem}

\subsection{Spectral measures}
We recall that for a unitary operator in a Hilbert space, its spectral measure is a projection-valued measure on $\bbT$ and for a self-adjoint operator, its spectral measure is a projection-valued measure on $\bbR$. Furthermore, if $U$ is unitary and $\Re U=(U+U^*)/2$, then the spectral measure of $\Re U$ on $\bbR$ is the pushforward of the spectral measure of $U$ on $\bbT$ by the map $z\mapsto (z+\overline{z})/2$. From here we obtain the following simple conclusion, which we will use throughout the paper. 
\begin{proposition}\label{prp.realpart}
The spectrum of a unitary operator $U$ is purely a.c. (resp. purely singular) if and only if the spectrum of the self-adjoint operator $\Re U$ is purely a.c. (resp. purely singular). 
\end{proposition}

\section{Reduction to asymptotic stability: proof of Theorem~\ref{thm.a5}}\label{sec.dd1}

\subsection{The ``only if'' part}
We use the model $(\calH, \wt\calH, \Sigma^*)$ of Section~\ref{sec.model}.
If the triple $(\rho,\Psi,\wt\Psi)$ is the spectral datum for some Hankel operator $H_u$, then by 
the definition \eqref{e: model Sigma*}, the operator $\Sigma^*$ is unitarily equivalent to the restriction of the 
backward shift $S^*$ to the $S^*$-invariant subspace $\overline{\Ran} H_u$ (this 
subspace may coincide with the whole space $H^2(\bbT)$). 

The operator $S^*$ is asymptotically stable, and so its restriction to any invariant subspace is also asymptotically stable. We conclude that $\Sigma^*$ is asymptotically stable. 

\medskip

In the rest of this section, we prove the ``if'' part; this requires several steps. 
Throughout the proof, we use the model of Section~\ref{sec.absmodel}.

\subsection{The ``if'' part: checking the commutation relations}
Let us show that the model operators $\calH$, $\wt\calH$, $\Sigma^*$ (defined in  \eqref{w3}, \eqref{w4}, \eqref{w7})
satisfy the relations 
\begin{align}
\label{e: Sigma* 02}
\wt \calH = \Sigma^* \calH =\calH \Sigma .
\end{align}
Using the fact that $\Psi$ is unimodular, we have
\begin{align*}
\Sigma^*\calH f &= \wt\Psi(\wt\calM)\wt\calM\calM^{-1}\Psi(\calM) \overline{\Psi}(\calM) \calM \overline{f} \\
&= \wt\Psi(\wt\calM)\wt\calM \overline{f} = \wt\calH f.
\end{align*}
For any bounded Borel function $\Phi$ we have 
$$
\overline{(\Phi(\calM)f)}=\overline{\Phi}(\calM)\overline{f}
\quad\text{ and }\quad
\overline{(\Phi(\wt\calM)f)}=\overline{\Phi}(\wt\calM)\overline{f}
$$
and therefore
$$
\overline{(\Sigma f)}=\Psi(\calM)\calM^{-1}\wt\calM\wt\Psi(\wt\calM)\overline{f}.
$$
It follows that
$$
\calH\Sigma f=\overline{\Psi}(\calM)\calM\overline{\Sigma f}=\wt\calM\wt\Psi(\calM)\overline{f}=\wt\calH f. 
$$
We have checked \eqref{e: Sigma* 02}. 

\subsection{The ``if'' part: setting up the unitary equivalence}

Define the operator $\calU : L^2(\rho)\to H^2(\bbT)$ as 
\begin{equation}
\calU f (z) := \sum_{k=0}^\infty \jap{ (\Sigma^*)^k f, q}z^k =  \sum_{k=0}^\infty \jap{f,\Sigma^k 
q}z^k, \qquad z\in \bbT\ ,
\label{defU}
\end{equation}
where $q(s)=\overline{\Psi(s)}/s$. 
Below we check that $\calU$ is an isometry.

It follows from the definition \eqref{w7} of $\Sigma^*$ that 
\begin{align}
\notag
\Sigma\Sigma^* & =  \Psi(\calM)^* \calM^{-1} \wt\calM^2 \calM^{-1}\Psi(\calM) \\
\notag
& = \Psi(\calM)^* \calM^{-1} \left( \calM^2 - \jap{\fdot, \1} \1  \right)\calM^{-1} \Psi(\calM) \\
\label{c13}
& = I - \jap{\fdot, q}q.
\end{align}
From here it follows that that 
\begin{align*}
\norm{f}^2-\norm{\Sigma^*f}^2=\abs{\jap{f,q}}^2.
\end{align*}
Applying this identity to  $(\Sigma^*)^k f$ and summing over $k$ from $0$ to $n-1$ we get that 
 for any $f\in L^2(\rho)$ and any $n\in\bbN$, 
\begin{align*}
\norm{f}^2- \norm{(\Sigma^*)^n f}^2 = \sum_{k=0}^{n-1}\abs{\jap{f,\Sigma^k q}}^2.
\end{align*}
Here comes the crucial point in the proof: 
\emph{by the asymptotic stability of $\Sigma^*$}, we have that $\norm{(\Sigma^*)^n f}^2\to 0$ as 
$n\to\infty$, and so 
\[
\sum_{m=0}^{\infty}\abs{\jap{f,\Sigma^mq}}^2=\norm{f}^2 ,  
\]
i.e.~the map $\calU : L^2(\rho)\to H^2(\bbT)$, defined in \eqref{defU}, is an isometry.

\subsection{The ``if'' part: defining the Hankel operator}

Define the operators $A$ and $\wt A$ on $H^2(\bbT)$ by 
\begin{align*}
A:= \calU \calH \calU^*, \qquad \wt A:= \calU \wt\calH \calU^* .
\end{align*}
We would like to check that $A$ and $\wt A$ are Hankel operators. First we show that $\calU$ intertwines $S^*$ and $\Sigma^*$. 
By the definition \eqref{defU} of the map $\calU$, we have 
\begin{align*}
\calU\Sigma^* f(z)
=\sum_{k=0}^\infty \jap{\Sigma^*f,\Sigma^k q}z^k
=\sum_{k=0}^\infty \jap{f,\Sigma^{k+1}q}z^k
= S^*\calU f(z), 
\end{align*}
and so we find that 
\begin{align}
\calU\Sigma^* & =S^*\calU
\label{c6}
\intertext{and by taking adjoints}
\label{e: U*S}
\Sigma \calU^*& =\calU^* S \ .
\end{align}
Note that \eqref{c6} implies that $\Ran \calU$ is a $S^*$-invariant subspace of $H^2(\bbT)$.

Using \eqref{e: U*S} and \eqref{e: Sigma* 02},  we find
\begin{align*}
A S = \calU \calH \calU^* S = \calU \calH \Sigma \calU^* = \calU \wt\calH \calU^* = \wt A.
\end{align*}
Similarly, 
\begin{align*}
S^* A = S^* \calU \calH \calU^* = \calU \Sigma^* \calH \calU^* = \calU \wt\calH \calU^* =\wt A .
\end{align*}
Therefore $A$ satisfies the commutation relation 
\begin{align*}
AS=S^*A =\wt A, 
\end{align*}
and so $A$ is a Hankel operator. Setting $u:=Az^0$, we can write $A=H_u$ and 
then $\wt A = H_u S=\wt H_u$. It remains to prove that $u\in\BMOA_\simp(\bbT)$ and that the spectral datum of $u$ coincides with the abstract spectral datum $(\rho,\Psi,\wt\Psi)$.

\subsection{The ``if'' part: concluding the proof}
Denote by $U$ the operator $\calU$ with the target space restricted to $\Ran \calU$, so $U$ is a 
unitary operator. Here we use the same notation as for the map \eqref{e: U spectral Thm 01};  
as we shall soon see, this is indeed the same map in disguise. 

Since $\overline{\Ran}\calH=L^2(\rho)$, 
from the definition $H_u=\calU \calH \calU^*$ we find that 
\begin{align*}
\overline{\Ran} H_u=\Ran \calU.
\end{align*}
Thus, in our new notation we find
\begin{equation}
H_u\ute=U\calH U^*, \quad \wt H_u\ute=U\wt\calH U^*. 
\label{eq.*e1}
\end{equation}
Let us check that $u\in\BMOA_\simp(\bbT)$. By the definition of $H_u$, it is a bounded operator and therefore $u\in\BMOA(\bbT)$. Next, from \eqref{eq.*e1} we find 
$$
(H_u\ute)^2=U\calH^2 U^*=U\calM^2 U^*, 
\quad
(\wt H_u\ute)^2=U\wt \calH^2 U^*=U\wt\calM^2 U^*;
$$
recall that here $\calM$ is the multiplication by the independent variable in $L^2(\rho)$ and $\wt\calM$ is defined by \eqref{defwtcalM}.
It is obvious that $\calM^2$ has a simple spectrum with the cyclic element $\1$. 
By Lemma~\ref{l: cyclicity}(i), the same is true for $\wt\calM^2$. Thus, the simplicity of spectrum condition \eqref{a4} is satisfied and so $u\in\BMOA_\simp(\bbT)$. 

Our next step is to check the identity $U^*u=\1$. We first note that by the definition \eqref{defU} of $\calU$, for any $f\in L^2(\rho)$ we have
$$
\jap{\calU f,z^0}_{H^2}=\jap{f,q}_{L^2(\rho)},
$$
and therefore $\calU^*z^0=q$. Further, we have
$$
u=H_u z^0=\calU\calH\calU^*z^0=\calU\calH q. 
$$
Recalling formula \eqref{w3} for the action of $\calH$, we find that 
$$
\calH q(s)=s\overline{\Psi}(s)\overline{q(s)}=s\overline{\Psi}(s)\overline{\overline{\Psi}(s)/s}=1,
$$
and so we conclude that $u=\calU\1=U\1$ and therefore $U^*u=\1$. 

Finally, we check that the map $U$ coincides with the map defined by \eqref{e: U spectral Thm 01}. For $f\in L^2(\rho)$, we find
$$
f(\abs{H_u\ute})u=Uf(\abs{\calH})U^*u=Uf(\calM)\1=Uf, 
$$
as required. 

We conclude that for the Hankel operator $H_u$ and the map $U$, satisfying \eqref{e: U spectral Thm 01}, we have the identities \eqref{eq.*e1}, where $\calH$ and $\wt\calH$ correspond to our abstract spectral datum $\Lambda=(\rho,\Psi,\wt\Psi)$. This means that the abstract spectral datum $\Lambda$ coincides with the spectral datum $\Lambda(u)$.

\section{Initial results about asymptotic stability} \label{s: asy 1}

In this section, as a warm-up, we present some easy initial results on asymptotic stability. 
In what follows, $(\rho,\Psi,\wt\Psi)$ is an abstract spectral datum. We recall that this means that $\rho$ is a finite Borel measure with a bounded support on $(0,\infty)$, satisfying the normalisation condition \eqref{rholeq1}, and $\Psi\in L^\infty(\rho)$ and $\wt\Psi\in L^\infty(\wt\rho_0)$ are unimodular complex-valued functions. 

\subsection{The operator $\Sigma_0^*$ is asymptotically stable }
Let the operators $\calM$ and $\wt\calM$  in $L^2(\rho)$ be as defined in Section~\ref{s: ASD}. 
Recall that the operator $\Sigma_0^* $ in $L^2(\rho)$  was defined by $\Sigma_0^*:= \wt\calM 
\calM^{-1}$. Our purpose here is to prove
\begin{theorem}\label{lma.*f1}
The operator $\Sigma_0^*$ is asymptotically stable.
\end{theorem}

By Theorem~\ref{thm.a5}, this implies that any spectral datum of the form $(\rho,\1,\1)$ is in $\Lambda(\BMOA_\simp)$; this was one of the main results of \cite{GP}. 

In order to prove Theorem~\ref{lma.*f1}, we consider the \emph{symmetrization}  $\fS_0$  of $\Sigma_0^*$
\begin{align*}
\fS_0 := \calM^{-1/2} \Sigma_0^*\calM^{1/2} = \calM^{-1/2} \wt\calM \calM^{-1/2}.  
\end{align*}
Note that $\|\fS_0\|\le 1$. Indeed, from $0\le \wt\calM^2 \le \calM^2$ by the Heinz inequality 
we find $\wt\calM^{1/2} \le \calM^{1/2}$, and therefore by Douglas' Lemma 
(Lemma~\ref{l: DLemma}) the operator $Q:= \wt\calM^{1/2}\calM^{-1/2}$ extends from a dense set to a 
contraction, and its adjoint is given by $Q^*=\calM^{-1/2} \wt \calM^{1/2}$. Thus 
\begin{align}\label{e: fS=QQ*}
\fS_0 = Q^* Q, 
\end{align}
so $\fS_0$ is a contraction. 
\begin{lemma}\label{lma.*f2}
The operator $\fS_0$ is asymptotically stable.
\end{lemma}
\begin{proof}
By \eqref{e: fS=QQ*}, the operator $\fS_0$ is self-adjoint and 
$0\le \fS_0 \le I$. So in order to prove the asymptotic stability of $\fS_0$, it is sufficient to show that $1$ is not an eigenvalue of  
$\fS_0$. Let us prove this. 
We have 
\begin{align*}
\wt\calM =\calM^{1/2}\fS_0 \calM^{1/2} , 
\end{align*}
and therefore
\begin{align*}
\wt\calM^2 = \calM^{1/2}\fS_0 \calM \fS_0 \calM^{1/2}\ .
\end{align*}
On the other hand, 
\begin{align*}
\wt\calM^2 =\calM^2 -\jap{\fdot,\1}\1 = \calM^{1/2} \bigl(\calM -\jap{\fdot,b}b \bigr)   \,  
\calM^{1/2}, 
\end{align*}
where $b=\calM^{-1/2} \1$, i.e.~$b(s) = s^{-1/2}$. 

Comparing these two representations for $\wt\calM^2$ and using the fact that $\Ker\calM^{1/2} 
=\{0\}$ we find 
\begin{align}\label{c1-1}
\fS_0 \calM \fS_0 = \calM -\jap{\fdot,b}b.
\end{align}
Suppose $f\in\Ker(\fS_0-I)$, i.e.~$\fS_0 f =f$. 
Evaluating the quadratic form of the last identity on $f$, we find 
\begin{align*}
\jap{\calM f,f}=\jap{\calM f,f}-\abs{\jap{f,b}}^2,
\end{align*}
and so $f\perp b$. 
Substituting $f\perp b$ back into \eqref{c1-1}, we get
\begin{align*}
\fS_0 \calM f = \calM f , 
\end{align*}
and so $ \calM f\in\Ker (\fS_0-I)$. 

Thus, $\Ker (\fS_0-I)$ is an invariant subspace of $\calM$ which is orthogonal to $b$. 
Since $b$ is a cyclic element for $\calM$, it follows that $\Ker (\fS_0-I)=\{0\}$. 
\end{proof}

\begin{corollary}\label{c: Q strict}
The operator $Q = \wt\calM^{1/2}\calM^{-1/2}$ is a strict contraction, i.e.
\begin{align*}
\|Q x\|<\|x\| \qquad \forall x \ne 0. 
\end{align*}
\end{corollary}
\begin{proof}
By construction, $Q$ is a contraction, $\norm{Qx}\leq \norm{x}$ for all $x$. 
Assume that $\|Qx\|=\|x\|$ for some $x\ne0$. Since $\fS_0=Q^*Q$, we conclude that $\jap{\fS_0 x,x} = 
\|Qx\|^2=\|x\|^2$. But $\|\fS_0\|\le 1$, so $\fS_0 x =x$, which contradicts  Lemma~\ref{lma.*f2}.
\end{proof}

\begin{lemma}\label{lma.*f3}
Let bounded operators $A$, $B$, $K$ satisfy
\begin{align}
\label{e: KA=BK}
KA = BK,  
\end{align}
and let $\|B\|\le1$. 
Assume that $\Ran K$ is dense, and that the operator $A$ is asymptotically stable. 
Then $B$ is also asymptotically stable. 
\end{lemma}
\begin{proof}
Iterating \eqref{e: KA=BK} we get that 
\begin{align*}
KA^n = B^nK, \qquad n\in\bbN . 
\end{align*}
Since $A$ is asymptotically stable, we  see that for all $ x\in \Ran K$
\begin{align}\label{e: Bnx to 0}
\| B^nx\|\to 0 \ \text{as }n\to\infty .  
\end{align}
But $\|B^n\|\le1$, so operators $B^n$ are uniformly bounded. Since $\Ran K$ is dense, the 
$\eps/3$-Theorem says that \eqref{e: Bnx to 0} holds for all $x$, i.e. $B$ is asymptotically 
stable. 
\end{proof}

\begin{proof}[Proof of Theorem~\ref{lma.*f1}]
From the definition of $\fS_0$ we see that 
$$
\calM^{1/2} \fS_0 = \Sigma_0^* \calM^{1/2},
$$
and $\Ran\calM^{1/2}$ is dense in $L^2(\rho)$. 
Now we apply Lemma~\ref{lma.*f3} with $K=\calM^{1/2}$, $A=\fS_0$ and $B=\Sigma_0^*$. 
\end{proof}

\subsection{Self-adjoint Hankel operators and  positivity}
\label{s: positivity}

In this subsection we discuss the self-adjoint case, when both operators $\Gamma$ and 
$\wt\Gamma = \Gamma S $ are self-adjoint. According to Theorem~\ref{thm.sa1}, in terms of the spectral data, this corresponds to the case when  both unimodular functions $\Psi$ and 
$\wt\Psi$ are real-valued. 

\begin{theorem}\label{t: G positive}
Let $\Lambda=(\rho,\Psi,\wt\Psi)$ be an abstract spectral datum such that $\Psi$ and $\wt\Psi$ are real-valued and one of them is identically equal to $1$. Then $\Sigma^*$ is asymptotically stable, i.e. $\Lambda\in\Lambda(\BMOA_\simp)$. 
\end{theorem}

This theorem gives us a complete description of the spectral data in the case of self-adjoint Hankel 
operators, when one of the operators $\Gamma$, $\wt\Gamma$ is non-negative. 

\begin{proof}[Proof of Theorem~\ref{t: G positive}]
Let us introduce the \emph{symmetrization} $\fS^*$ of $\Sigma^*$, 
\begin{align*}
\fS^* := \calM^{-1/2} \Sigma^* \calM^{1/2} = Q^* \wt\Psi(\wt\calM) Q \Psi(\calM), 
\end{align*}
where $Q= \wt\calM^{1/2}\calM^{-1/2}$ is as in Corollary~\ref{c: Q strict}.
Since $\Psi$ and $\wt\Psi$ are real-valued, the operators $\Psi(\calM)$ and $\wt\Psi(\wt\calM)$ are 
self-adjoint. Let us prove that $\fS^*$ is asymptotically stable. 

If $\Psi\equiv1$, we have $\fS^* =  Q^* \wt\Psi(\wt\calM) Q$, so $\fS^*$ is self-adjoint. The fact 
that $Q$ is a strict contraction (see Corollary~\ref{c: Q strict}) implies that $\pm1$ are not 
eigenvalues of $\fS^*$, so $\fS^*$ is asymptotically stable. 

If $\wt\Psi\equiv1$, we get that $\fS^*=Q^*Q\Psi(\calM)$. This operator is not self-adjoint, but 
\begin{align}\label{e: fS* 02}
\left(\fS^*\right)^n = Q^* (Q\Psi(\calM) Q^*)^{n-1} Q\Psi(Q), 
\end{align}
and the operator $Q\Psi(\calM) Q^*$ is self-adjoint. Since $Q^*$ is a strict 
contraction, the points $\pm1$ are not the eigenvalues of $Q\Psi(\calM) Q^*$, so $Q\Psi(\calM) Q^*$ 
is  asymptotically stable. Identity \eqref{e: fS* 02} together with Lemma~\ref{lma.*f3} shows that $\fS^*$ is asymptotically stable as well. 

Finally, we have 
$$
\calM^{1/2} \fS^* = \Sigma^* \calM^{1/2},
$$ 
so by Lemma~\ref{lma.*f3}
with $K=\calM^{1/2}$ the asymptotic stability of $\fS^*$ implies the asymptotic stability of
$\Sigma^*$.
\end{proof}

\section{Asymptotic stability and singular spectrum}\label{s: asy and sing spectr}
In this section we present one of our key result which related the asymptotic stability of $\Sigma^*$ to its spectral properties. As in the previous section, below $(\rho,\Psi,\wt\Psi)$ is an abstract spectral datum, and $\Sigma^*$ is the operator in $L^2(\rho)$ defined  in \eqref{w7}.

\subsection{Defect indices of $\Sigma^*$} 
In what follows, the consideration of $\Sigma^*$ will proceed in two slightly different ways depending on the defect indices of $\Sigma^*$. In the following lemma we describe these two possible cases. 

\begin{lemma}\label{lma.e1} Let $\Lambda= (\rho, \Psi, \wt\Psi)$ be an abstract spectral datum, and 
let $\Sigma^*$ be the operator \eqref{w7} constructed from it. 
\begin{enumerate}
\item
If 
\begin{align}
\int_0^\infty \frac{d\rho(s)}{s^2}=1
\quad \text{ and }\quad
\int_0^\infty \frac{d\rho(s)}{s^4}=\infty, 
\label{ker0a}
\end{align}
then the defect indices of $\Sigma^*$  are $(1,0)$, so $\Sigma $ is an isometry. 
\item
If \eqref{ker0a} fails, i.e.\ if we have either 
\begin{align}
\int_0^\infty \frac{d\rho(s)}{s^2}=1 \quad \text{ and }\quad 
\int_0^\infty \frac{d\rho(s)}{s^4}<\infty,
\label{a2}
\end{align}
or 
\begin{align}
\int_0^\infty \frac{d\rho(s)}{s^2}<1, 
\label{a3}
\end{align}
then the defect indices of $\Sigma^*$ are $(1,1)$. 
\end{enumerate}
\end{lemma}
\begin{proof}
We have $\Sigma^* = \wt\Psi(\wt\calM) \Sigma_0^* \Psi(\calM)$, where $\Sigma_0^*:=  \wt\calM 
\calM^{-1}$. The operators $\Psi(\calM)$ and $ \wt\Psi(\wt\calM)$ are unitary, and so the defect 
indices of $\Sigma^*$ and $\Sigma_0^*$ coincide. Thus, it suffices to consider the defect indices of $\Sigma_0^*$.

By Theorem~\ref{lma.*f1}, the operator $\Sigma_0^*$ is asymptotically 
stable. By Theorem~\ref{thm.a5}, this means that the triple $(\rho, \1, \1)$ is the spectral datum of some Hankel operator $H_u$. 

(i) Suppose \eqref{ker0a} is satisfied. Note that \eqref{ker0a} is identical to \eqref{e: triv ker}, and so by Theorem~\ref{thm.b3aa}, we have $\Ker H_u=\{0\}$, and therefore (see \eqref{e: model Sigma*}) the operator  $\Sigma_0^*$ is unitarily equivalent to the backward shift $S^*$, and so the defect indices of $\Sigma^*$ are $(1,0)$. 

(ii) Suppose \eqref{ker0a} fails. Then again by Theorem~\ref{thm.b3aa}, the kernel of $H_u$ is non-trivial and so $\Sigma_0^*$ is unitarily equivalent to the restriction of $S^*$ to the subspace $\overline{\Ran}H_u$. By Beurling's theorem, this subspace is a model space  $K_\theta:=H^2\ominus\theta H^2$ for some inner function $\theta$ and so $\Sigma_0^*$ is unitarily equivalent to $S_\theta^*$, where $S_\theta$ is the compressed shift on $K_\theta$. It follows (see \eqref{e1a}) that the defect indices of $\Sigma_0^*$ are $(1,1)$. 
\end{proof}

\subsection{Asymptotic stability and singular spectrum.}

Recall that the a.c.\ spectrum of a self-adjoint or unitary operator is said to equal to a Borel 
set $E$ if the a.c.\ part of the spectral measure is mutually absolutely continuous 
with the Lebesgue measure restricted to $E$. 

\begin{theorem}\label{thm.a6}
Let the triple $\Lambda=(\rho, \Psi, \wt\Psi)$ be an abstract spectral datum. 
\begin{enumerate}
\item
Assume that \eqref{ker0a} holds, i.e.\ that $\Sigma^*$ has defect indices $(1,0) $. Then 
$\Sigma^*$ is asymptotically stable iff the a.c. spectrum of 
$\Re\Sigma$ is $[-1,1]$ with multiplicity one. 
\item
Assume that \eqref{ker0a} does not hold, i.e.\ that $\Sigma^*$ has defect indices $(1,1) $. Then 
$\Sigma^*$ is asymptotically stable iff the a.c. part of $\Re\Sigma$ is empty.
\end{enumerate}
\end{theorem}

Before proceeding to the proof, we need a lemma.
This lemma is one of the central points of our argument. 
Below we refer to the unitary and c.n.u. parts of a contraction according to Langer's lemma, see Lemma~\ref{lma.b2}. 

\begin{lemma}\label{lma.e3}
Let $\Sigma^*$ be the operator constructed from an abstract spectral datum. Then
the unitary part of $\Sigma$ is either purely absolutely 
continuous or absent.
\end{lemma}
\begin{proof}
We use the model $(\calH,\wt\calH,\Sigma)$ as described in Section~\ref{sec.absmodel}.
Let us write
$$
L^2(\rho)=X_{\rm sing}\oplus X_{\rm r},
$$
where $X_{\rm sing}$ is the singular subspace of the unitary part of $\Sigma$, and $X_{\rm r}$ is 
the ``remainder'' part, i.e. the sum of the completely non-unitary subspace of $\Sigma$ and the 
absolutely continuous subspace of the unitary part of $\Sigma$. Our aim is to show that $X_{\rm sing}=\{0\}$. 

\emph{Step 1: the spectral measures associated with $\Sigma_{\rm r}$.}
By construction, $\Sigma_{\rm r}$ is an orthogonal sum of a unitary part $\Sigma_{\rm u}$ with the purely a.c. spectrum and a completely non-unitary part $\Sigma_{\rm cnu}$. 

For any $f\in X_{\rm r}$ and any polynomial $\varphi$ of $z$, we have
\begin{equation}
\norm{\varphi(\Sigma_{\rm r})f}^2=\norm{\varphi(\Sigma_{\rm u})f_{\rm u}}^2+\norm{\varphi(\Sigma_{\rm cnu})f_{\rm cnu}}^2\ ,
\label{eq.*g1}
\end{equation}
where $f_{\rm u}$ and $f_{\rm cnu}$ are the projections of $f$ onto the corresponding subspaces. 
We can write 
$$
\norm{\varphi(\Sigma_{\rm u})f_{\rm u}}^2=\int_{\bbT}\abs{\varphi(z)}^2d\mu^{\rm u}_f(z),
$$
where $\mu^{\rm u}_f$ is the spectral measure of $\Sigma_{\rm u}$, associated with the vector $f_{\rm u}$. By construction, this measure is absolutely continuous. 

Now let us consider the second term in the r.h.s. of \eqref{eq.*g1}. Since $\Sigma_{\rm cnu}$ is a c.n.u. contraction, we can consider its minimal unitary dilation $V$, which has a purely a.c. spectrum, see Theorem~\ref{thm.dilation}. Taking linear combinations of \eqref{eq.dilation}, we obtain
$$
\varphi(\Sigma_{\rm cnu})f_{\rm cnu}=P_{\rm cnu}\varphi(V)f_{\rm cnu},
$$
where $P_{\rm cnu}$ is the orthogonal projection onto the c.n.u. subspace of $\Sigma$. This yields
$$
\norm{\varphi(\Sigma_{\rm cnu})f_{\rm cnu}}^2
=
\norm{P_{\rm cnu}\varphi(V)f_{\rm cnu}}^2
\leq
\norm{\varphi(V)f_{\rm cnu}}^2
=
\int_{\bbT}\abs{\varphi(z)}^2d\mu^{\rm cnu}_f(z),
$$
where $\mu^{\rm cnu}_f$ is the spectral measure of $V$ associated with the vector $f_{\rm cnu}$. 
By Theorem \ref{thm.dilation}, the measure $\mu^{\rm cnu}_f$ is purely a.c.

Summarizing, we can write
\begin{equation}
\norm{\varphi(\Sigma_{\rm r})f}^2
\leq 
\int_{\bbT}\abs{\varphi(z)}^2 d\mu_f(z),
\label{eq.*g2}
\end{equation}
where $\mu_f=\mu^{\rm u}_f+\mu^{\rm cnu}_f$ is an absolutely continuous measure on $\bbT$.

\emph{Step 2: a commutation relation.}
We have
$$
\Sigma=
\begin{pmatrix}
\Sigma_{\rm sing} & 0
\\
0 & \Sigma_{\rm r}
\end{pmatrix}, 
\quad
\calH=
\begin{pmatrix}
h_{11} & h_{12}
\\
h_{21} & h_{22}
\end{pmatrix}
$$
with respect to our decomposition $L^2(\rho)=X_{\rm sing}\oplus X_{\rm r}$.
Iterating the commutation relation \eqref{e: Sigma* 02}, we find
$$
{\Sigma^*}^n\calH=\calH\Sigma^n. 
$$
In our orthogonal decomposition, we can write this relation as 
$$
\begin{pmatrix}
\Sigma_{\rm sing}^{*n} & 0
\\
0 & \Sigma_{\rm r}^{*n}
\end{pmatrix}
\begin{pmatrix}
h_{11} & h_{12}
\\
h_{21} & h_{22}
\end{pmatrix}
=
\begin{pmatrix}
h_{11} & h_{12}
\\
h_{21} & h_{22}
\end{pmatrix}
\begin{pmatrix}
\Sigma_{\rm sing}^n & 0
\\
0 & \Sigma_{\rm r}^n
\end{pmatrix}.
$$
If we write this as a system of four equations, one of them will read
$$
\Sigma_{\rm sing}^{*n} h_{12}=h_{12}\Sigma_{\rm r}^n. 
$$
Taking a linear combination of these equations and taking into account the anti-linearity of 
$h_{12}$, we obtain
\begin{align}\label{e: Comm Rel 03}
\varphi(\Sigma_{\rm sing})^* h_{12}=h_{12}\varphi(\Sigma_{\rm r})
\end{align}
for any \emph{analytic} polynomial $\f(z) = \sum_{k=0}^n a_k z^k$.

\emph{Step 3: $\calH$ is diagonal.}
Let us choose a sequence $\{\varphi_n\}_{n=1}^\infty$ of analytic polynomials
such that:
\begin{enumerate}
\item
$\norm{\varphi_n}_{H^\infty}\leq1$;
\item
$\varphi_n(z)\to0$ for a.e. $z\in\bbT$ (with respect to the Lebesgue measure);
\item
$\liminf_{n\to\infty}\abs{\varphi_n(z)}\geq c>0$ for a.e $z\in\bbT$ with respect to the 
(singular)  spectral measure of $\Sigma_{\rm sing}$. 
\end{enumerate}
The existence of such polynomials $\f_n$ is given by the following lemma. 
\begin{lemma}
\label{:poly}
Let $\nu$ be a  singular (regular, Borel)  measure on the unit circle $\bbT$. There exists a 
sequence of analytic polynomials $\f_n$, satisfying properties \cond1--\cond3 above. 
\end{lemma}

\begin{proof}
Let $E$ be the set of Lebesgue measure zero ($|E|=0$), supporting $\nu$, i.e.~such that 
$\nu(\bbT\setminus E)=0$. By the regularity of $\nu$ and the Lebesgue measure there exist 
increasing sequences of compacts $K_n\subset E$, $F_n\subset \bbT\setminus E$ such that 
\begin{align*}
  \lim_{n\to\infty} \nu(K_n)=\mu(E), \qquad \lim_{n\to\infty} |\bbT\setminus F_n| = 0. 
\end{align*}
Since $\dist(F_n,K_n)>0$ for all $n$, one can choose continuous functions $f_n:\bbT\to [0,1]$ such that $f_n\big|_{K_n} \equiv 1$, 
$f_n\big|_{F_n} \equiv 0$. 

Using the Weierstrass approximation theorem, let us choose trigonometric polynomials $p_n=\sum_{k=-N_n}^{N_n} a_k z^k$  such that $\|f_n -p_n\|\ci{L^\infty(T)}\leq 2^{-n}$ for all $n\geq1$. 
Then the analytic polynomials $\f_n(z):= z^{N_n} p_n(z)/2$ give the desired sequence. 
\end{proof}

We continue the proof of Lemma~\ref{lma.e3}.
Let us substitute $\varphi_n$ into \eqref{eq.*g2}. Since $\mu_f$ is absolutely continuous, 
by conditions \cond1 and \cond2 and the dominated convergence theorem we find 
$$
\norm{\varphi_n(\Sigma_{\rm r})f}^2
\leq 
\int_{\bbT}\abs{\varphi_n(z)}^2 d\mu_f(z)\to0,\quad n\to\infty,
$$
i.e. $\varphi_n(\Sigma_{\rm r})\to0$ 
strongly. 
On the other hand, condition \cond3 and Fatou's Lemma imply that for any element $f\in X_{\rm 
sing}$ we have 
\[
\liminf_{n\to\infty}\norm{\varphi_n(\Sigma_{\rm sing})^*f}^2 = \liminf_{n\to\infty}\int |\f_n|^2 
d\nu_f \geq \int c^2 d\nu_f = 
c^2\norm{f}^2; 
\]
here $\nu_f$ is the spectral measure of $\Sigma\ti{sing}$ associated with the vector $f$.

Substituting $\f_n$ into \eqref{e: Comm Rel 03} and letting $n\to\infty$ we then conclude that 
$h_{12}=0$. Using the symmetry 
condition 
\[
\jap{\calH f,g}=\jap{\calH g,f}, 
\]
from here it is easy to see that $h_{21}=0$, and so $\calH$ is diagonal in the
orthogonal decomposition $X_{\rm sing}\oplus X_{\rm r}$. 
Thus, $\calH^2=\calM^2$ is also diagonal in this decomposition, and in particular $X_{\rm sing}$ is an invariant subspace for $\calM^2$. 

\emph{Step 4: concluding the proof.}
By \eqref{c13}, we know that $\Sigma$ satisfies
\[
\Sigma\Sigma^* =I-\jap{\fdot,q}q, \qquad q(s)=\overline{\Psi}(s)/s. 
\]
Since $\Sigma\Sigma^*=I\oplus \Sigma_{\rm r}\Sigma_{\rm r}^*$, we conclude that  $q\in X_{\rm r}$. 
On the other hand, $q(s) \ne 0$ $\rho$-a.e. and so $q$ is a cyclic element for $\calM^2$. We find that $X_{\rm sing}$ is an invariant subspace for $\calM^2$, orthogonal to its cyclic 
element $q$. 
Thus, $X_{\rm sing}=\{0\}$. The proof is complete. 
\end{proof}

\begin{proof}[Proof of Theorem~\ref{thm.a6}]
By Langer's lemma (Lemma~\ref{lma.b2}) we have
\begin{align}
\Re \Sigma=\Re \Sigma_{\rm u}\oplus\Re\Sigma_{\rm cnu} \,, 
\label{e2a}
\end{align}
where $\Sigma_{\rm u}$ is unitary and $\Sigma_{\rm cnu}$ is completely non-unitary. 

(i) Suppose \eqref{ker0a} holds and so $\Sigma^*$ has defect indices $(1,0)$. 

First suppose that $\Sigma^*$ is asymptotically stable. Then the unitary part of $\Sigma$ is absent, and so $\Sigma$ is completely non-unitary. Then, by Theorem~\ref{thm.b3a}, $\Sigma$ is unitarily equivalent to the shift operator $S$ and so $\Re \Sigma$ has a purely a.c. spectrum $[-1,1]$ with multiplicity one, as required. 

Conversely, suppose that the a.c. spectrum of $\Re \Sigma$ is $[-1,1]$ with multiplicity one. 
By Lemma~\ref{lma.e3} and Proposition~\ref{prp.realpart}, the spectrum of $\Re \Sigma_{\rm u}$ is purely a.c. Next, 
applying Theorem~\ref{thm.b3a} again, we find that $\Sigma_{\rm cnu}$ is unitarily equivalent to the shift operator $S$ and so $\Re \Sigma_{\rm cnu}$ has a purely a.c. spectrum $[-1,1]$ with multiplicity one. Denoting the unitary equivalence by $\simeq$, we can rewrite \eqref{e2a} for the a.c. parts as
$$
\Re S\simeq \Re \Sigma_{\rm u}\oplus \Re S.
$$
Considering the multiplicity functions of the spectrum on both sides, we see that the term $\Re \Sigma_{\rm u}$ must be absent from this expression. Thus, $\Sigma^*=\Sigma_{\rm cnu}^*\simeq S^*$, and so $\Sigma^*$ is asymptotically stable. 

(ii) Suppose \eqref{ker0a} fails and so $\Sigma$ has defect indices $(1,1)$. 

First suppose that $\Sigma^*$ is asymptotically stable. Then the unitary part of $\Sigma$ is absent, and so $\Sigma$ is completely non-unitary and by Theorem~\ref{thm2} the operator $\Re\Sigma$ has a purely singular spectrum, as required. 

Conversely, suppose that the a.c. spectrum of $\Re \Sigma$ is $[-1,1]$ is absent. By Lemma~\ref{lma.e3}, the spectrum of $\Re \Sigma_{\rm u}$ is purely a.c.; thus, the unitary part of $\Sigma_u$ is absent, and so $\Sigma$ is c.n.u. Applying Theorem~\ref{thm2} again, we find that $\Sigma^*$ is asymptotically stable, as required. 
\end{proof}

\section{Reduction to spectral properties of $\wt\Psi(\wt\calM)\Psi(\calM)$}
\label{sec.8}

In this section we show that under some additional assumptions the operator $\Sigma^*$ is asymptotically stable if and only if the spectrum of the unitary operator $\wt\Psi(\wt\calM) \Psi(\calM)$ is purely singular. While at first glance this condition does not look much simpler than the conditions in Theorem~\ref{thm.a6}, it will allow us to easily  construct examples and counterexamples. 

Below $\Lambda=(\rho, \Psi, \wt\Psi)$ is an abstract spectral datum and $\calM$, $\wt\calM$ are the operators in $L^2(\rho)$ constructed as in Section~\ref{s: ASD}. We recall that the operator $\Sigma_0^* $ in $L^2(\rho)$  was defined by $\Sigma_0^*:= \wt\calM \calM^{-1}$.

\subsection{Reduction is possible if $I-\Sigma_0\in\bS_1$}
We start with the easiest case, when the difference $I-\Sigma_0$ is trace class.

\begin{theorem}
\label{p: sing Psi wtPsi} 
Let the abstract spectral datum $\Lambda=(\rho, \Psi, \wt\Psi)$ be such that $I-\Sigma_0\in\bS_1$. 
Then $\Sigma^*$ is asymptotically stable if and only if the unitary operator 
$$
\wt\Psi(\wt\calM) \Psi(\calM)
$$ 
has a purely singular spectrum. 
\end{theorem}

\begin{remark}
\label{r: not shift}
By Lemma~\ref{lma.e1}, the operator $\Sigma_0^*$ has defect indices either $(1,0)$ or $(1,1)$. By Theorem~\ref{lma.*f1}, the operator $\Sigma_0^*$ is asymptotically stable, and hence it is c.n.u. 
By Theorems~\ref{thm.b3a} and \ref{thm2}, we see that there are two possibilities:
\begin{enumerate}[\rm (i)]
\item
$\Sigma_0^*$ has defect indices $(1,0)$, and then it is unitarily equivalent to $S^*$, where $S$ is the shift operator in $H^2$;
\item
$\Sigma_0^*$ has defect indices $(1,1)$, and then it is unitarily equivalent to $S_\theta^*$, where $S_\theta$
is the compressed shift operator $S_\theta$ in a model space $K_\theta$ for some inner function $\theta$. 
\end{enumerate}
Observe that condition $I-\Sigma_0\in\bS_1$ is incompatible with (i), because $I-S$ is not a trace class operator. So the assumption $I-\Sigma_0\in\bS_1$ necessitates that we have (ii). 
\end{remark}

\begin{proof}[Proof of Theorem~\ref{p: sing Psi wtPsi}]
We have 
\begin{align*}
\Psi(\calM)^* \wt\Psi(\wt\calM)^* - \Sigma
&=
\Psi(\calM)^* \wt\Psi(\wt\calM)^* - \Psi(\calM)^*\Sigma_0 \wt\Psi(\wt\calM)^* \\
&=
\Psi(\calM)^*(I-\Sigma_0)\wt\Psi(\wt\calM)^* \in 
\bS_1 
\end{align*}
and so, taking real parts, 
\[
\Re(\wt\Psi(\wt\calM)\Psi(\calM)) - \Re\Sigma  \in 
\bS_1 .
\]
Applying Proposition~\ref{prp.realpart} and the Kato-Rosenblum Theorem, we find that the spectrum of $\wt\Psi(\wt\calM)\Psi(\calM)$ is purely singular if and only if the spectrum of $\Re\Sigma$ is purely singular. 

Finally, as discussed in Remark~\ref{r: not shift}, the operator $\Sigma_0$ has defect indices $(1,1)$, and so $\Sigma$ has the same defect indices. Thus Theorem~\ref{thm.a6}(ii) applies and so the spectrum of $\Re\Sigma$ is purely singular if and only if $\Sigma^*$ is asymptotically stable.
\end{proof}

\subsection{Sufficient conditions for $I-\Sigma_0\in\bS_1$}
The previous theorem leads to the natural question: how to characterise abstract spectral data which correspond to the case $I-\Sigma_0\in\bS_1$? We give some sufficient conditions that guarantee this inclusion. 
We start with the simplest condition. 

\begin{lemma}
\label{l: I-Sigma trace class}
Let the abstract spectral datum $\Lambda=(\rho, \Psi, \wt\Psi)$ be such that $\supp \rho$ is separated away from $0$; then $I-\Sigma_0\in\bS_1$. 
\end{lemma}

\begin{proof}
The idea is to apply Theorem~\ref{thm.krein} to the operators $\calM^2$ and $\wt\calM^2 = \calM^2 -\jap{\fdot, \1}\1\ge0$ and the function $\f(s)=\sqrt s$. The function $\f$ is not sufficiently smooth to comply with the hypothesis of Theorem~\ref{thm.krein}, but we can modify it so that the resulting function is in $C^\infty_0(\bbR)$. 

Indeed, by assumptions $\sigma(\calM^2)\subset[a,R]$ with some $0<a<R<\infty$; since $\wt\calM^2$ is a rank one perturbation of $\calM^2$, we find that $\sigma(\wt\calM^2)\subset\{\lambda_0\}\cup[a,R]$, with some eigenvalue $\lambda_0\geq0$. It is clear that we can modify $\f$ outside the set $\{\lambda_0\}\cup[a,R]$ such that the resulting function $\wt\f$ is in $C^\infty_0(\bbR)$. 
Thus, Theorem~\ref{thm.krein} applies to $\wt\f$  and we get
$$
\wt\f(\wt\calM^2)-\wt\f(\calM^2)=\f(\wt\calM^2)-\f(\calM^2)=\wt\calM-\calM\in\bS_1.
$$
By assumption, the operator $\calM$ is invertible, so left multiplying $\calM-\wt\calM$ by 
$\calM^{-1}$ and recalling that $\Sigma_0=\calM^{-1}\wt\calM$, we get the conclusion of the lemma. 
\end{proof}

Next, we give a slightly more precise sufficient condition. As discussed in Remark~\ref{r: not shift}, if $I-\Sigma_0\in\bS_1$, then condition \eqref{ker0a} is not satisfied, which means that either \eqref{a2} or \eqref{a3} holds. The following lemma says that under conditions that are slightly stronger  than \eqref{a2} or \eqref{a3},  we have $I-\Sigma_0\in\bS_1$.

\begin{lemma}\label{lma.a9}
Assume that for some $\eps>0$, we have either
\begin{align}
\int_0^\infty \frac{d\rho(s)}{s^2}=1, \quad \int_0^\infty \frac{d\rho(s)}{s^{4+\eps}}<\infty
\label{a5}
\end{align}
or 
\begin{align}
\int_0^\infty \frac{d\rho(s)}{s^2}<1, \quad \int_0^\infty \frac{d\rho(s)}{s^{2+\eps}}<\infty.
\label{a6}
\end{align}
Then $I-\Sigma_0$ is trace class. 
\end{lemma}
 The proof is elementary but a little technical; it is given in Appendix~\ref{sec.f}.

\subsection{Reduction is possible if $\Psi$ and $\wt\Psi$ are H\"{o}lder at $0$}
Finally, we turn to the case when $I-\Sigma_0$ is not necessarily trace class. 
We give a more precise condition, whose proof is based on the application of Ismagilov's Theorem. 
\begin{theorem}
\label{thm.a7}
Let the abstract spectral datum $\Lambda=(\rho, \Psi, \wt\Psi)$ be such that the limits $\Psi(0_+)$ and $\wt\Psi(0_+)$ exist and that for some $\eps>0$ we have
$$
\sup_{t>0}t^{-\eps}\abs{\Psi(t)-\Psi(0_+)}<\infty, 
\quad
\sup_{t>0}t^{-\eps}\abs{\wt\Psi(t)-\wt\Psi(0_+)}<\infty.
$$
Then $\Sigma^*$ is asymptotically stable if and only if the unitary operator 
$$
\wt\Psi(\wt\calM)\Psi(\calM)
$$
has a purely singular spectrum. 
\end{theorem}
We give the proof in Appendix~\ref{sec.f}.

\subsection{Open question}
The previous Theorem naturally leads to the following question. 

\noindent
\textbf{Open question:} \emph{for a general abstract spectral datum $\Lambda$, is it true that 
$\Sigma^*$ is asymptotically stable if and only if the spectrum of $\wt\Psi(\wt\calM)\Psi(\calM)$ is purely singular?}

At first glance, reduction to $\wt\Psi(\wt\calM) \Psi(\calM)$ does not seem very useful since in general it is not an easy task to decide if this operator has a purely singular spectrum. But in concrete situations this allows us to give convenient sufficient conditions for $\Sigma^*$ to be asymptotically stable, i.e. (see Theorem~\ref{thm.a5}) for a spectral datum $\Lambda$ to be in $\Lambda(\BMOA_\simp)$. Most importantly, it also allows to construct counterexamples.

\subsection{Positive results}

\begin{theorem}\label{Psi0}
Let  $\Lambda=(\rho, \Psi, \wt\Psi)$ be an abstract spectral datum. 
	Let $\Psi_\#$  be a differentiable unimodular complex valued function on $[0,\infty)$ such that 
	its derivative admits the representation
	$$
	\frac{d}{ds}\Psi_\#(s)=s\int_{-\infty}^\infty e^{is^2t}d\mu(t)
	$$
	with some finite complex-valued measure $\mu$ on $\bbR$. 
	If $(\rho,\Psi,\wt\Psi)\in \Lambda(\BMOA_\simp)$, then 
	$(\rho,\overline{\Psi}_\#\Psi,\Psi_\#\wt\Psi)\in \Lambda(\BMOA_\simp)$; in particular, 
	$(\rho,\overline{\Psi}_\#,\Psi_\#)\in \Lambda(\BMOA_\simp)$.
\end{theorem}
\begin{proof}
By our assumptions on $\Psi_\#$, the function $s\mapsto\Psi_\#(\sqrt{s})$ satisfies the hypothesis of Theorem~\ref{thm.krein}. Since $\wt\calM^2-\calM^2$ is a rank one operator, it follows that 
$$
\Psi_\#(\wt\calM)-\Psi_\#(\calM)=
\Psi_\#(\sqrt{\wt\calM^2})-\Psi_\#(\sqrt{\calM^2})\in\bS_1,
$$
and so, left-multiplying by $\Psi_\#(\calM)^*$, we find 
$$
\Psi_\#(\calM)^*\Psi_\#(\wt\calM)-I\in\bS_1.
$$
Next, denote 
$$
\Sigma^*=\wt\Psi(\wt\calM)\Sigma_0^*\Psi(\calM), \quad 
\Sigma_\#^*=\Psi_\#(\wt\calM)\Sigma^*\Psi_\#(\calM)^*;
$$
here the operator $\Sigma$ corresponds to the spectral datum $(\rho,\Psi,\wt\Psi)$ and 
$\Sigma_\#$ corresponds to the spectral datum $(\rho,\overline{\Psi}_\#\Psi,\Psi_\#\wt\Psi)$.
We have
\begin{align*}
\Sigma_\#^*
&=
\Psi_\#(\calM)\bigl(\Psi_\#(\calM)^*\Psi_\#(\wt\calM)\bigr)\Sigma^*\Psi_\#(\calM)^*
\\
&=
\Psi_\#(\calM)\Sigma^*\Psi_\#(\calM)^*+\text{trace class operator}.
\end{align*}
By taking real parts, it follows that 
$$
\Re \Sigma_\#=\Psi_\#(\calM)(\Re \Sigma^*) \Psi_\#(\calM)^*+\text{trace class operator}.
$$
Thus, using the Kato-Rosenblum theorem, we find that $\Re \Sigma_\#$ satisfies the hypothesis of 
Theorem~\ref{thm.a6} if and only if $\Re \Sigma$ satisfies them. Thus, $\Sigma_\#^*$ is 
asymptotically stable if and only if $\Sigma^*$ is. 
Finally, if $\Psi=\wt\Psi=1$, then by Theorem~\ref{lma.*f1} we have $(\rho,1,1)\in\Lambda(\BMOA_\simp)$, and therefore $(\rho,\overline{\Psi}_\#,\Psi_\#)\in \Lambda(\BMOA_\simp)$. 
\end{proof}

\subsection{Counterexamples}

If one of the functions $\Psi$ or $\wt\Psi$ is constant, then the problem of spectral analysis of $\wt\Psi(\wt\calM)\Psi(\calM)$ simplifies significantly and reduces to the spectral analysis of a multiplication operator. 
Recall that the spectral type of a multiplication operator is easy to determine. Namely, if $\calM$ is 
the multiplication by the independent variable $s$ in $L^2(\rho)$, then a spectral measure (of maximal spectral type) of the operator $\Psi(\calM)$ is the pushforward of $\rho$ by $\Psi$; we denote this pushforward measure by $\rho\circ(\Psi^{-1})$. 

We immediately get the following generalization of Theorem~\ref{t: G positive}. 

\begin{theorem}
Let $\Lambda=(\rho, \Psi, \wt\Psi)$ be an abstract spectral datum  such that $I-\Sigma_0\in\bS_1$ and suppose that one of the two measures $\rho\circ(\Psi^{-1})$, $\wt\rho\circ(\wt\Psi^{-1})$ is supported at a single point. Then $\Lambda\in\Lambda(\BMOA_\simp)$ if and only if the other measure is purely singular. 
\end{theorem}
\begin{proof}
Suppose that $\wt\rho\circ(\wt\Psi^{-1})$ is supported at a point $\zeta$, where $\abs{\zeta}=1$. 
Then 
$$
\wt\Psi(\wt\calM)\Psi(\calM)=\zeta\Psi(\calM).
$$
Thus, the spectrum of $\wt\Psi(\wt\calM)\Psi(\calM)$ is singular if and only if the measure $\rho\circ(\Psi^{-1})$ is singular. It remains to apply Theorem~\ref{p: sing Psi wtPsi}. 

The case when  $\rho\circ(\Psi^{-1})$ is supported a point is considered in the same way. 
\end{proof}

Finally, for definiteness, we give a concrete example of a spectral data $\Lambda$ that is not in $\Lambda(\BMOA_{\simp})$. 

\begin{corollary}
Let $\Lambda=(\rho,\Psi,\wt\Psi)$ be an abstract spectral datum, where the measure $\rho$ is absolutely continuous with $\supp\rho=[a,b]$, $0<a<b<\infty$, and $\Psi(s)=e^{is}$, $\wt\Psi(s)=1$. Then $\Lambda\not\in\Lambda(\BMOA_{\simp})$, i.e. $\Lambda$ does not correspond to any Hankel operator. 
\end{corollary}

\begin{remark*}
It is known (see \cite[Propositions 9.1.11, 9.1.12]{Bogachev2007}) that for a finite
measure $\mu$ without atoms on $\bbT$, there exists a Borel measurable (and even continuous) 
function $F:\bbT\to[0,1]$ such that the measure $\mu\circ (F^{-1})$ is the Lebesgue measure on 
$[0,1]$. Using this fact, for any given $\rho$ without atoms one can always construct $\Psi$ such that $(\rho,\Psi,1)\notin\Lambda(\BMOA_\simp)$. 
\end{remark*}

\section{The self-adjoint case}\label{sec.d}

\subsection{A counterexample for self-adjoint Hankel operators}
\label{s:counterex selfadjoint}

In this section we consider the question of surjectivity of the spectral map in the case of self-adjoint Hankel operators $\Gamma_u$. By Theorem~\ref{thm.sa1}, in this case the spectral datum $(\rho,\Psi,\wt\Psi)$ satisfies the additional constraint that $\Psi$ and $\wt\Psi$ take values $\pm1$. It is reasonable to ask whether all abstract spectral datum with this additional constraint are in $\Lambda(\BMOA_\simp)$. It turns out that the answer to this is negative. However, the corresponding counterexample is more subtle and based on a deep result \cite{P} of perturbation theory. 

Let $\rho$ be an absolutely continuous measure on an interval $(a,b)$, 
$0<a<b<\infty$, $\dd \rho(x) = w(x)\dd x$, where $w$ is a strictly positive H\"{o}lder continuous 
function on $(a,b)$. Multiplying $\rho$ by an appropriate positive constant we can ensure
that the normalization condition \eqref{rholeq1} is satisfied. 

Take any $s_0\in (a,b)$,  and define 
\begin{align*}
\Psi(s)=\wt \Psi(s)=
\begin{cases}
-1, & s<s_0, 
\\
1,& s\geq s_0.
\end{cases}
\end{align*}

\begin{theorem}
\label{t: self-adj counterexample}
Under the above assumptions the operator $\Sigma^*=\wt\Psi(\wt\calM)\wt\calM \calM^{-1} 
\Psi(\calM)$ 
is not asymptotically stable, and so (by Theorem~\ref{thm.a5}) we have 
$(\rho, \Psi, \wt\Psi)\notin\Lambda(\BMOA_\simp)$. 
\end{theorem}
In the rest of this section, we present the proof.

\subsection{Overview of the proof}

Since $\supp\rho$ is separated from $0$, the condition $\Sigma_0-I\in\bS_1$ is easily seen to 
be satisfied, see Lemma~\ref{l: I-Sigma trace class}. 
Then by Theorem~\ref{p: sing Psi wtPsi}, in order to show that $\Sigma^*$ is not asymptotically 
stable 
it is sufficient to show that the absolutely continuous spectrum of operator $\wt\Psi(\wt\calM) 
\Psi(\calM)$ is non-empty. Denoting by 
$E\ci A$ the (projection-valued) spectral measure of a self-adjoint operator $A$, we can write 
\begin{align*}
\wt\Psi(\wt\calM) \Psi(\calM) & 
= \left(I-2E\ci{\wt \calM}\left((-\infty,s_0)\right)\right) 
\left(I-2E\ci{\calM}\left((-\infty,s_0)\right) \vphantom{E\ci{\wt \calM}}\right)     \\
& = \left(I-2E\ci{\wt 
\calM^2}\left((-\infty,s_0^2)\right)\right)\left(I-2E\ci{\calM^2}\left((-\infty,s_0^2)\right)\right),
\end{align*}
and so the question reduces to investigating the geometry of the ranges of 
the two spectral projections $E_{\calM^2}\bigl((-\infty,s_0^2)\bigr)$ and $E_{\wt \calM^2}\bigl((-\infty,s_0^2)\bigr)$. 
This question has been studied in \cite{P} in the general framework of scattering theory. 
We recall the relevant results of \cite{P} in the next subsection. They assert that in our case the a.c.~spectrum of 
the product
$$
E\ci{\calM} \bigl((-\infty, s_0)\bigr)    E\ci{\wt\calM} \bigl([s_0, \infty)\bigr)  E\ci{\calM} \bigl((-\infty, s_0)\bigr) 
$$
is non-empty.  
From here, using some general results on the geometry of two subspaces in a Hilbert space
(we use Halmos' paper \cite{Halmos}) it is not difficult to derive that the a.c. 
spectrum of $\wt\Psi(\wt\calM) \Psi(\calM)$ is also non-empty.

\subsection{Products of spectral projections}

Here we briefly recall some of the results of \cite{P}, adapted to the particular case at hand. 
Let $A_0$ and $A_1$ be bounded (for simplicity) self-adjoint operators in a Hilbert space $X$, such that the difference $A_1-A_0$ is a (negative) rank one operator:
\[
A_1-A_0=-\jap{\fdot,\omega}\omega,
\]
where $\omega$ is a non-zero element in $X$. 

Assume that the operators $A_0$ and $A_1$ have a purely absolutely continuous spectrum on an  interval $(\alpha, \beta)$. 
Assume also that the derivatives $F_0'(s)$ and $F_1'(s)$, where 
\[
F_0(s):=\jap{E_{A_0}(-\infty,s)\omega,\omega}\ ,
\qquad 
F_1(s):=\jap{E_{A_1}(-\infty,s)\omega,\omega}
\]
exist for $s\in(\alpha,\beta)$ and are H\"older continuous functions of $s$. 
Define 
\begin{align}
\label{e: A(lambda)}
\varkappa(s)=\pi^{2}\left(F_{0}^{\prime}(s)\right)^{1 / 2} F_1^{\prime}(s)\left(F_{0}^{\prime}(s)\right)^{1 / 2}, \qquad s \in (\alpha, \beta). 
\end{align}
The following fact was proved in \cite{P}, see Lemma~3.2(ii) there. 

\begin{lemma}
\label{l:spectrum oper sin}
Under the above assumptions the absolutely continuous part of the operator 
\begin{align}
\label{e:op sin^2 01}
E\ci{A_0} \bigl((-\infty, s)\bigr)    E\ci{A_1} \bigl([s, \infty)\bigr)  E\ci{A_0}\bigl((-\infty, s)\bigr) 
\end{align}
is unitarily equivalent to the operator of multiplication by $x$ in $L^2([0,\varkappa(s)], \dd x)$. 
\end{lemma}

In Section~\ref{s: diff proj 04} below we  take $A_0=\calM^2$, $A_1=\wt\calM^2$, $\omega=\1\in L^2(\rho)$ and show that the above hypothesis are satisfied and $\varkappa(s)>0$, and so the operator \eqref{e:op sin^2 01} has a non-trivial absolutely continuous part. 
From there we will deduce that the operator $\wt\Psi(\wt\calM) \Psi(\calM)$ has a non-trivial 
absolutely continuous part. In order to do this, we will use some Hilbert space geometry; this is discussed in the next subsection.

\begin{remark*}
The focus of \cite{P} was the connection between the a.c. spectrum of combinations of spectral 
projections \eqref{e:op sin^2 01} (and other similar ones) and the eigenvalues of the 
\emph{scattering matrix} of the pair of operators $A_0$, $A_1$. 
In the case at hand (when $A_1-A_0$ is a rank one operator), the scattering matrix is simply a unimodular function on the a.c. spectrum of $A_0$, and it can be expressed directly in terms of $\varkappa(s)$. In any case, 
Lemma~\ref{l:spectrum oper sin}, which was an intermediate step in \cite{P}, suffices for our purposes, and so we are not discussing the scattering matrix here. 
\end{remark*}

\subsection{Pairs of projections and the a.c.~spectrum of $\wt\Psi(\wt\calM)\Psi(\calM)$}
\label{s: diff proj Halmos}
Let us start with a brief discussion of some of the construction of Halmos' beautiful paper 
\cite{Halmos}. 
Let $P$ and $Q$ be two orthogonal projections in a Hilbert space. 
Following Halmos, we will say that $P$ and $Q$ are in \emph{generic position}, if 
each of the four subspaces 
\begin{align}
\Ran P\cap \Ran Q, \quad
\Ker P\cap \Ker Q, \quad
\Ran P\cap \Ker Q, \quad
\Ker P\cap \Ran Q
\label{d2}
\end{align}
are trivial.

\begin{theorem}\cite{Halmos}\label{thm.d2}
Let $P$, $Q$ be two orthogonal projections in a generic position. 
Then there exist self-adjoint positive semi-definite commuting contractions $\calS$ and $\calC$, with $\calS^2+\calC^2=I$ 
and $\Ker \calS=\Ker \calC=\{0\}$, such that the pair 
$P$, $Q$ is unitarily equivalent to the pair 
$$
\begin{pmatrix}
I & 0
\\
0 & 0
\end{pmatrix}, \qquad
\begin{pmatrix}
\calC^2 & \calC\calS
\\
\calC\calS & \calS^2
\end{pmatrix} .
$$
\end{theorem}

We can now put this together with Lemma~\ref{l:spectrum oper sin}.
\begin{lemma}\label{lma.d3}
Assume, in the hypothesis of Lemma~\ref{l:spectrum oper sin}, that for some $s\in(\alpha,\beta)$ we have $\varkappa(s)>0$. Then the a.c. spectrum of the unitary operator 
\[
\biggl(I-2E_{A_0}\bigl((-\infty,s)\bigr)\biggr)\biggl(I-2E_{A_1}\bigl((-\infty,s)\bigr)\biggr)
\]
coincides with the arc on the unit circle 
\begin{align}
\{1-2\sigma^2+i\sigma\sqrt{1-\sigma^2}: \quad \sigma \in [-\varkappa(s),\varkappa(s)]\}.
\label{d7}
\end{align}
In particular, this a.c. spectrum is non-empty.
\end{lemma}
\begin{proof}
Denote 
$$
P=E_{A_0}(-\infty,s), \quad Q=E_{A_1}(-\infty,s).
$$ 
These projections are not necessarily in generic position, but for our purposes it is sufficient to 
consider their generic parts.  
Namely, let us write our Hilbert space $X$ as 
$$
X=X_0\oplus X_{\text{gen}},
$$ 
where $X_0$ is the orthogonal sum of the four subspaces \eqref{d2}. 
It is easy to see that each of these four subspaces is invariant for both $P$ and $Q$, 
and therefore $X_{\text{gen}}$ is also invariant for both $P$ and $Q$. 
Furthermore, the pair 
$$
P_{\text{gen}}:=P|_{X_{\text{gen}}}, 
\quad
Q_{\text{gen}}:=Q|_{X_{\text{gen}}}
$$ 
is in a generic position. Thus, according to Theorem~\ref{thm.d2}, we can write 
\begin{align}
P_{\text{gen}}=
U\begin{pmatrix}
I & 0
\\
0 & 0
\end{pmatrix}U^*,\quad 
Q_{\text{gen}}=
U\begin{pmatrix}
\calC^2 & \calC\calS
\\
\calC\calS & \calS^2
\end{pmatrix}U^*,
\label{d5}
\end{align}
where $U$ is a unitary operator. 

Since the restriction of $P$ and $Q$ onto each of the four subspaces \eqref{d2} is either $0$ or 
$I$, the absolutely continuous parts of the the operators $P(I-Q)P$ and $(I-2Q)(I-2P)$ coincide 
with the absolutely continuous parts of their generic counterparts $P\ti{gen} (1-Q\ti{gen}) P\ti{gen} $ and 
$(I-2Q\ti{gen}) (I-2P\ti{gen})$ respectively.

One can see from \eqref{d5}  that the operator $\calS^2$ is unitarily equivalent  to the operator 
$$
P\ti{gen} (1-Q\ti{gen}) P\ti{gen}|_{X\ti{gen}},
$$ 
so the absolutely continuous part of $\calS^2$ is unitarily 
equivalent to the absolutely continuous part of $P(I-Q)P$, which is described by Lemma~\ref{l:spectrum oper sin}. So, the absolutely continuous part of $\calS^2$ is unitarily equivalent to the multiplication by the 
independent variable $x$ in $L^2([0, \varkappa(s)], \dd x)$.  Alternatively: the absolutely 
continuous part of $\calS$ is unitarily equivalent to the multiplication by $x$ in $ L^2([0, 
\sqrt{\varkappa(s)}], \dd x)$.

On the other hand, according to our model \eqref{d5}, we have
$$
U^*(I-2Q_{\text{gen}})(I-2P_{\text{gen}}) U=
\begin{pmatrix}
-I+2\calC^2 & -2\calC\calS
\\
2\calC\calS & I-2\calS^2
\end{pmatrix}
=
\begin{pmatrix}
I-2\calS^2 & -2\calC\calS
\\
2\calC\calS & I-2\calS^2
\end{pmatrix},
$$
where we have used the identity $\calS^2+\calC^2=I$ at the last step.
The numerical matrix  
\begin{align*}
B({\sf s}):=
\begin{pmatrix}
1-2{\sf s}^2 & -2{\sf c}{\sf s} \\ 2{\sf c}{\sf s} & 1-2{\sf s}^2 
\end{pmatrix} , \qquad 
0<{\sf s}<1, \ {\sf c}=\sqrt{1-{\sf s}^2}>0
\end{align*}
has eigenvalues $\lambda_{\pm}({\sf s}):= 1-2{\sf s}^2 \pm 2 i {\sf s} \sqrt{1-{\sf s}^2}$, 
and 
therefore  $B(\sf s)$ can be decomposed as 
\begin{align*}
B({\sf s}) = V(\sf s) 
\begin{pmatrix}
\lambda_+({\sf s}) & 0 \\ 0 &\lambda_-({\sf s}) 
\end{pmatrix} V({\sf s})^*
\end{align*}
where $V({\sf s})=(v_{j,k}({\sf s}))_{j,k=1}^2$ is a unitary $2\times 2$ matrix. The matrix $V({\sf s})$ 
can be explicitly computed, and can be chosen so the function ${\sf s}\mapsto V({\sf s}) $ is 
continuous (and so measurable) on the interval $(0,1)$. Therefore 
\begin{align*}
U^*(I-2Q_{\text{gen}})(I-2P_{\text{gen}})U & = B(\calS)
 =
V(\calS) \begin{pmatrix}
\lambda_+ (\calS) & 0 \\ 0 & \lambda_- (\calS)
\end{pmatrix} V(\calS)^* , 
\end{align*}
where $V(\calS) = (v_{j,k} (\calS))_{j,k=1}^2$. 
So 
$(I-2Q_{\text{gen}})(I-2P_{\text{gen}})$ 
is unitarily equivalent to  the direct sum $\lambda_+(\calS) \oplus \lambda_-(\calS)$. 
From here we see that the a.c. spectrum of $(I-2Q)(I-2P)$ is given by the arc \eqref{d7}, and in particular it is non-empty. 
\end{proof}

\subsection{Proof of Theorem~\ref{t: self-adj counterexample}}
\label{s: diff proj 04}
As mentioned above, we take $A_0=\calM^2$, $A_1=\wt\calM^2$, $\omega=\1\in X=L^2(\rho)$, $(\alpha,\beta)=(a^2,b^2)$ in Lemma~\ref{l:spectrum oper sin}; we need to check that the hypothesis of this Lemma is satisfied and $\varkappa(s)>0$. 
We have 
$$
F_0(s)=\jap{E_{\calM^2}\bigl((-\infty,s)\bigr)\1,\1}=\int_{-\infty}^{\sqrt{s}}d\rho(s')=\int_{-\infty}^{\sqrt{s}}w(s')ds',
$$
and so the derivative 
$$
F_0'(s)=\frac{d}{ds}\int_{-\infty}^{\sqrt{s}}w(s')ds'=\frac1{2\sqrt{s}}w(\sqrt{s})
$$
exists and (by our assumptions on $w$) is H\"older continuous on $(a^2,b^2)$. 
Let us consider $F_1(s)$. We use the standard rank one identity (which follows from the resolvent identity)
$$
T_1(z) = \frac{T_0(z)}{1 - T_0(z)}\ ,
$$
where
$$
T_0(z)=\jap{(\calM^2-z)^{-1}\1,\1}\ , 
\quad
T_1(z)=\jap{(\wt\calM^2-z)^{-1}\1,\1}\ .
$$
By definition, the operator $T_0(z)$ is the Cauchy transform of a H\"{o}lder continuous function, 
and therefore $T_0(x+i0)$ is H\"{o}lder continuous on $(\alpha, \beta)$.
Note that in our case $\Im T(x+i0)>0$ on the interval $(\alpha,\beta)$, and so $1-T_0(x+i0)\ne 0$ 
on this interval. 
Further, we have 
\begin{align*}
\Im T_1(x+i0) = \frac{\Im T_0(x+i0)}{|1 - T_0(x+i0)|^2}\ ,
\end{align*}
and so the  density $F_1'(s)= \Im T_1(s+i0)$ is also H\"{o}lder continuous and non-vanishing on 
$(\alpha, \beta)$. 

We have checked the hypotheses of Lemma~\ref{l:spectrum oper sin} and we have established that 
$\varkappa(s)>0$ in \eqref{e: A(lambda)}. Using Lemma~\ref{lma.d3}, we find that the a.c. spectrum of $\wt\Psi(\wt\calM)\Psi(\calM)$ is non-empty. 
The proof of Theorem~\ref{t: self-adj counterexample} is complete. \qed

\section{Applications to the cubic Szeg\H{o} equation}\label{sec.sz}

\subsection{The cubic Szeg\H{o} equation}\label{sec.a11}
The cubic Szeg\H{o} equation is the Hamiltonian evolution equation 
\begin{equation}\label{szego}
i\partial_tu=P(|u|^2u)\ , 
\end{equation}
where $P$ is the Szeg\H{o} projection, i.e.~the orthogonal projection from $L^2(\bbT)$ 
onto $H^2$. Here $u=u(t, z)$, $t\in\R$, $z\in\bbT$, and the projection $P$ is taken in variable 
$z$. 
In this section for typographical reasons we omit the variable  $z$, and will be using $u(t)$ 
instead of more formal $u(t, \fdot)$. 
 
This equation was introduced in \cite{GG00} where it has been proved to be wellposed on the 
intersection of 
$H^2(\bbT)$ with the Sobolev space  $W^{s,2}(\bbT )$, for every $s\geq \frac 12$. More recently, 
the wellposedness was extended to $\BMOA(\bbT )$ in \cite{GK}. In this case, since 
$\BMO(\bbT)\subset \bigcap_{p<\infty}L^p(\bbT)$, the right hand side of \eqref{szego} is in 
$H^2(\bbT)$, so \eqref{szego} can be interpreted as an ODE with $H^2(\bbT)$-valued functions. 

An important property of this equation is that it admits a Lax pair structure involving Hankel 
operators $H_u$ and $\wt H_u$, which stimulated the study of the spectral map $\Lambda $, starting 
with 
functions $u$ in  $H^2(\bbT)\cap W^{\frac 12, 2}(\bbT) $ and $\VMOA(\bbT)$; see \cite{GG1,GGAst}.

Our first result is the following description of the action of the Szeg\H{o} dynamics on the set 
$\Lambda(\BMOA_\simp)$.
\begin{theorem}\label{actionszego}
	Let $u_0\in \BMOA_\simp(\bbT)$ with $\Lambda (u_0)=(\rho, \Psi_0,\wt \Psi_0)$. Denote by $u$ 
	the solution   of \eqref{szego} such that $u(0)=u_0$. 
	Then, for every $t\in \bbR$, we have $u(t)\in \BMOA_\simp(\bbT)$  and
	$$\Lambda (u(t))=(\rho, {\rm e}^{-its^2}\Psi_0 (s),  {\rm e}^{its^2}\wt \Psi_0 (s))\ .$$
\end{theorem}
This result is consistent with Theorem~\ref{Psi0} with $\Psi_\#(s)=e^{its^2}$. 
The proof is given in Sections~\ref{sec.i2}--\ref{sec.i4}.

An important issue in the study of the cubic Szeg\H{o} equation is the long time behavior of its solutions. 
Firstly let us discuss the boundedness of trajectories. Using the Lax pair structure, one can prove 
that 
trajectories are bounded in $W^{\frac 12, 2}(\bbT) $ and in $\BMOA(\bbT )$. However, in 
\cite{GGAst}, it is proved that trajectories are generically unbounded in $W^{s, 2}(\bbT) $ for 
every $s>\frac 12$. 

Secondly comes the problem of almost periodicity of the trajectories in spaces where they are 
bounded. Let us recall that a function $F=F(x)$ on the real line with values in a Banach space $X$ is called \emph{almost periodic} is it can be approximated (in $C(\bbR;X)$) by finite linear combinations of functions of the form $F(x)=e^{ia x}\psi$, where $a\in\bbR$ and $\psi\in X$. 

In \cite{GGAst}, it is proven that
every trajectory in $W^{\frac 12, 2}(\bbT) $  is almost periodic. Using the same method, a similar 
result holds for trajectories in $\VMOA(\bbT )$. It is therefore natural to ask whether this almost 
periodicity holds for every 
trajectory in $\BMOA(\bbT )$. The following theorem shows that the dynamics are much richer in this 
case.
\begin{theorem}\label{notAP}
Let $u_0\in \BMOA_\simp(\bbT)$ with $\Lambda (u_0)=(\rho, \Psi_0,\wt \Psi_0)$. Denote by $u=u(t)$ 
the solution   of \eqref{szego} such that $u(0)=u_0$. 
Then, if $\rho$ is not a pure point measure, then the Fourier coefficient $\widehat u_0(t)$ of $u(t)$ is not almost periodic, and therefore $u=u(t)$ (as a function with values in $\BMOA$) is not almost periodic. 
\end{theorem}
The proof is given in Section~\ref{sec.i5}. 

We conclude by discussing the role of the simplicity condition \eqref{a4}. 
In \cite{GGAst}, the action of the Szeg\H{o} dynamics on the spectral data was described for all compact operators $H_u$ (without the simplicity assumption). In fact, the formula is exactly the same as in Theorem~\ref{actionszego}, where $\Psi_0(s)$ and $\widetilde\Psi_0(s)$ are functions with values in the set of Blaschke products. The multiplicity of singular values seems to play a role in the phenomenon of weak turbulence (i.e. growth of high Sobolev norms) of solutions to the Szeg\H{o} equations. More precisely, in \cite{GGAst}, using the vicinity of solutions with multiple spectrum, the authors construct a $G_\delta$-dense set of initial conditions such that the corresponding solutions are weakly turbulent.

\subsection{The action of $H_{u(t)}$ and $\wt H_{u(t)}$ for smooth initial data}
\label{sec.i2}
Here we make the first step towards the proof of Theorem~\ref{actionszego}: we describe the evolution under the cubic Szeg\H{o} equation for smooth initial data. 
\begin{lemma}\label{lma.i3}
Assume the hypothesis of Theorem~\ref{actionszego} and assume in addition that $u_0$ is smooth: $u_0\in C^\infty (\bbT)\cap 
H^2(\bbT)$. Then, using our notation \eqref{b1} for spectral measures, we have
\begin{equation}
\rho_{u(t)}^{\abs{H_{u(t)}}}=\rho_{u_0}^{\abs{H_{u_0}}}, 
\quad
\rho_{u(t)}^{\abs{\wt H_{u(t)}}}=\rho_{u_0}^{\abs{\wt H_{u_0}}}
\label{eq.spmeas}
\end{equation}
for all $t>0$. Furthermore, for any continuous $f$ and any $t>0$ we have
\begin{align}
H_{u(t)}f(|H_{u(t)}|)u(t)&=|H_{u(t)}|\overline f(|H_{u(t)}|)\overline \Psi_{u_0} (|H_{u(t)}|){\rm e}^{it 
H_{u(t)}^2}u(t)\ ,\label{Psit}\\
{\wt H}_{u(t)}f(|{\wt H}_{u(t)}|)u(t)&=|{\wt H}_{u(t)}|\overline f(|{\wt H}_{u(t)}|)\wt \Psi_{u_0} (|{\wt H}_{u(t)}|){\rm e}^{it 
{\wt H}_{u(t)}^2}u(t)\ .
\label{tildePsit} 
\end{align}
\end{lemma}
\begin{proof}
Throughout the proof, we write $u$ in place of $u(t)$ if there is no danger of confusion.
For $u_0\in C^\infty (\bbT)\cap H^2(\bbT)$, we borrow from \cite{GG2} the following Lax pair identities, 
\begin{align*}
	\frac{dH_u}{dt}&=[B_u,H_u]\ ,\qquad B_u:=\frac i2H_u^2-iT_{\vert u\vert ^2}\ ,
	\\
	\frac{d\wt H_u}{dt}&=[{\wt B}_u,\wt H_u]\ ,\qquad {\wt B}_u:=\frac i2\wt H_u^2-iT_{\vert u\vert ^2}\ ,
\end{align*}
where $T_a$ denotes the \emph{Toeplitz operator} with symbol $a\in L^\infty(\bbT)$, 
$T_a:H^2\to H^2$, 
\[
T_a f= P(af), \qquad f\in H^2. 
\]
Then we define $W=W(t)$, ${\wt W}={\wt W}(t)$ to be the solutions of the following linear ODEs
on the set of bounded linear operators on $H^2(\bbT)$,
$$
\frac {dW}{dt}=B_uW,\quad 
\frac{d\wt W}{dt}={\wt B}_u\wt W,\quad 
W(0)={\wt W}(0)=I\ .
$$
One easily checks that $W(t)$ and  ${\wt W}(t)$ are unitary operators and
\begin{equation}\label{conj}
H_{u(t)}=W(t)H_{u_0} W(t)^*,\quad {\wt H}_{u(t)}={\wt W}(t){\wt H}_{u_0} {\wt W}(t)^*\ .
\end{equation}
Consequently,
$$|H_{u(t)}|=W(t)|H_{u_0}| W(t)^*\ ,\ |{\wt H}_{u(t)}|={\wt W}(t)|{\wt H}_{u_0}| {\wt W}(t)^*\ .$$    
Next, let us  identify $W(t)^*z^0, W(t)^*u(t), {\wt W}(t)^*u(t)$. 
We begin with $W(t)^*z^0$:
$$
\frac{d}{dt}W(t)^*z^0=-W(t)^*B_u z^0\ ,$$
with
$$B_uz^0=\frac i2 H_u^2z^0-iT_{\vert u\vert ^2}z^0=-\frac i2H_u^2z^0\ .$$
Hence
$$\frac{d}{dt}W(t)^*z^0=\frac i2W(t)^*H_u^2z^0=\frac i2H_{u_0}^2W(t)^*z^0\ .$$
This yields
\[
W(t)^*z^0={\rm e}^{i\frac t2 H_{u_0}^2}z^0\ .
\]
Consequently,
$$W(t)^*u(t)=W(t)^*H_{u(t)}z^0=H_{u_0}W(t)^*z^0=H_{u_0}{\rm e}^{i\frac t2 H_{u_0}^2}z^0\ ,$$
and therefore, using the anti-linearity of $H_{u_0}$, 
\begin{equation}\label{U*u}
W(t)^*u(t)={\rm e}^{-i\frac t2 H_{u_0}^2}u_0\ .
\end{equation}
On the other hand,
\begin{align*}
\frac d{dt}W(t)^*{\wt W}(t)&= -W(t)^*B_{u(t)} {\wt W}(t)+W(t)^*{\wt B}_{u(t)}{\wt W}(t) = 
W(t)^*({\wt B}_{u(t)}-B_{u(t)}){\wt W}(t) 
\\
&=\frac i2 W(t)^*({\wt H}_{u(t)}^2-H_{u(t)}^2){\wt W}(t)=\frac i2 (W(t)^*{\wt W}(t){\wt H}_{u_0}^2-H_{u_0}^2W(t)^*{\wt W}(t))\ .
\end{align*}
We infer
$$W(t)^*{\wt W}(t)={\rm e}^{-i\frac t2 H_{u_0}^2}{\rm e}^{i\frac t2 {\wt H}_{u_0}^2}\ ,$$
and consequently
\begin{equation}\label{V*u}
{\wt W}(t)^*u(t)
=
{\rm e}^{-i\frac t2 {\wt H}_{u_0}^2}{\rm e}^{i\frac t2 H_{u_0}^2}W(t)^*u(t)
=
{\rm e}^{-i\frac t2 {\wt H}_{u_0}^2}{\rm e}^{i\frac t2 H_{u_0}^2}{\rm e}^{-i\frac t2 
H_{u_0}^2}u_0={\rm e}^{-i\frac t2 {\wt H}_{u_0}^2}u_0\ .
\end{equation}
Next, using \eqref{U*u} and \eqref{conj},
\begin{align*}
	\jap{f(|H_{u(t)}|)u(t), u(t)}&=\jap{W(t)^*f(|H_{u(t)}|)u(t), W(t)^*u(t)}
	\\
	&=\jap{f(|H_{u_0}|)\, {\rm e}^{-i\frac t2 H_{u_0}^2}u_0,{\rm e}^{-i\frac t2 H_{u_0}^2}u_0}
	\\
	&=\jap{f(|H_{u_0}|)u_0,u_0}
\end{align*}
and we obtain the first one of the identities \eqref{eq.spmeas}. The second one is obtained in a similar way. 

Next, since $u_0\in\BMOA_\simp(\bbT)$, for every continuous function $f$ we have (cf. \eqref{e: model A 02}) 
\begin{align*}
H_{u_0}f(|H_{u_0}|)u_0&=|H_{u_0}|\overline \Psi_{u_0} (|H_{u_0}|)\overline f(|H_{u_0}|)u_0
\\ 
{\wt H}_{u_0}f(|{\wt H}_{u_0}|)u_0&=|{\wt H}_{u_0}|\wt \Psi_{u_0} (|{\wt H}_{u_0}|)\overline 
f(|{\wt H}_{u_0}|)u_0\ .
\end{align*}
Then, using \eqref{U*u}, \eqref{V*u} and \eqref{conj},
\begin{align*}
W(t)^*H_{u(t)}f(|H_{u(t)}|)u(t)&=H_{u_0}f(|H_{u_0}|){\rm e}^{-i\frac t2 H_{u_0}^2}u_0
\\
&=
|H_{u_0}|\overline \Psi_{u_0} (|H_{u_0}|)\overline f(|H_{u_0}|){\rm e}^{i\frac t2 H_{u_0}^2}u_0
\\ 
&=W(t)^*|H_{u(t)}|\overline f(|H_{u(t)}|)\overline \Psi_{u_0} (|H_{u(t)}|){\rm e}^{it H_{u(t)}^2}u(t)\ ,
\\
{\wt W}(t)^*{\wt H}_{u(t)}f(|{\wt H}_{u(t)}|)u(t)&= {\wt H}_{u_0}f(|{\wt H}_{u_0}|){\rm e}^{-i\frac t2 {\wt H}_{u_0}^2}u_0
\\
&=|{\wt H}_{u_0}|\wt \Psi_{u_0} (|{\wt H}_{u_0}|)\overline f(|{\wt H}_{u_0}|){\rm e}^{i\frac t2 {\wt H}_{u_0}^2}u_0
\\
&={\wt W}(t)^*|{\wt H}_{u(t)}|\overline f(|{\wt H}_{u(t)}|)\wt \Psi_{u_0} (|{\wt H}_{u(t)}|){\rm e}^{it {\wt H}_{u(t)}^2}u(t)\ ,
\end{align*}
and finally we arrive at \eqref{Psit} and \eqref{tildePsit}.
\end{proof}

\subsection{Approximation argument}
\label{sec.i3}
As the second step, we extend  identities \eqref{eq.spmeas}, \eqref{Psit} and \eqref{tildePsit}  to the general case of 
initial data $u_0\in \BMOA_\simp $. The new difficulty here is that the operator $T_{|u(t)|^2}$ is 
unbounded, hence the unitary operators $W(t)$ and ${\wt W}(t)$ are more difficult to define. Therefore we 
prefer to use an approximation argument.  

\begin{lemma}
\label{l: rat approx}
Let $u\in\BMOA\ti{simp}$; then there exists a sequence of polynomial functions $u_n\in \BMOA\ti{simp}$ converging 
to $u$ strongly in $H^2(\bbT)$ with a uniform bound in the $\BMOA$ norm. 
\end{lemma}
\begin{proof}
\emph{Step 1: approximation by polynomial functions.}
Take $0\le r_n \nearrow 1$, and define $u_n(z):= u(r_n z)$, $z\in\bbT$. Clearly $u_n\to u$ 
strongly in $H^2(\bbT)$, and writing $H_u$ in a matrix form (with respect to the standard basis in $H^2$), it is easy to see that 
$$
\norm{H_{u_n}}\leq \norm{H_u}. 
$$
Observe that one of the equivalent norms on $\BMOA$ is given by $\norm{u}_{\BMOA}=\norm{H_u}$. It follows that 
\begin{align*}
\sup_n \|u_n\|\ci{\BMOA} \leq \|u\|\ci{\BMOA}.
\end{align*}
Functions $u_n$ are analytic in the closed unit disc $\overline{\bbD}$, so they can be approximated by polynomial functions uniformly in $\overline{\bbD}$. Since the norm in $C(\overline{\bbD})$ is stronger than the norm of $\BMOA$, we obtain approximations by polynomial functions in $H^2$ with the uniform bound on the $\BMOA$ norm.

\emph{Step 2: approximation by polynomial functions in $\BMOA_\simp$.}
Denote by $\mathscr P_N$ the vector space of polynomial functions of degree at most $N$, and by $\mathscr P_{N,\simp}$ the subset of  $u\in 
\mathscr P_N$ satisfying the simplicity of spectrum condition \eqref{a4}. To complete the proof of the lemma, it suffices to show that 
$\mathscr P_{N,\simp}$ is dense in $\mathscr P_N$. Given $u\in \mathscr P_N$, we observe that the range of $H_u$ is contained in $\mathscr P_N$. It follows that $u\in \mathscr P_{N,\simp}$ whenever the $(N+1)$ vectors $u,H_u^2(u),\dots ,H_u^{2N}(u)$ are linearly independent, or equivalently whenever the Gram determinant
$$G_N(u):=\det \jap {H_u^{2k}(u) ,  H_u^{2\ell}(u)} _{0\leq k,\ell \leq N}$$
is not zero. Since $G_N(u)$ is a polynomial function of the real parts and of the imaginary parts 
of the Fourier coefficients of $u$, the set $\{ G_N\neq 0\}$ is either empty  or a dense open  
subset of $\mathscr P_N$. Therefore we are reduced to proving that $\mathscr P_{N,\simp }$ is not 
empty. Consider $u(z):=z^{N-1}+z^N$. 
The matrix of the linear Hankel operator  
$\Gamma_u=\Gamma_u^*$ in the basis $\{ z^k, k=0,\dots ,N\}$ of $\mathscr P_N\supset \Ran \Gamma_u$ 
is
\begin{align*} 
\left (  \begin{array}{ccccc} 
0& 0&\cdot &1&1\\
0&\cdot &1&1&0\\
\cdot&\cdot&\cdot&\cdot& \cdot\\
1&1&0&\cdot&0\\
1&0&\cdot&0&0   \end{array}\right ) . 
\end{align*}
Consequently, $\Gamma_u$ and so $H_u$ are injective on $\mathscr P_N$. 
Moreover, since $H_u^2 =\Gamma_u\bC \bC\Gamma_u^* = \Gamma_u\Gamma_u^* =\Gamma_u^2$, one can check 
that the matrix of $H_u^2$ in the same basis is three-diagonal, viz. 
\begin{align*}
H_u^2(z^0)&=2z^0+z, 
\\
H_u^2(z^k)&=z^{k-1}+2z^k+z^{k+1},\quad\text{ if }\quad 1\leq k\leq N-1,
\\
H_u^2(z^N)&=z^{N-1}+z^N.
\end{align*}
From these formulae, we infer, via an induction argument on $k$, that there exist real numbers $c_{k,j}$ such that  $H_u^{2k}(z^0)=z^k+\sum_{j<k}c_{k,j}z^j$ for $k=0,\dots ,N$. 
We conclude that the vectors $H_u^{2k}(z^0), k=0,\dots, N,$ are linearly independent, and, applying $H_u$, that $H_u^{2k}(u), k=0,\dots, N,$ are linearly independent, or that $u\in \mathscr P_{N,\simp}$.
\end{proof}

\bigskip

Before proceeding, for the purposes of clarity we state a (well-known) simple fact as a lemma. 

\begin{lemma}\label{lma.i5}
Let $u,u_n\in\BMOA$, $n\in\bbN$, with $\sup_n\norm{u_n}_{\BMOA}<\infty$, and assume that $\norm{u_n-u}_{H^2}\to0$ as $n\to\infty$. Then we have the strong convergence $H_{u_n}\to H_u$, $\wt H_{u_n}\to\wt H_u$ and $f(\abs{H_{u_n}})\to f(\abs{H_{u}})$, $f(\abs{\wt H_{u_n}})\to f(\abs{\wt H_{u}})$ for any continuous function $f$. 
\end{lemma}
\begin{proof}
For any $m\geq0$ we have 
$$
H_uz^m=H_uS^mz^0=(S^*)^mH_uz^0=(S^*)^m u
$$
and therefore $\norm{H_{u_n}z^m-H_uz^m}_{H^2}\to0$ as $n\to0$. It follows that $\norm{H_{u_n}p-H_up}_{H^2}\to0$ for all polynomials $p$. The uniform bound on $\norm{u_n}_{\BMOA}$ is equivalent to the uniform bound on the operator norms of $H_{u_n}$, and so by the ``$\eps/3$-argument'' we conclude that $H_{u_n}\to H_u$ strongly. It follows that $H_{u_n}^2\to H_u^2$ strongly, and therefore $f(H_{u_n}^2)\to f(H_u^2)$ for any continuous function $f$. Since $\abs{H_u}=\sqrt{H_u^2}$, we also obtain $f(\abs{H_{u_n}})\to f(\abs{H_u})$ for any continuous $f$.

Finally, since $\wt H_{u}=H_{S^*u}$ and $\norm{S^*u_n\to S^*u}_{H^2}\to0$  with the uniform bound on the $\BMO$ norms of $S^*u_n$, we obtain the corresponding statements for $\wt H_{u_n}$. 
\end{proof}

We also quote a corollary of the main result of \cite{GK} on the continuous dependence of the solution to the cubic Szeg\H{o} equation on the initial data. 
\begin{proposition}\cite[Theorem 1]{GK}\label{prp.i6}
Suppose $u_0,u_{0,n}\in\BMOA$, $n\geq1$ are such that $\norm{u_0-u_{0,n}}_{H^2}\to0$ and $\sup_n\norm{u_{0,n}}_{\BMOA}<\infty$. Let $u(t)$, $u_n(t)$ be the solutions to \eqref{szego} with the initial data $u(0)=u_0$, $u_n(0)=u_{0,n}$. Then for any $t>0$, we have $\norm{u_n(t)-u(t)}_{H^2}\to0$ as $n\to\infty$. Furthermore, the $\BMO$ norm is preserved by the Szeg\H{o} dynamics, i.e. 
$$
\norm{u(t)}_{\BMOA}=\norm{u_0}_{\BMOA}, \quad t>0.
$$
\end{proposition}

Now we are ready to extend Lemma~\ref{lma.i3} to non-smooth initial data. 

\begin{lemma}\label{lma.i7}
Assume the hypothesis of Theorem~\ref{actionszego}. 
Then for any continuous $f$ relations  \eqref{eq.spmeas}, \eqref{Psit} and \eqref{tildePsit} hold true. 
\end{lemma}
\begin{proof}
Using Lemma~\ref{l: rat approx}, for a given $u_0\in\BMOA_{\simp}$ we construct a sequence of polynomial  functions $u_{0,n}\in\BMOA_{\simp}$ converging to $u_0$ in $H^2(\bbT)$ and uniformly bounded in $\BMO$ norm. For each $u_{0,n}$, the conclusion of Lemma~\ref{lma.i3} holds. Our purpose is to pass to the limit $n\to\infty$ in \eqref{eq.spmeas}, \eqref{Psit} and \eqref{tildePsit}.

\emph{Step 1: convergence of measures: passing to the limit in \eqref{eq.spmeas}.}
For each $n$, we write \eqref{eq.spmeas} in the weak form as 
\begin{align}
\jap{f(\abs{H_{u_n(t)}})u_n(t),u_n(t)}
&=
\jap{f(\abs{H_{u_n(0)}})u_n(0),u_n(0)}\ , 
\label{i3}
\\
\jap{f(\abs{\wt H_{u_n(t)}})u_n(t),u_n(t)}
&=
\jap{f(\abs{\wt H_{u_n(0)}})u_n(0),u_n(0)}
\notag
\end{align}
for any continuous function $f$. By Lemma~\ref{lma.i5}, we can pass to the limit $n\to\infty$ in the right hand side. Similarly, by Proposition~\ref{prp.i6} combined with Lemma~\ref{lma.i5}, we can pass to the limit in the left hand side. We obtain the desired relations \eqref{eq.spmeas}, expressed in the weak form.

\emph{Direction of further proof:}
For every $n$, we have the identities \eqref{Psit} and \eqref{tildePsit}: 
\begin{align}
H_{u_n(t)}f(|H_{u_n(t)}|)u_n(t)&=|H_{u_n(t)}|\overline f(|H_{u_n(t)}|)\overline \Psi_{u_{n}(0)} (|H_{u_n(t)}|){\rm e}^{it H_{u_n(t)}^2}u_n(t)\ ,
\label{Psitn}
\\
{\wt H}_{u_n(t)}f(|{\wt H}_{u_n(t)}|)u_n(t)&=|{\wt H}_{u_n(t)}|\overline f(|{\wt H}_{u_n(t)}|)\wt \Psi_{u_{n}(0)} (|{\wt H}_{u_n(t)}|){\rm e}^{it {\wt H}_{u_n(t)}^2}u_n(t)\ .
\label{tildePsitn} 
\end{align}
Our aim is to pass to the limit here as $n\to\infty$. 

In order to motivate the next step, let us make the following remark. Assume that $\Psi_{u_{n}(0)}$ was a continuous function independent of $n$. Then we could pass to the limit in \eqref{Psitn} by Lemma~\ref{lma.i5}. Unfortunately, this assumption is not true and so we need to use a roundabout argument; we will pass to the limit in the right hand sides of \eqref{Psitn} and \eqref{tildePsitn} by considering the weak forms of these identities. But first we need to establish the weak convergence of spectral measures multiplied by the factors $\overline\Psi_{u_{n}(0)}$ and $\wt \Psi_{u_{n}(0)}$ appearing in the right hand sides.

\emph{Step 2: convergence of measures multiplied by $\overline{\Psi}$, $\wt \Psi$.}
At $t=0$ by Lemma~\ref{lma.i5} we have for every continuous function $f$ 
$$
H_{u_n(0)}f(|H_{u_n(0)}|)u_n(0) 
\to 
H_{u(0)}f(|H_{u(0)}|)u(0)\ .
$$
Since both $u_n(0)$ and $u(0)$ are in $\BMOA_\simp$, we can write this as
\begin{equation}
\overline \Psi_{u_n(0)} (|H_{u_n(0)}|)\overline f(|H_{u_n(0)}|)u_n(0)
\to 
\overline \Psi_{u(0)} (|H_{u(0)}|)\overline f(|H_{u(0)}|)u(0),
\label{i1}
\end{equation}
and similarly we obtain 
$$
\wt \Psi_{u_n(0)} (|\wt H_{u_n(0)}|)\overline f(|\wt H_{u_n(0)}|)u_n(0)
\to 
\wt \Psi_{u(0)} (|\wt H_{u(0)}|)\overline f(|\wt H_{u(0)}|)u(0). 
$$
Taking the inner product of \eqref{i1} with $u_n(0)$ and observing that $\norm{u_n(0)-u(0)}_{H^2}\to0$, we find
\begin{equation}
\int_0^\infty \overline \Psi_{u_n(0)}(s) \overline f(s)\, d\rho_n(s)\to \int_0^\infty \overline \Psi_{u(0)}(s) \overline f(s)\, d\rho(s), 
\label{i2}
\end{equation}
where $\rho_n=\rho^{\abs{H_{u_n(0)}}}_{u_n(0)}$ and $\rho=\rho^{\abs{H_{u(0)}}}_{u(0)}$. 
Similarly, we obtain 
$$	
\int_0^\infty \wt  \Psi_{u_n(0)}(s) \overline f(s)\, d\wt  \rho^n(s)\to \int_0^\infty \wt \Psi_{u(0)}(s) \overline f(s)\, d\wt  \rho(s)\ ,
$$
where $\wt\rho_n=\rho^{\abs{\wt H_{u_n(0)}}}_{u_n(0)}$ and $\wt\rho=\rho^{\abs{\wt H_{u(0)}}}_{u(0)}$.

\emph{Step 3: passing to the limit in \eqref{Psitn}, \eqref{tildePsitn}.}
We will pass to the limit in \eqref{Psitn}; the second identity  \eqref{tildePsitn} can be treated similarly. 
Fix $t>0$ and denote 
$$
v(t):=H_{u(t)}f(|H_{u(t)}|)u(t), \quad w(t):=|H_{u(t)}|\overline f(|H_{u(t)}|)\overline\Psi_{u(0)} 
(|H_{u(t)}|){\rm e}^{it H_{u(t)}^2}u(t);
$$
our aim is to prove that $v(t)=w(t)$. By Proposition~\ref{prp.i6} and Lemma~\ref{lma.i5}, we have 
$$
v_n(t):=H_{u_n(t)}f(|H_{u_n(t)}|)u_n(t)\to H_{u(t)}f(|H_{u(t)}|)u(t)=v(t)
$$
in $H^2$, and therefore
$$
\jap{v_n(t),u_n(t)}\to \jap{v(t),u(t)}.
$$
On the other hand, by \eqref{Psitn} and \eqref{i3}, 
\begin{align*}
\jap{v_n(t),u_n(t)}
&=
\jap{\abs{H_{u_n(t)}}\overline{f}(\abs{H_{u_n(t)}})\overline{\Psi}_{u_n(0)}(\abs{H_{u_n(t)}})e^{itH_{u_n(t)}^2}u_n(t),u_n(t)}
\\
&=
\jap{\abs{H_{u_n(0)}}\overline{f}(\abs{H_{u_n(0)}})\overline{\Psi}_{u_n(0)}(\abs{H_{u_n(0)}})e^{itH_{u_n(0)}^2}u_n(0),u_n(0)}
\\
&=
\int_0^\infty s\overline{\Psi}_{u_n(0)}(s)\overline{f}(s)e^{its^2}d\rho_n(s).
\end{align*}
Using \eqref{i2}, followed by \eqref{eq.spmeas} (which was established at the first step of the proof), we find
$$
\jap{v_n(t),u_n(t)}\to \int_0^\infty s\overline{\Psi}_{u(0)}(s)\overline{f}(s)e^{its^2}d\rho(s)
=
\jap{w(t),u(t)}.
$$
Putting this together, we obtain 
$$
\jap{v(t),u(t)}=\jap{w(t),u(t)}. 
$$
Changing $f$ into $fg$, the above identity implies that 
the orthogonal projection of $v(t)$ onto $\jap {u(t)}_{H_{u(t)}^2}$ equals $w(t)$. 
Since $v(t)$ and $w(t)$ have the same norm, we conclude that these two vectors are equal.
\end{proof}

\subsection{The simplicity of spectrum; concluding the proof of Theorem~\ref{actionszego}}
\label{sec.i4}

It remains to prove that $u(t)\in \BMOA_\simp$ for every $t\in \bbR$. This is a consequence of the 
following lemma.
\begin{lemma}
Let $u\in \BMOA(\bbT)$ be such that 
\begin{equation}
H_u(\jap{u}_{H_u^2})\subset \jap{u}_{H_u^2}\ ,\quad
{\wt H}_u(\jap{u}_{{\wt H}_u^2})\subset \jap{u}_{{\wt H}_u^2}\ .
\label{i4}
\end{equation}
Then $u\in \BMOA_\simp$.
\end{lemma}
\begin{proof}
	Recall that, since ${\wt H}_u^2=H_u^2-\jap{\cdot,u}u$, we have 
	$\jap{u}_{{\wt H}_u^2}=\jap{u}_{H_u^2}=\jap{u}$. Denote
	$$
	Z:=\overline{\Ran}H_u\cap \jap{u}^\perp;
	$$
our aim is to prove that $Z=\{0\}$. 
By definition, we have $H_u(Z)\subset Z$ and ${\wt H}_u(Z)\subset Z$. Moreover every 
	$h\in Z$ can be written as
	$$h=\lim_{n\to \infty}H_u h_n\ ,\ h_n:=H_u(H_u^2+\tfrac 1n)^{-1}h\in Z\ .$$
	Consequently, $S^*h=\lim_{n\to \infty}{\wt H}_uh_n\in Z$ and 
$$
\jap{h,z^0}=\lim_{n\to\infty}\jap{H_uh_n,z^0}=\lim_{n\to\infty}\jap{H_uz^0,h_n}=
\lim_{n\to \infty}\jap{u,h_n}=0\ .
$$
	We conclude that $S^*(Z)\subset Z$ and $Z\perp z^0$, hence $Z\perp z^n$ for every $n$, and 
	finally $Z=\{ 0\}$. 
\end{proof}

\begin{proof}[Proof of Theorem~\ref{actionszego}]
By Lemma~\ref{lma.i7}, we have the inclusions \eqref{i4} for $u=u(t)$. It follows that $u(t)\in\BMOA_{\simp}$. The first relation in \eqref{eq.spmeas} shows that the measure $\rho=\rho_{u(t)}^{\abs{H_{u(t)}}}$ is independent of $t$. Relations \eqref{Psit} and \eqref{tildePsit} show that the dynamics of the unimodular functions $\Psi$ and $\wt\Psi$ is as claimed in the statement of the theorem. 
\end{proof}

\subsection{Proof of Theorem~\ref{notAP}} 
\label{sec.i5}
By Theorem~\ref{thm.a5a} (see \eqref{e: u_k 01}), we have
\[
\wh u_k(t) =\jap{(\Sigma(t)^*)^k \1, q(t) }\ci{L^2(\rho)}\ ,
\]
where $\Sigma(t)$ is given by \eqref{w7} with functions $\Psi$, $\wt\Psi$ replaced by 
$$
\Psi_{u(t)}(s)=e^{-its^2}\Psi_0(s)
\quad \text{  and }\quad
\wt\Psi_{u(t)}(s)= e^{its^2}\wt\Psi_0(s)
$$
respectively, 
and the function $q(t) = q(t, \fdot)$ is given by 
\begin{align*}
q(t, s) = \overline{\Psi_{u(t)}(s)}/s = e^{its^2}\overline{\Psi_{0}(s)}/s\ .
\end{align*}
In particular, we get for $k=0$ that 
\[
\wh u_0(t) = \jap{\1, q(t) }\ci{L^2(\rho)} = \int_\bbR e^{-its^2} \frac{\Psi_0(s)}{s} d\rho(s).
\]
That means the function $\wh u_0(t)$ is the Fourier transform of the image (pushforward) of the 
complex measure (of bounded variation)
$$\frac{\Psi_0 (s)}{s}\, d\rho (s)$$ under the map $s\mapsto s^2$. 
Therefore, Theorem~\ref{notAP} follows from the following lemma. 
\begin{lemma}
\label{l: a-p fourier}
	Let $\mu $ be a complex Borel measure  on $\bbR$ of bounded variation such that the Fourier transform
	$$\wh \mu (t)=\int_{\bbR}{\rm e}^{-it\lambda}\, d\mu (\lambda)$$
	is an almost periodic function. Then $\mu $ is pure point.
\end{lemma}
\begin{proof}
	We decompose $\mu $ as the sum of a pure point measure and a diffuse measure 
	$$\mu =\sum_{j=1}^\infty a_j\delta (\lambda -\lambda_j)+\mu_d\ ,$$
	where $\sum_{j=1}^\infty |a_j|<\infty $ and $\mu_d(\{ \lambda\})=0$ for every $\lambda \in 
	\bbR$.
	Then 
	$$\wh \mu (t)=\sum_{j=1}^\infty a_j{\rm e}^{-i\lambda_jt}+\wh\mu_d(t)\ ,$$
	and the almost periodicity of $\wh \mu$ implies the almost periodicity of $\wh \mu_d$.
	For every $T>0$, the Fubini theorem yields
	$$\frac{1}{2T}\int_{-T}^T |\wh \mu_d(t)|^2\, dt =\int_{\bbR}\int_{\bbR} \frac{\sin 
	T(\lambda-\lambda')}{T(\lambda-\lambda')}d\mu_d (\lambda)d\overline \mu_d (\lambda')\ .$$
	As $T\to +\infty $, the integrand in the right hand side tends to $0$ for every $\lambda \ne 
	\lambda '$. Since $\mu_d$ does not see points, $\mu_d\otimes \overline \mu_d$ does not see the 
	diagonal. Therefore the dominated convergence theorem implies that the right hand side tends to 
	$0$. The almost periodic function $\wh \mu_d$ satisfies
	$$\frac{1}{2T}\int_{-T}^T |\wh \mu_d(t)|^2\, dt \to 0$$
	as $T\to +\infty$, hence it is identically $0$. From the injectivity of the Fourier 
	transformation, this implies $\mu_d=0$, hence $\mu $ is pure point.
\end{proof}

\appendix

\section{Proof of reductions to spectral properties of $\wt\Psi(\wt\calM)\Psi(\calM)$}\label{sec.f}

Here we give the proofs of two technical statements: Lemma~\ref{lma.a9} and Theorem~\ref{thm.a7}.

\subsection{Sufficient conditions for $\Sigma_0-I\in\bS_1$ in terms of $\rho$}

\begin{proof}[Proof of Lemma~\ref{lma.a9}] 
We start with the formula
$$
\calM=\frac{2}{\pi}\int_0^\infty \calM^2(\calM^2+t^2 I)^{-1}\, dt = 
\frac{2}{\pi}\int_0^\infty  \left( I - t^2  ( \calM^2+t^2 I)^{-1} \right) dt ;
$$
if $\calM$ is a nonnegative real number, this is a trivial identity, and if $\calM$ is a positive semi-definite self-adjoint operator, it suffices to combine the scalar identity with the spectral representation of $\calM$. 
Of course, the same identity holds for $\wt\calM$.

The operator $\wt\calM^2= \calM^2 - \left\langle \fdot, \1 \right\rangle\1$ is a rank one 
perturbation of $\calM$, so by
the standard  resolvent identity
$$
(\wt\calM^2+t^2I)^{-1}=
(\calM^2+t^2I)^{-1}+
\frac1{\Delta(-t^2)}
\jap{\fdot ,(\calM^2+t^2I)^{-1}\1}(\calM^2+t^2I)^{-1}\1\ ,
$$
where $\Delta$ is the perturbation determinant, 
\[
\Delta(-t^2)=1-\jap{(\calM^2+t^2 I)^{-1}\1,\1}=1-\int_0^\infty \frac{d\rho(s)}{s^2+t^2} \, ,
\]
we get that 
\begin{align}
\wt \calM=\calM-\frac{2}{\pi}
\int_0^\infty \frac{t^2}{\Delta(-t^2)}
\jap{\fdot ,(\calM^2+t^2 I)^{-1}\1}(\calM^2+t^2 I)^{-1}\1\, dt\ .
\label{e4}
\end{align}
Recall that $\Sigma_0^*=\wt\calM\calM^{-1}$. Multiplying \eqref{e4} by $\calM^{-1}$ on the right, we find that $I-\Sigma_0^*$ can be represented as an integral of rank one operators:
\begin{align}
I-\Sigma_0^* &=
\frac{2}{\pi}
\int_0^\infty \frac{t^2}{\Delta(-t^2)}
\jap{\fdot ,\calM^{-1}(\calM^2+t^2 I)^{-1}\1}(\calM^2+t^2 I)^{-1}\1\, dt
\notag
\\
&=\frac{2}\pi\int_0^\infty \frac{t^2}{\Delta(-t^2)} 
\langle\fdot,a_t\rangle b_t \ dt, 
\label{e4b}
\\  
a_t(s) &:=s^{-1}(s^2+t^2)^{-1}, 
\quad
b_t(s):=(s^2+t^2)^{-1}\ .
\notag
\end{align}
First assume \eqref{a6}. Then 
$$
\Delta(-t^2)\geq \Delta(0)=1-\int_0^\infty \frac{d\rho(s)}{s^2}>0.
$$
We estimate the norms of $a_t$ and $b_t$ as 
follows:
\begin{align*}
\norm{a_t}^2&=\int_0^\infty s^{-2}(s^2+t^2)^{-2}d\rho(s)
\leq 
t^{-4}\int_0^\infty  s^{-2}d\rho(s)=Ct^{-4}, \quad t>0,
\notag
\\
\norm{b_t}^2&=\int_0^\infty (s^2+t^2)^{-2}d\rho(s)
\leq 
t^{-2+\eps}\int_0^\infty  s^{-2-\eps}d\rho(s)=Ct^{-2+\eps}, \quad 0<t<1,
\notag
\\
\norm{b_t}^2&=\int_0^\infty (s^2+t^2)^{-2}d\rho(s)
\leq 
t^{-4}\int_0^\infty  d\rho(s)=Ct^{-4}, \quad t>1. 
\end{align*}
Then 
$$
\norm{I-\Sigma_0^*}_{\bS_1}
\leq 
C\int_0^\infty t^2\norm{a_t}\norm{b_t}dt
\leq
C\int_0^1 t^2 t^{-2}t^{-1+\eps/2}dt+C\int_1^\infty t^2 t^{-4}dt<\infty.
$$
Next, assume \eqref{a5}. 
Then 
\begin{align}
\notag
\Delta(-t^2)
& =\int_0^\infty\frac{d\rho(s)}{s^2}-\int_0^\infty \frac{d\rho(s)}{s^2+t^2}
\\ & =t^2\int_0^\infty \frac{d\rho(s)}{s^2(s^2+t^2)}
\geq
t^2 \int_0^\infty s^{-4} d\rho(s)=ct^2.
\label{e4a}
\end{align}
We estimate the 
norms of $a_t$ and $b_t$ as follows:
\begin{align*}
\norm{a_t}^2&=\int_0^\infty s^{-2}(s^2+t^2)^{-2}d\rho(s)
\leq 
t^{-2+\eps}\int_0^\infty  s^{-4-\eps}d\rho(s)=Ct^{-2+\eps}, \quad t>0,
\notag
\\
\norm{b_t}^2&=\int_0^\infty (s^2+t^2)^{-2}d\rho(s)
\leq 
\int_0^\infty  s^{-4}d\rho(s)=C, \quad 0<t<1,
\notag
\\
\norm{b_t}^2&=\int_0^\infty (s^2+t^2)^{-2}d\rho(s)
\leq 
t^{-4}\int_0^\infty  d\rho(s)=Ct^{-4}, \quad t>1. 
\end{align*}
Then 
$$
\norm{I-\Sigma_0^*}_{\bS_1}
\leq 
C\int_0^\infty \norm{a_t}\norm{b_t}dt
\leq
C\int_0^1 t^{-1+\eps/2}dt+C\int_1^\infty t^{-1+\eps/2}t^{-2}dt<\infty.
$$
The proof is complete. 
\end{proof}

\subsection{Trace class inclusions for $(I-\Sigma_0^*)\calM^\eps$}
In this subsection we prove preliminary statements that will be used below in the proof of Theorem~\ref{thm.a7}. 

\begin{lemma}\label{lma.g3}
For any $\eps>0$, the operator $(I-\Sigma_0^*)\calM^\eps$ is trace class. 
\end{lemma}
\begin{proof}
We may assume $0<\eps <1$. 
As in \eqref{e4b}, we represent $(I-\Sigma_0^*)\calM^\eps$ as an integral of rank one operators: 
\begin{align*}
(I-\Sigma_0^*)\calM^\eps &=\frac{2}\pi\int_0^\infty \frac{t^2}{\Delta(-t^2)} 
\langle\fdot,a_t\rangle b_t \ dt, \\ a_t(s) & =s^{-1+\eps}(s^2+t^2)^{-1}, 
\qquad
b_t(s)=(s^2+t^2)^{-1}\ .
\end{align*}
First assume that 
\[
\int_0^\infty \frac{d\rho(s)}{s^2}<1. 
\]
Then $\Delta(-t^2)\geq \Delta(0)>0$. We estimate the norms of $a_t$ and $b_t$ as follows:
\begin{align*}
\norm{a_t}^2
&=
\int_0^\infty s^{-2+2\eps}(s^2+t^2)^{-2}d\rho(s)
\leq 
(t^2)^{-2+\eps}
\int_0^\infty
s^{-2+2\eps}(s^2+t^2)^{-\eps}d\rho(s)
\\
&\leq
t^{-4+2\eps}\int_0^\infty s^{-2+2\eps}s^{-2\eps}d\rho(s)=Ct^{-4+2\eps}, \quad t>0,
\\
\norm{b_t}^2
&=
\int_0^\infty (s^2+t^2)^{-2}d\rho(s)\leq t^{-2}\int_0^\infty s^{-2}d\rho(s)=Ct^{-2}, \quad 0<t<1,
\\
\norm{b_t}^2
&=
\int_0^\infty (s^2+t^2)^{-2}d\rho(s)\leq t^{-4}\int_0^\infty d\rho(s)=Ct^{-4}, \quad t>1.
\end{align*}
Then 
$$
\norm{(I-\Sigma_0^*)\calM^\eps}_{\bS_1}
\leq
\int_0^\infty t^2\norm{a_t}\norm{b_t}dt
\leq
C\int_0^1t^2t^{-2+\eps}t^{-1}dt+C\int_1^\infty t^2t^{-2+\eps}t^{-2}dt<\infty.
$$
Now consider the case
$$
\int_0^\infty \frac{d\rho(s)}{s^2}=1.
$$
For $t>1$, as in \eqref{e4a}, we have $\Delta(-t^2)\geq ct^2$ 
and the above estimates for $\norm{a_t}$ and $\norm{b_t}$ will do. 
For $0<t<1$ we need to be more careful. We write
$$
t^{-2}\Delta(-t^2)
=\int_0^t  \frac{d\rho(s)}{s^2(s^2+t^2)}+\int_t^\infty  \frac{d\rho(s)}{s^2(s^2+t^2)}
\geq 
\frac1{2t^2}\int_0^t  \frac{d\rho(s)}{s^2}
+
\frac12 \int_t^\infty  \frac{d\rho(s)}{s^4}.
$$
Next, 
\begin{align*}
\norm{b_t}^2
&=
\int_0^\infty \frac{d\rho(s)}{(s^2+t^2)^{2}}
=
\int_0^t \frac{d\rho(s)}{(s^2+t^2)^{2}}
+
\int_t^\infty \frac{d\rho(s)}{(s^2+t^2)^{2}}
\\
&\leq 
\frac1{t^2}\int_0^t \frac{d\rho(s)}{s^2}
+
\int_t^\infty \frac{d\rho(s)}{s^4}
\end{align*}
and similarly 
\begin{align*}
\norm{a_t}^2
&=
\int_0^\infty \frac{s^{2\eps}d\rho(s)}{s^2(s^2+t^2)^2}
=
\int_0^t \frac{s^{2\eps}d\rho(s)}{s^2(s^2+t^2)^2}
+
\int_t^\infty \frac{s^{2\eps}d\rho(s)}{s^2(s^2+t^2)^2}
\\
&\leq 
\frac1{t^{4-2\eps}}\int_0^t \frac{d\rho(s)}{s^2}+\frac1{t^{2-2\eps}}\int_t^\infty 
\frac{d\rho(s)}{s^4}
\\
&=t^{-2+2\eps}\biggl(\frac1{t^2}\int_0^t \frac{d\rho(s)}{s^2}+\int_t^\infty 
\frac{d\rho(s)}{s^4}\biggr)\ .
\end{align*}
Integrating, we find 
\begin{align*}
\norm{(I-\Sigma_0^*)\calM^\eps}_{\bS_1}
&\leq
\frac{2}\pi\int_0^1\frac{t^2}{\Delta(-t^2)} \norm{a_t}\norm{b_t}dt
+
C\int_1^\infty \norm{a_t}\norm{b_t}dt 
\\
&\leq
2\int_0^1 t^{-1+\eps}dt+C\int_1^\infty t^{-2+\eps}t^{-2}dt<\infty\ .
\end{align*}
The proof is complete. 
\end{proof}

\begin{lemma}\label{lma6}
For any $\eps>0$, the operators $\calM^\eps(I-\Sigma_0^*)$, $\wt \calM^\eps(I-\Sigma_0^*)$, 
$(I-\Sigma_0^*)\wt \calM^\eps$  
are trace class. 
\end{lemma}
\begin{proof}
Taking adjoints in the previous lemma, we find $\calM^\eps(I-\Sigma_0)\in\bS_1$ for any $\eps>0$. 
Since 
$$
\Sigma_0\Sigma_0^*=I-\jap{\fdot,q_0}q_0, \quad q_0(s)=1/s,
$$
we have
$$
\calM^\eps(I-\Sigma_0^*)=\calM^\eps(\Sigma_0-I)\Sigma_0^*+\text{rank one operator}, 
$$
and so $\calM^\eps(I-\Sigma_0^*)$ is trace class. 

Next, from $\wt \calM^2\leq \calM^2$ by Heinz inequality we have $\wt \calM^{2\eps}\leq 
\calM^{2\eps}$ for any $0<\eps<1$, and so by Lemma~\ref{l: DLemma} $\wt \calM^\eps \calM^{-\eps}$  
is a bounded 
operator. Therefore the operators
$$
\wt \calM^\eps(I-\Sigma_0^*)=(\wt \calM^\eps \calM^{-\eps})\bigl(\calM^\eps(I-\Sigma_0^*)\bigr)
$$
and 
$$
\wt \calM^\eps(I-\Sigma_0)=(\wt \calM^\eps \calM^{-\eps})\bigl(\calM^\eps(I-\Sigma_0)\bigr)
$$
are trace class. 
\end{proof}

\subsection{Proof of Theorem~\ref{thm.a7} (reduction to a.c. part of $\wt\Psi(\wt\calM) \Psi(\calM)$)}
\label{sec.j}

We denote  
$$
c=\wt\Psi(0_+)\Psi(0_+), \quad W=\wt\Psi(\wt \calM)\Psi(\calM).
$$
First we prove a lemma.
\begin{lemma}\label{thm.g1}
Under the hypothesis of Theorem~\ref{thm.a7}, we have
\begin{align}
\Sigma^*=c\Sigma_0^*-cI+W+\text{trace class operator}, 
\label{g1}
\end{align}
and the products 
\begin{equation}
(\Sigma_0^*-I)(W-cI), 
\quad 
(W-cI)(\Sigma_0^*-I), 
\quad
(\Sigma_0-I)(W-cI), 
\quad 
(W-cI)(\Sigma_0-I)
\label{g1a}
\end{equation}
are trace class. 
\end{lemma}
\begin{proof}
First let us prove that the operators 
\begin{align*}
(\Psi(\calM)-\Psi(0_+)I)(I-\Sigma_0^*), \quad
(I-\Sigma_0^*)(\Psi(\calM)-\Psi(0_+)I), 
\\
(\wt\Psi(\wt \calM)-\wt\Psi(0_+)I)(I-\Sigma_0^*), \quad
(I-\Sigma_0^*)(\wt\Psi(\wt \calM)-\wt\Psi(0_+)I)
\end{align*}
are trace class.
The first two inclusions follow from Lemmas~\ref{lma.g3} and \ref{lma6} by writing 
$$
\Psi(\calM)-\Psi(0_+)I=\calM^\eps\varphi(\calM)=\varphi(\calM)\calM^\eps, 
$$
where 
$$
\varphi(t)=t^{-\eps}(\Psi(t)-\Psi(0_+)), \quad \varphi\in L^\infty.
$$
The second two inclusions are obtained in the same way from Lemma~\ref{lma6}.

Now  consider the four operator products \eqref{g1a}. For the first one, we have
\begin{align*}
(\Sigma_0^*-I)&(\wt\Psi(\wt \calM)\Psi(\calM)-\wt\Psi(0_+)\Psi(0_+)I)
\\
=&
(\Sigma_0^*-I)(\wt\Psi(\wt \calM)-\wt\Psi(0_+)I)\Psi(\calM)
+
\wt\Psi(0_+)(\Sigma_0^*-I)(\Psi(\calM)-\Psi(0_+)I),
\end{align*}
where the right hand side is trace class by the first part of the proof.
The other three operators are considered in the same way. 

Let us prove \eqref{g1}. We have
\begin{align*}
\Sigma^*
=&
\wt\Psi(\wt \calM)\Sigma_0^*\Psi(\calM)
\\
=&\wt\Psi(\wt \calM)(\Sigma_0^*-I)\Psi(\calM)+W
\\
=&\wt\Psi(0_+)(\Sigma_0^*-I)\Psi(\calM)
+(\wt\Psi(\wt \calM)-\wt\Psi(0_+))(\Sigma_0^*-I)\Psi(\calM)+W
\\
=&\wt\Psi(0_+)(\Sigma_0^*-I)\Psi(0_+)
+\wt\Psi(0_+)(\Sigma_0^*-I)(\Psi(\calM)-\Psi(0_+))
\\
&+(\wt\Psi(\wt \calM)-\wt\Psi(0_+))(\Sigma_0^*-I)\Psi(\calM)+W
\\
=&c(\Sigma_0^*-I)+W+\text{trace class operator},
\end{align*}
where we have used the first part of the proof at the last step. 
\end{proof}

\begin{proof}[Proof of Theorem~\ref{thm.a7}]
Now let us give the proof of Theorem~\ref{thm.a7}.
We shall denote by $A_{\rm ac}$ the a.c. part of a self-adjoint operator $A$ and by $\simeq$ the 
unitary equivalence between operators. 

From Lemma~\ref{thm.g1} it follows that 
$$
\Re (\Sigma^*-cI)=\Re(c\Sigma_0^*-cI)+\Re(W-cI)+\text{trace class operator}
$$
and 
$$
\Re(c\Sigma_0^*-cI)\Re(W-cI)\in \bS_1, 
\quad
\Re(W-cI)\Re(c\Sigma_0^*-cI)\in \bS_1.
$$
Applying Ismagilov's theorem and the Kato-Rosenblum theorem (see Section~\ref{sec.e}), we find
$$
(\Re(\Sigma^*-cI))_{\rm ac}\simeq (\Re(c\Sigma_0^*-cI))_{\rm ac}\oplus(\Re(W-cI))_{\rm ac}. 
$$
Shifting all operators here by $\Re c$, this simplifies to 
\begin{align}
(\Re \Sigma^*)_{\rm ac}\simeq (\Re(c\Sigma_0^*))_{\rm ac}\oplus(\Re W)_{\rm ac}. 
\label{g2}
\end{align}
This is our key formula. The rest of the proof proceeds slightly differently, depending on the 
defect indices of  $\Sigma^*$.

\emph{The case of defect indices $(1,1)$:}
Recall that in this case by Theorem~\ref{thm2} (applied to $c\Sigma_0^*$), 
$\Re(c \Sigma_0^*)$ has a purely singular spectrum. 
By \eqref{g2}, it follows that 
\[
(\Re \Sigma^*)_{\rm ac}\simeq (\Re W)_{\rm ac}. 
\]
Now by Theorem~\ref{thm.a6}(ii), $\Sigma^*$ is asymptotically stable iff the spectrum of $\Re W$ is 
singular.  Applying Proposition~\ref{prp.realpart}, we see that this is true iff the spectrum of $W$ is singular. 
The proof in this case is complete. 

\emph{The case of defect indices $(1,0)$:}
In this case, the proof is similar but we have to look at the multiplicity of the a.c. spectrum. 

Here $\Sigma_0^*\simeq S^*$ and so $\Re(c \Sigma_0^*)\simeq \Re(c S^*)$, where $\Re(c S^*)$ (which 
is a Jacobi matrix) has a purely a.c. spectrum $[-1,1]$ of multiplicity one. From \eqref{g2} we find
\begin{align}
(\Re \Sigma^*)_{\rm ac}\simeq(\Re(c S^*))_{\rm ac}\oplus (\Re W)_{\rm ac}. 
\label{g6}
\end{align}
Looking at the multiplicity function of the a.c. spectrum and 
applying Theorem~\ref{thm.a6}(i), we find that $\Sigma^*$ is asymptotically stable if and only if 
the second term in \eqref{g6} disappears, i.e. if and only if the spectrum of $W$ is singular.
The proof is complete.
\end{proof}

\section{Proof of Theorem~\ref{thm.b3aa}}\label{app.b}
Denote for brevity $R=\Ran H_u$. We first prove that $\Ker H_u=\{0\}$ is equivalent to $z^0\in\overline{R}\setminus R$. 

Assume that $\Ker H_u=\{0\}$. Then $\overline{R}=H^2$ and so, of course, $z^0\in \overline{R}$; we need to prove that $z^0\notin R$. Suppose $z^0\in R$; then $z^0=H_uw$ for some $w\in H^2$. Denote $\psi=zw$; then $H_u\psi=H_uSw=S^*H_uw=0$, and so $\psi\in\Ker H_u$, which contradicts our assumption. 

Assume that $\Ker H_u\not=\{0\}$. Suppose $z^0\in\overline{R}$; we need to check that $z^0\in R$. By Beurling's theorem, $\Ker H_u=\varphi H^2$ for some inner function $\varphi$. Since $z^0\in\overline{R}$, we have $z^0\perp \Ker H_u=\varphi H^2$ and so $z^0\perp \varphi$. Then $\varphi=Sw$ for some inner function $w$. We have $0=H_u\varphi=H_uSw=S^*H_uw$, so $H_uw$ is a constant function. This constant function is non-zero, because otherwise we would have $w\in\Ker H_u=\varphi H^2=zw H^2$, which is impossible. Thus, renormalising $w$ if necessary, we find that $z^0=H_uw$, and so $z^0\in R$. 

Next, we prove that $z^0\in \overline{R}$ is equivalent to the first condition in \eqref{e: triv ker}. 
Indeed, $z^0\in \overline{R}$ is equivalent to $\int_0^\infty d\rho_{z^0}^{\abs{H_u}}(s)=1$. 
Since 
$$
d\rho_u^{\abs{H_u}}(s)=d\rho_{H_uz^0}^{\abs{H_u}}(s)=s^2d\rho_{z^0}^{\abs{H_u}}(s),
$$
this is equivalent to $\int_0^\infty s^{-2}d\rho_{u}^{\abs{H_u}}(s)=1$, which is 
the first condition in \eqref{e: triv ker}.

Finally, suppose $z^0\in \overline{R}$; let us prove that $z^0\notin R$ is equivalent to the second condition in \eqref{e: triv ker}. If $z^0=H_uw$ with $w\in H^2$, then 
$$
d\rho_{z^0}^{\abs{H_u}}(s)=d\rho_{H_u w}^{\abs{H_u}}(s)=s^2d\rho_w^{\abs{H_u}}(s),
$$
and so 
$$
\int_0^\infty s^{-4}d\rho_u^{\abs{H_u}}(s)
=\int_0^\infty s^{-2}d\rho_{z^0}^{\abs{H_u}}(s)
=\int_0^\infty d\rho_w^{\abs{H_u}}(s)<\infty.
$$
Conversely, if $\int_0^\infty s^{-4}d\rho_u^{\abs{H_u}}(s)<\infty$, then 
$u=H_u^2w$ for some $w\in H^2$. It follows that 
$$
H_u(z^0-H_uw)=u-H_u^2w=0,
$$
and so $z^0-H_uw\in\Ker H_u$. Since by assumption $z^0\in\overline{R}$, we have $z^0-H_uw=0$, so $z^0\in R$. The proof of Theorem~\ref{thm.b3aa} is complete. \qed

\section{Proof of Theorem~\ref{thm2}}\label{app.c}
While it is probably possible to give an ``elementary'' proof of Theorem~\ref{thm2}, bypassing the Sz.-Nagy--Foia\c{s} functional  model, we prefer a more ``high brow'' approach, since it highlights a lot of interesting connections.  

\subsection{Functional model for c.n.u. contractions with defect indices $(1,1)$}
\label{s: N-F model}
Let us recall some known facts about the Sz.-Nagy--Foia\c{s} functional model for contractions, focussing on the case of defect indices $(1,1)$. 

Any  c.n.u. contraction $T$ is unitarily equivalent to its \emph{functional model}, which is 
completely determined by the so-called characteristic function $\theta$ of the operator $T$. This 
characteristic function $\theta$ is generally an operator-valued one; but in the case of defect 
indices $(1,1)$ is is a scalar-valued \emph{strictly contractive} (i.e. $\abs{\theta(0)}<1$) analytic function in the unit disc $\bbD$. 

If the characteristic function $\theta$ of $T$ is an inner function, then the model space for $T$ is 
the space $K_\theta$ defined above in Section \ref{sec.e1}, and the operator $T$ is  unitarily 
equivalent  to the compressed shift $S_\theta$. 

If $\theta$ is not inner, then the model is more complicated; in particular, in 
this case the model space consists of vector-valued functions with values in $L^2$. However, we do not need the 
complete description of the model here: we only need the following well-known fact. 

\begin{proposition}
\label{p: CharFunct and asy stability}
Let $T$ be a c.n.u. contraction with defect indices $(1,1)$, and let $\theta$ be its 
characteristic function. 
\begin{enumerate}
\item If $\theta$ is inner, then both $T$ and $T^*$  are asymptotically stable. 
\item If $\theta$ is not inner, then neither $T$ nor $T^*$ are asymptotically stable. 
\end{enumerate}
\end{proposition}

This proposition follows, for example, from \cite[Proposition VI.3.5]{SzNF2010}. In this 
proposition $T\in C_{\cdot 0}$ means that $T^*$ is asymptotically stable, and $T\in C_{0\cdot}$ 
means that $T$ is asymptotically stable. Note that for scalar-valued function the notion of inner and 
$*$-inner functions coincide. 

Of course, part (i) of Proposition~\ref{p: CharFunct and asy stability} follows directly from the fact that both the compressed shift $S_\theta$ and its adjoint $S_\theta^*$ are asymptotically stable; this is an easy exercise.

\subsection{Rank one unitary extensions and characteristic function}
\label{s: rk 1 extension}
For a contraction $T$ with defect indices $(1,1)$ there exists a rank one perturbation $K$ such 
that the operator $V=T+K$ (which we will call a \emph{rank one unitary extension} of $T$) is 
unitary. 

To construct such $V$, it suffices to notice that $T$ acts unitarily from 
$(\calD\ci{T^*})^\perp$ onto $(\calD\ci{T})^\perp$, and therefore it maps the one-dimensional defect space $\calD\ci{T^*}$ onto the defect space $\calD\ci{T}$.  Replacing the action of $T$ on $\calD\ci{T^*}$ by a unitary operator from $\calD\ci{T^*}$ to $\calD\ci{T}$ yields the desired rank one unitary extension $V$. Clearly, such extension is not unique and any two such extensions differ by a rank one operator.

For a unitary operator $V$ in a Hilbert space, a subspace $E$ is called \emph{$*$-cyclic} if the linear span of the set $\{ V^n E : n\in \bbZ\}$ is dense in our Hilbert space. If $V=T+K$ is a rank one unitary extension of a contraction $T$ with defect  indices $(1,1)$, 
then we can say that $T=V-K$ is a rank one perturbation of the unitary operator $V$. 

Let 
$b\in \Ran K$ be a unit norm vector. It is a simple exercise (see e.g. \cite[Section~1]{LT} or \cite[Section~1]{Liaw-Treil_APDE_2019}) to show that $T$ can be represented as  
\begin{align}
\label{e: rank 1 pert}
T=V + (\gamma-1) \jap{\fdot, V^* b}b
\end{align}
with $\gamma\in\bbD$. 

Is is also not hard to see that if $\spn\{b\}=\Ran K$ is  $*$-cyclic  for $V$ and 
$|\gamma|<1$, then $T$ is c.n.u. For the formal proof see \cite[Lemma 
1.4]{Liaw-Treil_APDE_2019}, where a more general case of finite rank perturbations was treated. We 
mention that in our case the matrix $\Gamma$ from \cite{Liaw-Treil_APDE_2019} reduces to a scalar 
 $|\gamma|<1$.

On the other hand, if $\Ran K$ is not a $*$-cyclic subspace for $V$, then trivially $T$ is 
not c.n.u. Indeed, in this case the subspace $(\spn\{V^nb: n\in\bbZ\})^\perp$ is a reducing subspace 
for both $V$ and $T$, and $T$ coincides with $V$ there. 

Combining these facts, we get the following statement. 

\begin{proposition}
\label{p: cnu cyclicity}
Let $T$ be a c.n.u. contraction with defect  indices $(1,1)$, and let $V=T+K$ be a rank one 
unitary extension of $T$. Then $\Ran K$ is a $*$-cyclic subspace for $V$. 
\end{proposition}

Finally, the following well-known statement relates the spectral properties of a rank one unitary 
extension and the characteristic function of a c.n.u. contraction. 

\begin{proposition}
\label{p: SingSp <-> inner}
Let $T$ be a c.n.u. contraction with defect  indices $(1,1)$, and let $V=T+K$ be a rank one 
unitary extension of $V$. 
Then the characteristic function $\theta$ of $T$ is inner if and only if $V$ has purely singular 
spectrum. 
\end{proposition}
\begin{remark*}
The choice of the extension $V$ is not important. Indeed, if $V_1$ and $V_2$ are two such extensions, then $V_1-V_2$ is a rank one operator, and so by the Kato--Rosenblum theorem, $V_1$ has a purely singular spectrum if and only if $V_2$ has the same property.  
\end{remark*}

\begin{proof}[Proof of Proposition \ref{p: SingSp <-> inner}]
D.~Clark \cite{Clark} has described all rank one unitary extensions of the compressed shift $S_\theta$; in particular, he showed that all these extensions have purely singular spectrum. This proves that if the characteristic function $\theta$ is inner, then all rank one unitary extensions have 
purely singular spectrum.

To prove the converse, we compute the characteristic function of the operator $T=T_\gamma$ given by \eqref{e: rank 1 pert}. For the case $\gamma=0$ the characteristic function $\theta=\theta_0$ is given by the relation 
\begin{align}
\label{e: theta0}
\frac{1+\theta_0(z)}{1-\theta_0(z)} =\int_\bbT \frac{1+z\overline{\xi}}{1-z\overline{\xi}} 
d\rho^V_f(\xi) , 
\end{align}
where $\rho^V_f$ is the spectral measure of $V$, corresponding to the unit vector $b$. For $\gamma\ne 
0$ the corresponding characteristic functions $\theta=\theta_\gamma$ can be computed as a linear 
fractional transformation of $\theta_0$, 
\begin{align}
\label{e: LFT theta0}
\theta_\gamma = \frac{\theta_0-\gamma}{1-\overline{\gamma}\theta_0}, 
\end{align}
see \cite[Section~2.4]{LT} for the details. One can see immediately from \eqref{e: theta0} that $\theta_0$ is inner if and only if the measure $\rho^V_f$ is purely singular. Identity \eqref{e: LFT theta0} 
implies that the same holds for all $\theta_\gamma$. 
\end{proof}

\subsection{Proof of Theorem~\ref{thm2}} It is convenient to 
introduce one more equivalent condition:

\begin{enumerate}
\setcounter{enumi}{4}
\item \emph{ Any rank one unitary perturbation $V$ of $T$ has purely singular spectrum}.
\end{enumerate}
The statement \cond4 means that the characteristic function of $T$ is inner, 
see Section~\ref{s: N-F model}. 

Equivalence of \cond1$\iff$\cond2$\iff$\cond4 follows from 
Proposition \ref{p: CharFunct and asy stability}. 

Equivalence \cond5$\iff$\cond4 follows from Proposition \ref{p: SingSp <-> inner}. 

To show that \cond5$\iff$\cond3, let us notice that $\Re V$ and $V$ have purely singular spectrum 
simultaneously, see Proposition \ref{prp.realpart}. But $\Re T$ is a finite rank perturbation of 
$\Re V$, so the Kato-Rosenblum 
Theorem implies the desired equivalence.  

The proof of Theorem~\ref{thm2} is complete. \qed

\def\cprime{$'$}

\end{document}